\documentclass[12pt]{amsart}
\usepackage{tikz}
\usetikzlibrary{calc}
\usetikzlibrary{decorations.markings}
\usepackage{amssymb,amsthm,amsmath}
\usepackage[margin=1.1in]{geometry}
\usepackage[colorlinks]{hyperref}
\def\H{{\mathcal{H}}}
\def\Sep{{\rm Sep}}
\def\I{{\mathcal{I}}}
\def\top{{\rm top}}
\def\tS{{\tilde S}}
\def\C{{\mathbb C}}
\def\wt{{\rm wt}}
\def\Gr{{\rm Gr}}
\def\X{{\mathcal{X}}}
\def\R{{\mathbb R}}
\def\E{{\mathcal{E}}}
\def\L{{\mathcal{L}}}
\def\M{{\mathcal{M}}}

\def\I{{\mathcal{I}}}
\def\Z{{\mathbb Z}}
\def\S{{\mathcal{S}}}
\def\spn{{\rm span}}
\def\NC{{\mathcal{NC}}}
\def\P{{\mathbb P}}
\def\uncrossed{{\rm uncrossed}}
\def\tGamma{\tilde \Gamma}
\def\tF{\tilde F}
\def\Bound{{\rm Bound}}
\def\Elec{{\rm Elec}}
\def\tsigma{\tilde \sigma}
\def\oPi{\mathring{\Pi}}
\def\oX{{\mathring{\X}}}
\def\J{{\mathcal {J}}}
\def\A{{\mathcal{A}}}
\newcommand{\defn}[1]{{\bf #1}}

\newtheorem{theorem}{Theorem}

\newtheorem{proposition}[theorem]{Proposition}
\newtheorem{prop}[theorem]{Proposition}
\newtheorem{corollary}[theorem]{Corollary}
\newtheorem{cor}[theorem]{Corollary}

\newtheorem{lemma}[theorem]{Lemma}

\newtheorem{definition}[theorem]{Definition}
\theoremstyle{remark}
\newtheorem{remark}[theorem]{Remark}
\numberwithin{theorem}{section}
\numberwithin{equation}{section}
\numberwithin{figure}{section}
\newtheorem{example}[theorem]{Example}

\begin{document}
\title{Electroid varieties and a compactification of the space of electrical networks}
\author{Thomas Lam}
\address{Department of Mathematics, University of Michigan,
2074 East Hall, 530 Church Street, Ann Arbor, MI 48109-1043, USA}
\email{tfylam@umich.edu}\thanks{T.L. was supported by NSF grant DMS-1160726.}
\begin{abstract}
We construct a compactification of the space of circular planar electrical networks studied by Curtis-Ingerman-Morrow \cite{CIM} and Colin de Verdi\`{e}re-Gitler-Vertigan \cite{CGV}, using cactus networks.  We embed this compactification as a linear slice of the totally nonnegative Grassmannian, and relate Kenyon and Wilson's grove measurements to Postnikov's boundary measurements.  Intersections of the slice with the positroid stratification leads to a class of electroid varieties, indexed by matchings.  The uncrossing partial order on matchings arising from electrical networks is shown to be dual to a subposet of affine Bruhat order.  The analogues of matroids in this setting are certain distinguished collections of non-crossing partitions.
\end{abstract}
\maketitle
\tableofcontents
\section{Introduction}
\subsection{Circular planar electrical networks and groves}
A \defn{circular planar electrical network} is a weighted undirected graph $\Gamma$ embedded into a disk, with distinguished boundary vertices on the boundary of the disk.  Each edge is thought of as a resistor, and the weight of the edge is the conductance of that resistor.  The electrical properties of $\Gamma$ are encoded in a response matrix
$$
\Lambda(\Gamma): \R^{\#\text{boundary vertices}} \to  \R^{\#\text{boundary vertices}}
$$
which sends a vector of voltages at the boundary vertices, to the vector of currents induced through the same vertices.  Two electrical networks are electrically-equivalent if they have the same response matrix.  Circular planar electrical networks were studied thoroughly by Curtis-Ingerman-Morrow \cite{CIM} and Colin de Verdi\`{e}re-Gitler-Vertigan \cite{CGV}, who classified such $\Gamma$ up to electrical-equivalence, and showed that the space of their response matrices decomposes into a disjoint union of cells $\R_{>0}^d$.

In \cite{KW}, Kenyon and Wilson studied \defn{grove measurements} $L_\sigma(\Gamma)$, which count spanning subforests of $\Gamma$ inducing a particular boundary partition $\sigma$.  Our point of view is that the grove measurements can be used as projective coordinates on the space of electrical networks, giving a map
$$
\L: \Gamma \mapsto (L_\sigma(\Gamma)) \in \P^{\NC_n}
$$
where $\P^{\NC_n}$ is the projective space with coordinate labeled by non-crossing partitions $\NC_n$ on $n$ objects, where $n$ is the number of boundary vertices.  A natural question is: what is the closure $E_n$ (in the Hausdorff topology) of the image in $\P^{\NC_n}$, and does it possess a natural stratification?

\subsection{Cactus networks and the compactification $E_n$}
Roughly speaking, a \defn{cactus network} is an electrical network where some boundary vertices have been identified according to a non-crossing partition (so the boundary now looks like a cactus).  We show (Theorem \ref{thm:cactus}) that each point in $E_n$ is uniquely represented by an electrical-equivalence class of cactus networks.  The space of cactus networks has a stratification 
$$
E_n =\bigsqcup_{\tau \in P_n} E_\tau
$$
into cells $E_\tau$ labeled by the set $P_n$ of medial pairings (that is, matchings) on $2n$ objects.  Each point $\L \in E_\tau$ has the same vanishing and non-vanishing pattern of grove coordinates $L_\sigma$.  In other words, the \defn{electroid}
$$
\E(\L):= \{ \sigma \in \NC_n \mid L_\sigma \neq 0\} \subset \NC_n
$$
is constant in each $E_\tau$.  We remark that the (noncompact) space of circular planar electrical networks studied in \cite{CIM} and \cite{CGV} has a stratification indexed by only {\it some} matchings, which is combinatorially unnatural.

\subsection{An embedding into the totally nonnegative Grassmannian}
There is a well known analogy between the properties of the space of electrical networks, and spaces of totally positive matrices.  We amplify this analogy by constructing an embedding (Theorem \ref{thm:realizability})
$$
\iota: E_n \hookrightarrow \Gr(n-1,2n)_{\geq 0}
$$
of the space of electrical networks into the totally nonnegative Grassmannian of $(n-1)$-planes in $2n$-space, studied by Postnikov \cite{Pos}.  On the level of graphs, the map is induced by a version of the \defn{generalized Temperley's trick} of Kenyon, Propp, and Wilson \cite{KPW}, adapted to our situation.  For each cactus network $\Gamma$ on $n$ boundary vertices, we obtain a planar bipartite graph $N(\Gamma)$ embedded into a disk with $2n$ boundary vertices.  

The space of electrical networks is essentially generated by the combinatorial operations of adding boundary spikes and boundary edges, studied by Curtis-Ingerman-Morrow \cite{CIM}, and by Lam and Pylyavskyy \cite{LP}.  Under the above embedding, these operations correspond to adding boundary bridges to $N(\Gamma)$.  This observation gives rise to a new representation (Proposition \ref{prop:elecbraid}) of the electrical braid relation studied in \cite{LP}.  

We give (Theorem \ref{thm:concordant}) linear relations between Postnikov's \defn{boundary measurements} $\Delta_I(N(\Gamma))$, that count almost perfect matchings in $N(\Gamma)$, and the grove measurements $L_\sigma(\Gamma)$:
$$
\Delta_I(N(\Gamma)) = \sum_{\sigma \in \E(I)} L_\sigma(\Gamma)
$$
where $\E(I) \subset \NC_n$ denotes the set of non-crossing partitions that are \defn{concordant} (defined in Section \ref{sec:concordant}) with a subset $I \in \binom{[2n]}{n-1}$.  Thus we obtain
$$
\iota(E_n) \subset \Gr(n-1,2n) \cap \H =: \X
$$
where $\H \subset \P^{\binom{[2n]}{n-1}}$ is a linear subspace of Pl\"ucker space.  One of our main theorems (Theorem \ref{thm:realizability}) is that $\iota(E_n) = \X \cap \Gr(n-1,2n)_{\geq 0}$, so that each totally nonnegative point in $\X$ is representable by a cactus network.  The Pl\"ucker relations for the Grassmannian also give rise to new quadratic relations (Proposition \ref{prop:quadratic}) for grove measurements.  These relations suffice to cut out the Zariski-closure of $E_n$ from $\P^{\NC_n}$, set-theoretically.

\subsection{Electroid varieties and positroid varieties}
The Grassmannian has a stratification $\Gr(n-1,2n) = \bigsqcup_f \oPi_f$ by \defn{positroid varieties} studied by Postnikov \cite{Pos}, and by Knutson, Lam, and Speyer \cite{KLS}.  Positroid varieties are labeled by a set $\Bound(n-1,2n)$ of affine permutations called \defn{bounded affine permutations}.  Each matching $\tau \in P_n$ naturally gives rise to a bounded affine permutation $f_\tau \in \Bound(n-1,2n)$.  We define the \defn{electroid varieties} by 
$$
\oX_{f_\tau} := \X \cap \oPi_{f_\tau}
$$
and show (Theorem \ref{thm:stratification}) that the intersections of $\X$ with other positroid strata are empty.  The totally nonnegative part $(\oX_{f_\tau})_{\geq 0}:= \oX_{f_\tau} \cap \Gr(n-1,2n)_{\geq 0}$ of the electroid variety is identified with $E_\tau$.  The partial order on matchings $P_n$ induced by the closure order on the $E_\tau$ has been studied by Alman, Lian and Tran \cite{ALT}, by Kenyon \cite{Ken} and by Huang, Wen and Xie \cite{HWX}.  We show (Theorem \ref{thm:poset}) that this partial order is dual to an induced subposet of affine Bruhat order.  In \cite{Lammatch}, we use this result to show that the poset $\hat P_n$, obtained from $P_n$ by adjoining a minimum element, is Eulerian.

Bounded affine permutations $f \in \Bound(n-1,2n)$ are in bijection with sequences $(I_1,I_2,\ldots,I_{2n})$ of subsets called \defn{Grassmann necklaces} \cite{Pos}.  Similarly, matchings $\tau \in P_n$ on $[2n]$ are in bijection with sequences $\Sigma(\tau)=(\sigma^{(1)},\ldots,\sigma^{(2n)})$ of non-crossing partitions, called \defn{partition necklaces} (Section \ref{sec:partnecklace}).  Non-crossing partitions are in bijection with Dyck paths, and inherit a dominance partial order that corresponds to one Dyck path always staying underneath another.  In analogy with a theorem of Oh \cite{Oh} concerning positroids (those matroids that arise from the totally nonnegative Grassmannian), we show (Theorem \ref{thm:Ohelectroid}) that the electroid $\E(\tau)$ is an intersection of a number of cyclically rotated order ideals on $\NC_n$, with respect to this partial order.  In summary, we have the following list of analogies:
\begin{center}
\begin{tabular}{|c |c|}
\hline
Planar bipartite graph $N$ & Cactus network $\Gamma$ \\
\hline
\hline
Almost perfect matchings in $N$ & Groves in $\Gamma$ \\
\hline
Pl\"ucker space $\P^{\binom{[n]}{k}}$ & Non-crossing partition space $\P^{\NC_n}$ \\
\hline
Grassmannian $\Gr(k,n) \subset \P^{\binom{[n]}{k}}$ & Zariski closure of $E_n \subset \P^{\NC_n}$\\
\hline
Alternating strand diagram & Medial graph \\
\hline
Bounded affine permutations $f$ & Matchings $\tau$ \\
\hline
Bruhat order on bounded affine permutations & ``Uncrossing'' partial order on matchings \\
\hline
Subsets $I \in \binom{[n]}{k}$ & Non-crossing partitions $\sigma \in \NC_n$ \\
\hline
Positroids $\M \subset \binom{[n]}{k}$ & Electroids  $\E \subset \NC_n$ \\
\hline
Grassmann necklaces $\I=(I_1,\ldots,I_n)$ & Partition necklaces $\Sigma = (\sigma^{(1)},\ldots,\sigma^{(2n)})$ \\
\hline
$GL_{2n}$ & electrical Lie group $EL_{2n}$\\
\hline
\end{tabular}
\end{center}

Our work is related to Henriques and Speyer's work \cite{HS} on the cube recurrence and isotropic Grassmannian via Fock and Goncharov's cluster $\A$-$\X$-variety duality, studied in a situation related to ours by Goncharov and Kenyon \cite{GK}.  We shall return to this point in future work.

\bigskip

{\bf Acknowledgements.}  This project began after several long conversations with Alex Postnikov, who impressed upon me the naturality of indexing cells in the space of electrical networks with all medial pairings, not just some of them.  The idea that the boundary spike and boundary edge generators that I studied with Pylyavskyy could be interpreted in terms of the TNN Grassmannian was obtained in conversations with him.

My thinking about electrical networks has to a large extent been shaped by the ideas of Pasha Pylyavskyy, and I am grateful to him for all the ideas he has shared.  I also benefitted greatly from a number of conversations with Rick Kenyon, and from the hospitality of ICERM which led to these conversations.

I also thank David Speyer for a number of comments on an earlier version of this work, and Yu-tin Huang for explaining the relation to ABJM scattering amplitudes.

\section{Circular planar electrical networks}\label{sec:elec}
We will use the notation $[n]:=\{1,2,\ldots,n\}$, and $\binom{[n]}{k}$ will mean the set of $k$-element subsets of $[n]$.  

We refer the reader to \cite{CIM,CGV,KW} for the material of this section.  

\subsection{Electrical networks and response matrices} \label{sec:response}
For our purposes, a \defn{circular planar electrical network} (or just electrical network) is a finite weighted undirected graph $\Gamma$ embedded into a disk, where the vertex set is divided into the {\it boundary} vertices and the {\it interior} vertices.  The boundary vertices are denoted $\bar 1, \bar 2, \ldots, \bar n$ and are arranged in clockwise order on the boundary of a disk.

\begin{example}
A circular planar electrical network on $\{\bar 1,\bar 2,\bar 3,\bar 4,\bar 5\}$.  Unlabeled edges are assumed to have weight $1$.
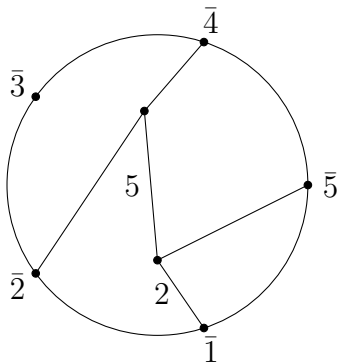
\begin{center}
\begin{tikzpicture}[scale = 0.5]
\node at (0:4.6) {$\bar 5$};
\node at (72:4.6) {$\bar 4$};
\node at (144:4.6) {$\bar 3$};
\node at (216:4.6) {$\bar 2$};
\node at (288:4.6) {$\bar 1$};
\draw (0,0) circle (4cm);
\draw (288:4) -- node [left] {$2$} (270:2);
\draw (270:2) -- node [left] {$5$} (100:2);
\draw (270:2) -- (0:4);
\draw (100:2) -- (216:4);
\draw (100:2) -- (72:4);
\filldraw[black] (0:4) circle (0.1cm);
\filldraw[black] (72:4) circle (0.1cm);
\filldraw[black] (144:4) circle (0.1cm);
\filldraw[black] (216:4) circle (0.1cm);
\filldraw[black] (288:4) circle (0.1cm);
\filldraw[black] (100:2) circle (0.1cm);
\filldraw[black] (270:2) circle (0.1cm);
\end{tikzpicture}
\end{center}
\end{example}

The weight $w(e)$ of an edge is to be thought of as the conductance of the corresponding resistor, and is generally taken to be a positive real number.  Note that if $w(e) = 0$, we may just remove the edge, and if $w(e) = \infty$ we may glue together the endpoints of the edge.

We will think of electrical current as flowing along edges, and each vertex to have a voltage.  The axioms of electricity can be summarized by two laws.  {\it Kirchhoff's law} says that for any interior vertex the total current that flows into the vertex is equal to the total current that flows out.  {\it Ohm's law} says that if $e = (u,v)$ is an edge then we have
$$
w(e)(V(u) - V(v)) = I(e)
$$
where $w(e)$ is the conductance of the edge, $V(u), V(v)$ are the voltages of the vertices, and $I(e)$ is the current flowing from $u$ to $v$.


The \defn{response matrix} $\Lambda(\Gamma)$ is the $n \times n$ matrix defined as follows.  If $v \in \R^n$ is a vector of voltages assigned to the boundary nodes, then $\Lambda(\Gamma)$ is the vector of currents flowing into the network at each of the boundary nodes.  This vector of currents can be determined using Ohm's law and Kirchhoff's law by solving a system of linear equations.  We declare $\Gamma$ and $\Gamma'$ to be \defn{electrically-equivalent} if $\Lambda(\Gamma) = \Lambda(\Gamma')$.

\subsection{Groves}
A \defn{grove} $F$ on $\Gamma$ is a spanning subforest (that is, an acyclic subgraph that uses all the vertices) such that each component $F_i \subset F$ is connected to the boundary.  The \defn{boundary partition} $\sigma(F)$ is the set partition of $\{\bar 1,\bar 2,\ldots,\bar n\}$ which specifies which boundary vertices lie in the same component of $F$.   Note that since $\Gamma$ is planar, $\sigma(F)$ must be a \defn{non-crossing partition}, also called a \defn{planar set partition}.  We will often write set partitions in the form $\sigma = (\bar a,\bar b,\bar c|\bar d,\bar e|\bar f,\bar g|\bar h)$ or more simply just $(\bar a\bar b\bar c|\bar d\bar e|\bar f\bar g|\bar h)$.  Let $\NC_n$ denote the set of non-crossing partitions on $[\bar n]$.

Each non-crossing partition $\sigma$ on $[\bar n]$ has a dual non-crossing partition on $[\tilde n]$ where by convention $\tilde i$ lies between $\bar i$ and $\overline{i+1}$.   For example $(\bar 1, \bar 4, \bar 6| \bar 2, \bar 3| \bar 5)$ is dual to $(\tilde 1, \tilde 3| \tilde 2| \tilde 4, \tilde 5| \tilde 6)$.

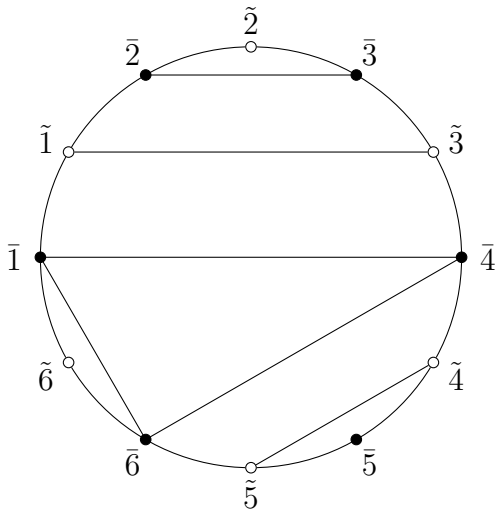
\begin{center}
\begin{tikzpicture}[scale = 0.7]
\draw (0,0) circle (4cm);
\node at (180:4.5) {$\bar 1$};
\node at (120:4.5) {$\bar 2$};
\node at (60:4.5) {$\bar 3$};
\node at (0:4.5) {$\bar 4$};
\node at (300:4.5) {$\bar 5$};
\node at (240:4.5) {$\bar 6$};
\node at (150:4.5) {$\tilde 1$};
\node at (90:4.5) {$\tilde 2$};
\node at (30:4.5) {$\tilde 3$};
\node at (-30:4.5) {$\tilde 4$};
\node at (-90:4.5) {$\tilde 5$};
\node at (-150:4.5) {$\tilde 6$};
\draw (180:4) -- (0:4) -- (240:4) -- (180:4);
\draw (120:4) -- (60:4);
\draw (150:4) -- (30:4);
\draw (-30:4) -- (-90:4);
\filldraw[black] (180:4) circle (0.1cm);
\filldraw[black] (120:4) circle (0.1cm);
\filldraw[black] (60:4) circle (0.1cm);
\filldraw[black] (0:4) circle (0.1cm);
\filldraw[black] (240:4) circle (0.1cm);
\filldraw[black] (300:4) circle (0.1cm);

\filldraw[white] (90:4) circle (0.1cm);
\filldraw[white] (150:4) circle (0.1cm);
\filldraw[white] (30:4) circle (0.1cm);
\filldraw[white] (-30:4) circle (0.1cm);
\filldraw[white] (-90:4) circle (0.1cm);
\filldraw[white] (-150:4) circle (0.1cm);
\draw (90:4) circle (0.1cm);
\draw (-90:4) circle (0.1cm);
\draw (150:4) circle (0.1cm);
\draw (-150:4) circle (0.1cm);
\draw (30:4) circle (0.1cm);
\draw (-30:4) circle (0.1cm);
\end{tikzpicture}
\end{center}

Let $|\sigma|$ denote the number of parts of $\sigma$.  In the above example $|\sigma| = 3$ and $|\tsigma| = 4$.  The following result is straightforward.
\begin{lemma}\label{lem:dualpart}
If $\sigma$ and $\tsigma$ are dual non-crossing partitions, then $|\sigma|+ |\tsigma| = n+1$.
\end{lemma}

If $\sigma$ is a non-crossing partition on $[\bar n]$ we get a non-crossing matching $\tau(\sigma)$ on $[2n]$.  The matching $\tau(\sigma)$ separates $\sigma$ from $\tsigma$.  We fix labellings as follows: the vertex $2i-1$ in $\tau(\sigma)$ lies between $\widetilde{i-1}$ and $\bar i$; the vertex $2i$ lies between $\bar i$ and $\tilde i$.  For our above example, $\sigma = (\bar 1, \bar 4, \bar 6| \bar 2, \bar 3| \bar 5)$ gives $\tau(\sigma) = \{(1,12),(2,7),(3,6),(4,5),(8,11),(9,10)\}$.

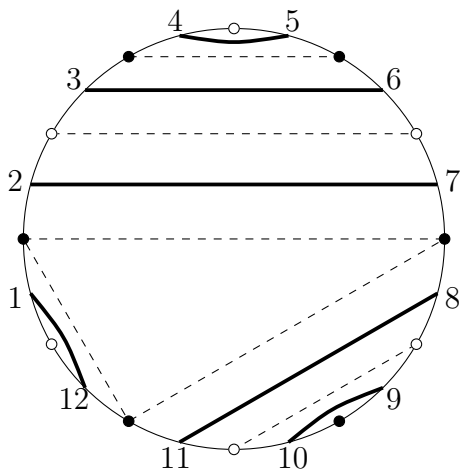
\begin{center}
\begin{tikzpicture}[scale = 0.7]
\draw (0,0) circle (4cm);
\foreach \i in {1,...,12}
{
	\node at (225 - \i*30:4.3) {$\i$};
}
\draw[line width = 0.05cm] (225-30:4) .. controls (225- 15:3.7) .. (225-12*30:4);
\draw[line width = 0.05cm] (225-4*30:4) .. controls (225- 4.5*30:3.7) .. (225-5*30:4);
\draw[line width = 0.05cm] (225-2*30:4) -- (225-7*30:4);
\draw[line width = 0.05cm] (225-3*30:4) -- (225-6*30:4);
\draw[line width = 0.05cm] (225-8*30:4) -- (225-11*30:4);
\draw[line width = 0.05cm] (225-9*30:4) .. controls (225- 9.5*30:3.7) .. (225-10*30:4);

\draw[dashed] (180:4) -- (0:4) -- (240:4) -- (180:4);
\draw[dashed] (120:4) -- (60:4);
\draw[dashed] (150:4) -- (30:4);
\draw[dashed] (-30:4) -- (-90:4);
\filldraw[black] (180:4) circle (0.1cm);
\filldraw[black] (120:4) circle (0.1cm);
\filldraw[black] (60:4) circle (0.1cm);
\filldraw[black] (0:4) circle (0.1cm);
\filldraw[black] (240:4) circle (0.1cm);
\filldraw[black] (300:4) circle (0.1cm);

\filldraw[white] (90:4) circle (0.1cm);
\filldraw[white] (150:4) circle (0.1cm);
\filldraw[white] (30:4) circle (0.1cm);
\filldraw[white] (-30:4) circle (0.1cm);
\filldraw[white] (-90:4) circle (0.1cm);
\filldraw[white] (-150:4) circle (0.1cm);
\draw (90:4) circle (0.1cm);
\draw (-90:4) circle (0.1cm);
\draw (150:4) circle (0.1cm);
\draw (-150:4) circle (0.1cm);
\draw (30:4) circle (0.1cm);
\draw (-30:4) circle (0.1cm);
\end{tikzpicture}
\end{center}

\begin{lemma}\label{lem:planarset}
$\sigma \mapsto \tau(\sigma)$ gives a bijection between $\NC_n$ and non-crossing matchings on $[2n]$.  Thus the number of non-crossing partitions on $n$ vertices is equal to the Catalan number $\frac{1}{n+1}\binom{2n}{n}$.
\end{lemma}

For a non-crossing partition $\sigma$, we define 
$$
L_\sigma(\Gamma) := \sum_{F \mid \sigma(F) = \sigma} \wt(F)
$$
where the summation is over all groves with boundary partition $\sigma$, and $\wt(F)$ is the product of weights of edges in $F$.  

\begin{center}
\begin{tikzpicture}[scale = 0.5]
\node at (0:4.6) {$\bar 5$};
\node at (72:4.6) {$\bar 4$};
\node at (144:4.6) {$\bar 3$};
\node at (216:4.6) {$\bar 2$};
\node at (288:4.6) {$\bar 1$};
\draw (0,0) circle (4cm);
\draw[line width = 0.05cm] (288:4) -- node [left] {$2$} (270:2);
\draw[line width = 0.005cm] (270:2) -- node [left] {$5$} (100:2);
\draw[line width = 0.05cm] (270:2) -- (0:4);
\draw[line width = 0.05cm] (100:2) -- (216:4);
\draw[line width = 0.05cm] (100:2) -- (72:4);
\filldraw[black] (0:4) circle (0.1cm);
\filldraw[black] (72:4) circle (0.1cm);
\filldraw[black] (144:4) circle (0.1cm);
\filldraw[black] (216:4) circle (0.1cm);
\filldraw[black] (288:4) circle (0.1cm);
\filldraw[black] (100:2) circle (0.1cm);
\filldraw[black] (270:2) circle (0.1cm);
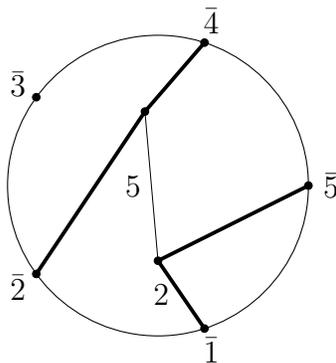
\node at (0,-5.3) {a grove $F$ with $\sigma(F) = (\bar 1 \bar 5| \bar 2 \bar 4| \bar 3)$ and $\wt(F) = 2$};
\end{tikzpicture}
\end{center}

Let ``$\uncrossed$" denote the boundary partition where each vertex is in its own part.  The following result is essentially due to Kirchhoff.  See Kenyon and Wilson \cite{KW}.

\begin{proposition}\label{prop:Lijresponse}
We have
$$
\dfrac{L_{(i,j|{\rm rest \; singletons})}(\Gamma)}{L_{\uncrossed}(\Gamma)} = - \Lambda_{i,j}(\Gamma).
$$
\end{proposition}

Thus the grove measurements determine the response matrix.  As shorthand, we write $L_{ij}$ for $L_{(i,j|{\rm rest \; singletons})}$.

\subsection{Boundary spikes and boundary edges}
Given an odd integer $2k-1$, for $k = 1,2,\ldots,n$ and a nonnegative real number $t$, we define $v_{2k-1}(t)(\Gamma)$ to be the electrical network obtained from $\Gamma$ by adding a new edge from a new vertex $v$ to $\bar k$, with weight $1/t$, followed by treating $\bar k$ as an interior vertex, and the new vertex $v$ as a boundary vertex (now named $\bar k$).

Given an even integer $2k$, for $k = 1,2,\ldots,n$, and a nonnegative real number $t$, we define $v_{2k}(t)(\Gamma)$ to be the electrical network obtained from $\Gamma$ by adding a new edge 
from $\bar k$ to $\overline{k+1}$ (indices taken modulo $n+1$), with weight $t$.

These operations are called \defn{adjoining a boundary spike}, and \defn{adjoining a boundary edge} respectively.  
Our notation suggests, as explained in Theorem \ref{thm:LP} below, that there is some symmetry between these two types of operations.  

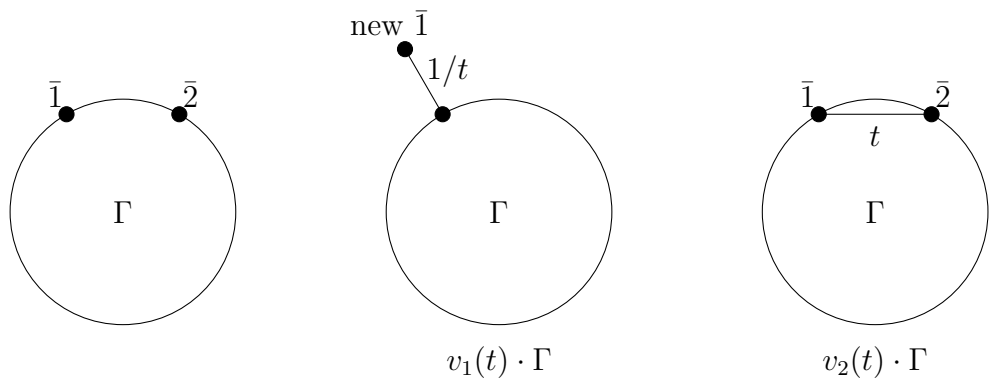
\begin{center}
\begin{tikzpicture}
\draw (0,0) circle (1.5cm);
\node at  (0,0) {$\Gamma$};
\node at (120:1.8) {$\bar 1$};
\filldraw[black] (120:1.5) circle (0.1cm);
\node at (60:1.8) {$\bar 2$};
\filldraw[black] (60:1.5) circle (0.1cm);
\begin{scope}[shift={(5,0)}]
\draw (0,0) circle (1.5cm);
\node at  (0,0) {$\Gamma$};
\draw (120:1.5) -- (120:2.5);
\node at (110:2) {$1/t$};
\node at (120:2.9) {new $\bar 1$};
\filldraw[black] (120:1.5) circle (0.1cm);
\filldraw[black] (120:2.5) circle (0.1cm);
\node at  (0,-2) {$v_1(t) \cdot \Gamma$};
\end{scope}

\begin{scope}[shift={(10,0)}]
\draw (0,0) circle (1.5cm);
\node at  (0,0) {$\Gamma$};
\draw (120:1.5) -- (60:1.5);
\node at (90:1) {$t$};
\node at (120:1.8) {$\bar 1$};
\filldraw[black] (120:1.5) circle (0.1cm);
\node at (60:1.8) {$\bar 2$};
\filldraw[black] (60:1.5) circle (0.1cm);
\node at  (0,-2) {$v_2(t) \cdot \Gamma$};
\end{scope}
\end{tikzpicture}
\end{center}

\begin{lemma}
$\Lambda(v_i(a) \cdot \Gamma)$ depends only on $\Lambda(\Gamma)$, giving an operation $v_i(t)$ on response matrices.  
\end{lemma}


\subsection{Electrically-equivalent transformations of networks} \label{ss:trans}
\label{sec:trans}

The following proposition is well-known and can be found for example in \cite{CGV}.
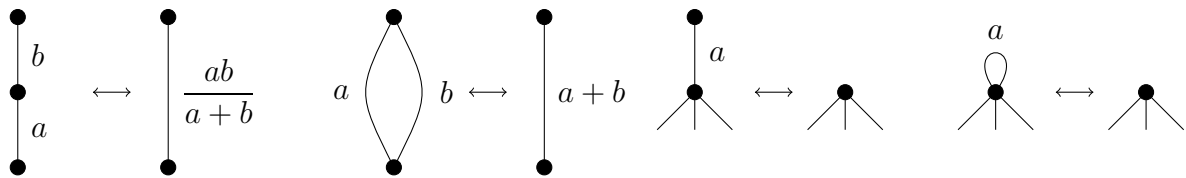
\begin{center}
\begin{tikzpicture}
\draw (0,0) -- node [right = 1pt] {$a$} (0,1) -- node [right = 1pt] {$b$} (0,2);
\filldraw[black] (0,0) circle (0.1cm);
\filldraw[black] (0,1) circle (0.1cm);
\filldraw[black] (0,2) circle (0.1cm);
\draw[<->] (1,1) -- (1.5,1);
\begin{scope}[shift = {(2.0,0)}]
\draw (0,0) -- node [right = 1pt] {$\dfrac{ab}{a+b}$} (0,2);
\filldraw[black] (0,0) circle (0.1cm);
\filldraw[black] (0,2) circle (0.1cm);
\end{scope}
\begin{scope}[shift = {(5,0)}]
\draw (0,0) .. controls (-0.5,1) .. (0,2);
\draw (0,0) .. controls (0.5,1) .. (0,2);
\node at (-0.7,1) {$a$};
\node at (0.7,1) {$b$};
\filldraw[black] (0,0) circle (0.1cm);
\filldraw[black] (0,2) circle (0.1cm);
\draw[<->] (1,1) -- (1.5,1);
\begin{scope}[shift = {(2.0,0)}]
\draw (0,0) -- node [right = 1pt] {${a+b}$} (0,2);
\filldraw[black] (0,0) circle (0.1cm);
\filldraw[black] (0,2) circle (0.1cm);
\end{scope}
\end{scope}
\begin{scope}[shift={(9,0)}]
\draw (0,1) -- (0,2);
\node at (0.3,1.5) {$a$};
\draw (0,1) -- (-0.5,0.5);
\draw (0,1) -- (0,0.5);
\draw (0,1) -- (0.5,0.5);
\filldraw[black] (0,1) circle (0.1cm);
\filldraw[black] (0,2) circle (0.1cm);
\draw[<->] (0.8,1) -- (1.3,1);
\begin{scope}[shift = {(2,0)}]
\draw (0,1) -- (-0.5,0.5);
\draw (0,1) -- (0,0.5);
\draw (0,1) -- (0.5,0.5);
\filldraw[black] (0,1) circle (0.1cm);
\end{scope}
\end{scope}
\begin{scope}[shift={(13,0)}]
\draw (0,1) .. controls (-0.5,1.7) and (0.5,1.7) .. (0,1);
\node at (0,1.8) {$a$};
\draw (0,1) -- (-0.5,0.5);
\draw (0,1) -- (0,0.5);
\draw (0,1) -- (0.5,0.5);
\filldraw[black] (0,1) circle (0.1cm);
\draw[<->] (0.8,1) -- (1.3,1);
\begin{scope}[shift = {(2,0)}]
\draw (0,1) -- (-0.5,0.5);
\draw (0,1) -- (0,0.5);
\draw (0,1) -- (0.5,0.5);
\filldraw[black] (0,1) circle (0.1cm);
\end{scope}
\end{scope}
\end{tikzpicture}
\end{center}


\begin{prop}\label{prop:SP}
Series-parallel transformations, removing loops, and removing interior degree 1 vertices (pendant removal), do not change the response matrix of a network.
\end{prop}



The following theorem is attributed to Kennelly \cite{Kenn}.

\begin{theorem}[$Y-\Delta$, or star-triangle transformation] \label{thm:YDelta}
 Assume that parameters $a$,$b$,$c$ and $A$,$B$,$C$ are related by 
$$A = \frac{bc}{a+b+c}, \;\; B = \frac{ac}{a+b+c}, \;\; C= \frac{ab}{a+b+c},$$
or equivalently by 
$$a = \frac{AB+AC+BC}{A}, \;\; b= \frac{AB+AC+BC}{B}, \;\; c = \frac{AB+AC+BC}{C}.$$
Then switching a local part of an electrical network between the two options shown does not change the response matrix of the whole network.
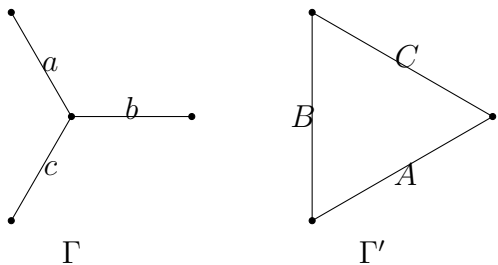
\begin{center}
\begin{tikzpicture}[scale = 0.4]

\draw (0:4) -- (0:0) -- (120:4);
\draw (0:0) -- (240:4);

\node at ($(120:2)+(0.3,0)$) {$a$};
\node at ($(0:2)+(0,0.3)$) {$b$};
\node at ($(240:2)+(0.3,0)$) {$c$};
\filldraw[black] (120:4) circle (0.1cm);
\filldraw[black] (0:4) circle (0.1cm);
\filldraw[black] (240:4) circle (0.1cm);
\filldraw[black] (0:0) circle (0.1cm);
\node at (0,-4.5) {$\Gamma$};
\begin{scope}[shift={(10,0)}]
\draw (0:4) -- (120:4) -- (240:4) -- (0:4);
\node at ($(240:4.4)!0.5!(0:4.4)$) {$A$};
\node at ($(240:4.6)!0.5!(120:4.6)$) {$B$};
\node at ($(120:4.6)!0.5!(0:4.6)$) {$C$};
\filldraw[black] (120:4) circle (0.1cm);
\filldraw[black] (0:4) circle (0.1cm);
\filldraw[black] (240:4) circle (0.1cm);
\node at (0,-4.5) {$\Gamma'$};
  \end{scope}
\end{tikzpicture}
\end{center}
\end{theorem}

In \cite{LP}, we showed with Pylyavskyy that Theorem \ref{thm:YDelta} implies
\begin{theorem}\label{thm:LP}
The generators $v_i(t)$ acting on response matrices satisfies the \defn{electrical braid relations}
\begin{enumerate}
\item
$v_i(a)v_i(b) = v_i(a+b)$ for each $i$,
\item
$v_i(a) v_j(b) = v_j(b) v_i(a)$ for $|i-j|\geq 2$, and
\item
$
v_i(a) v_{i \pm 1}(b) v_i(c) = v_{i \pm 1}({bc}/({a+c+abc})) v_i(a+c+abc) v_{i \pm 1}({ab}/({a+c+abc})).
$
\end{enumerate}
Here the index $i$ of $v_i$ is taken modulo $2n$.
\end{theorem}

\begin{remark}
The matrices
$$
x_1(a) = \left(\begin{array}{ccc} 1 & a & 0 \\ 0 & 1 & 0 \\ 0 & 0 & 1\end{array}\right) \qquad x_2(a) = \left(\begin{array}{ccc} 1 & 0 & 0 \\ 0 & 1 & a \\ 0 & 0 & 1\end{array}\right)
$$
satisfy the \defn{Lusztig braid relation} $$x_1(a) x_2(b) x_1(c) = x_2(bc/(a+c))x_1(a+c) x_2(ab/(a+c)).$$  The electrical braid relation can be thought of as a deformation of this, see \cite{LP}.  The matrices $x_i(a)$ will play an important role later, see Section \ref{sec:xy}.
\end{remark}

\subsection{Medial graphs}\label{sec:medial}
Let $\Gamma$ be an electrical network.  The medial graph $G(\Gamma)$ is defined as follows, and only depends on the underlying unweighted graph of $\Gamma$.  First place vertices $t_1,t_2,\ldots,t_{2n}$ on the boundary of the disk, so that in circular order we have $t_1 < \bar 1 < t_2 < t_3 <\bar  2 < \cdots$.  Next add a vertex $t_e$ for each edge $e \in E(\Gamma)$.  Now join $t_e$ with $t_{e'}$ if $e$ and $e'$ share a vertex and are incident to the same face.  For boundary vertices $t_{2i-1}$ or $t_{2i}$, we draw an edge to $t_e$ where $e$ is the ``closest" edge incident to vertex $i$ of $\Gamma$.  If vertex $i$ of $\Gamma$ is isolated, $t_{2i-1}$ and $t_{2i}$ are joined by an edge.  Note that each vertex $t_e$ is four-valent, and each vertex $t_i$ has degree $1$.  A \defn{strand} or \defn{wire} of a medial graph $G(\Gamma)$ is a maximal sequence of edges in $G$ such that we always go straight through any four-valent vertex encountered.  We will often write $T$ for a strand, or $T_i$ for the strand with one endpoint at $t_i$.  Medial strands either join boundary vertices to boundary vertices, or forms a cycle in the interior of the disk.  Thus a medial graph induces a matching on the set $[2n]$, which we will call the \defn{medial pairing} $\tau(\Gamma)$ of $\Gamma$.  Usually, we will only talk about the medial pairing $\tau(\Gamma)$ when $\Gamma$ is critical, to be defined below.  We will sometimes think of medial pairings as set partitions, and sometimes as involutions on $[2n]$.  For example, the set partition $\{(1,4),(2,5),(3,6)\}$ corresponds to the involution $\tau(i) = i+3 \mod 6$.

The electrical network $\Gamma$ can be recovered from the medial graph as follows: the graph $G$ divides the interior of the disk into regions.  The regions can be colored with two colors black and white, so that regions sharing an edge have different colors.  By convention, regions containing boundary vertices in their boundary are colored white.  To reconstruct $\Gamma$, place a vertex inside each white region (if this is a boundary region, then this vertex is just the boundary vertex), and join vertices with edges in $\Gamma$ when the corresponding white regions share a common vertex in $G$.

\begin{remark}
Note that in general we do not draw medial graphs with straight lines, but we draw edges as curves, giving an embedding of the graph into the interior of the disk.
\end{remark}

\begin{example}\label{ex:medial}
In the following picture, the medial graph of the electrical network from Section \ref{sec:response} is shown in dashed lines.  One of the medial strands is drawn extra thick.  The medial pairing is $\{(1,7),(2,9),(3,8),(4,10),(5,6)\}$.

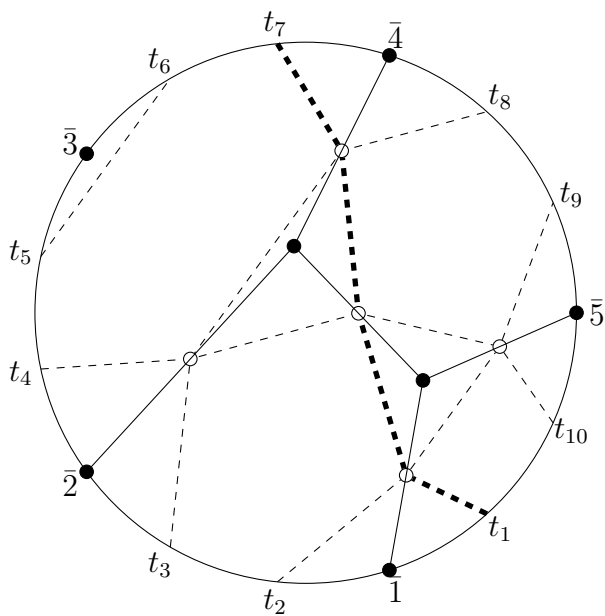
\begin{center}
\begin{tikzpicture}[scale=0.9]
\node at (0:4.3) {$\bar 5$};
\node at (72:4.3) {$\bar 4$};
\node at (144:4.3) {$\bar 3$};
\node at (216:4.3) {$\bar 2$};
\node at (288:4.3) {$\bar 1$};
\draw (0,0) circle (4cm);
\coordinate (c1) at (288:4);
\coordinate (c2) at (216:4);
\coordinate (c5) at (0:4);
\coordinate (c4) at (72:4);
\coordinate (c3) at (144:4);
\coordinate (a) at (330:2);
\coordinate (b) at (100:1);

\coordinate (ab) at ($(a)!0.5!(b)$);
\coordinate (b4) at ($(b)!0.5!(c4)$);
\coordinate (b2) at ($(b)!0.5!(c2)$);
\coordinate (a1) at ($(a)!0.5!(c1)$);
\coordinate (a5) at ($(a)!0.5!(c5)$);

\coordinate (t1) at (312:4);
\coordinate (t2) at (264:4);
\coordinate (t3) at (240:4);
\coordinate (t4) at (192:4);
\coordinate (t5) at (168:4);
\coordinate (t6) at (120:4);
\coordinate (t7) at (96:4);
\coordinate (t8) at (48:4);
\coordinate (t9) at (24:4);
\coordinate (t10) at (-24:4);

\node at (312:4.3) {$t_1$};
\node at (264:4.3) {$t_2$};
\node at (240:4.3) {$t_3$};
\node at (192:4.3) {$t_4$};
\node at (168:4.3) {$t_5$};
\node at (120:4.3) {$t_6$};
\node at (96:4.3) {$t_7$};
\node at (48:4.3) {$t_8$};
\node at (24:4.3) {$t_9$};
\node at (-24:4.3) {$t_{10}$};

\draw[dashed,line width=2pt] (t1) -- (a1) -- (ab) -- (b4) -- (t7);
\draw[dashed] (t2) -- (a1) -- (a5) -- (t9);
\draw[dashed] (t3) -- (b2) -- (b4) -- (t8);
\draw[dashed] (t4) -- (b2) -- (ab) -- (a5) -- (t10);
\draw[dashed] (t5) -- (t6);

\filldraw[white] (ab) circle (0.1cm);
\draw[black] (ab) circle (0.1cm);
\filldraw[white] (b4) circle (0.1cm);
\draw[black] (b4) circle (0.1cm);
\filldraw[white] (b2) circle (0.1cm);
\draw[black] (b2) circle (0.1cm);
\filldraw[white] (a1) circle (0.1cm);
\draw[black] (a1) circle (0.1cm);
\filldraw[white] (a5) circle (0.1cm);
\draw[black] (a5) circle (0.1cm);

\draw (c1) -- (a);
\draw (a) --  (b);
\draw (a) -- (c5);
\draw (b) -- (c2);
\draw (b) -- (c4);
\filldraw[black] (c1) circle (0.1cm);
\filldraw[black] (c2) circle (0.1cm);
\filldraw[black] (c3) circle (0.1cm);
\filldraw[black] (c4) circle (0.1cm);
\filldraw[black] (c5) circle (0.1cm);
\filldraw[black] (a) circle (0.1cm);
\filldraw[black] (b) circle (0.1cm);
\end{tikzpicture}
\end{center}
\end{example}

A \defn{lens} in a medial graph consists of two edge disjoint arcs $[x,x']_p$ and $[x,x']_q$ of wires $p$ and $q$ between two vertices $x, x'$ that lie on both wires:

\begin{center}
\begin{tikzpicture}[scale=0.5]
\draw (0,0) .. controls (5,2) .. (10,0);
\draw (0,2) .. controls (5,0) .. (10,2);
\node at (5,-1) {A lens};
\end{tikzpicture}
\end{center} 

A medial graph is \defn{lensless} if every wire begins and ends on the boundary of the disk, no wire has a self intersection, and there are no lens.  An electrical network $\Gamma$ is \defn{critical}, or \defn{reduced} if its medial graph is lensless.  For $\tau$ a medial pairing of a lensless medial graph $G$, let $c(\tau)$ denote the number of crossings of the medial pairing $\tau$.  (This number does not depend on the actual choice of medial graph, only that it is lensless.). For example, the medial pairing of Example \ref{ex:medial} has $c(\tau) = 5$.  The number $c(\tau)$ is also equal to the number of edges in the corresponding critical electrical network.

\begin{proposition}
If $\Gamma$ and $\Gamma'$ are related by $Y-\Delta$ moves, then $G(\Gamma)$ and $G(\Gamma')$ are related by Yang-Baxter (or star-triangle) moves (see Figure \ref{fig:Yang-Baxter}).  If $\Gamma$ and $\Gamma'$ are related by the reduction moves of Proposition \ref{prop:SP}, then $G(\Gamma)$ and $G(\Gamma')$ are related by lens removals or loop removals (see Figures \ref{fig:lens} and \ref{fig:loop}).
\end{proposition}

\subsection{Main results for circular planar electrical networks}
Let $A$ be an $n \times n$ matrix, and suppose $I = \{i_1,\ldots,i_r\}$ and $J = \{j_1,\ldots,j_r\}$ are disjoint subsets that index a subset of rows and columns respectively.  We call $A_{I,J}$ a \defn{circular minor} if after a cyclic permutation the sequence $(i_1,i_2,\ldots,i_r,j_r,\ldots,j_1)$ is in order.

\begin{theorem}[\cite{CIM,CGV}]\label{thm:CIM}\
 \begin{enumerate}
\item Any circular planar electrical network is electrically equivalent to some critical network.  
\item Any two circular planar electrical networks having the same response matrix can be connected by simple local transformations: series-parallel, loop removal, pendant removal, and star-triangle transformations.  Furthermore, if both networks are critical, then only star-triangle transformations are required.
\item The edge conductances of a critical circular planar electrical network can be recovered uniquely from the response matrix.
\item The response matrices realizable by circular planar networks is the space of $n \times n$ symmetric matrices such that each row sum is equal to 0, and $(-1)^r \det(A_{I,J}) \geq 0$ for any $r \times r$ circular minor $A_{I,J}$.
\item 
The space $E'_n$ of response matrices of circular planar networks has a stratification by cells $E'_n = \bigsqcup C_i$ where each $C_i \simeq \R_{>0}^{d_i}$ can be obtained as the set of response matrices for a fixed critical network with varying edge weights.
 \end{enumerate}
\end{theorem}

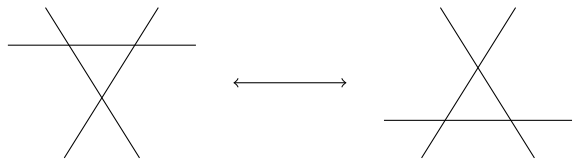
\begin{figure}
\begin{tikzpicture}[scale=0.5]
\draw (0,0) -- (5,0);
\draw (1,1) -- (3.5,-3);
\draw (4,1) -- (1.5,-3);
\draw[<->] (6,-1) -- (9,-1);
\begin{scope}[shift = {(10,-2)}]
\draw (0,0) -- (5,0);
\draw (1.5,3) -- (4,-1);
\draw (3.5,3) -- (1,-1);
\end{scope}
\end{tikzpicture}
\caption{The Yang-Baxter, or star-triangle transformation}
\label{fig:Yang-Baxter}
\end{figure}

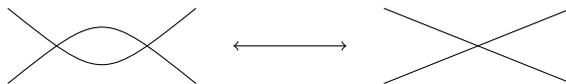
\begin{figure}
\begin{tikzpicture}[scale=0.5]
\draw (0,0) .. controls (2.5,2) .. (5,0);
\draw (0,2) .. controls (2.5,0) .. (5,2);
\draw [<->] (6,1) -- (9,1);
\draw (10,0) -- (15,2);
\draw (10,2) -- (15,0);
\end{tikzpicture}

\caption{Lens removal}
\label{fig:lens}
\end{figure}

\begin{figure}

\begin{tikzpicture}[scale = 0.5]
\draw (1.5,-1.5) -- (2,-2);
\draw (2,-2) .. controls (3,-3) and (3,-1) .. (2,-2);
\draw (2,-2) -- (1.5,-2.5);
\draw [<->] (5,-2) -- (8,-2);
\draw (10.5,-1.5) -- (10.5,-2.5);
\end{tikzpicture}

\caption{Loop removal}
\label{fig:loop}
\end{figure}
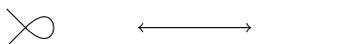

%

\begin{remark}
Not every matching on $[2n]$ can arise from electrical networks.  But in Section \ref{sec:cactus}, we will compactify the space to get all possible matchings on $[2n]$.
\end{remark}

\section{Planar bipartite graphs and the totally nonnegative Grassmannian}\label{sec:TNN}
In this section, we recall Postnikov's theory \cite{Pos} of the totally nonnegative (TNN) Grassmannian.  We follow the approach using planar bipartite graphs and matchings developed in \cite{Lamnotes}.  The connection between the matchings approach and Postnikov's is explained in the works of Talaska \cite{Tal} and Postnikov, Speyer, and Williams \cite{PSW}.  Though our approach is different, most of the results stated here are due to Postnikov, and we have also drawn from work of Oh \cite{Oh} and Knutson, Lam, and Speyer \cite{KLS}.  Some of the statements in Sections \ref{sec:xy} are new.

\subsection{TNN Grassmannian}
In this section, we fix integers $k,n$ and consider the real Grassmannian $\Gr(k,n)$ of (linear) $k$-planes in $\R^n$.  Recall that each $X \in \Gr(k,n)$ has Pl\"ucker coordinates $\Delta_I(X)$ labeled by $k$-element subsets $I \subset [n]$, defined up to a single common scalar.  The TNN Grassmannian $\Gr(k,n)_{\geq 0}$ is the subset of $\Gr(k,n)$ consisting of points $X$ represented by nonnegative Pl\"ucker coordinates $\Delta_I(X) \geq 0$ for all $I \in \binom{[n]}{k}$.  The Pl\"ucker coordinates $\{\Delta_I(X) \mid I \in \binom{[n]}{k}\}$ satisfy quadratic relations known as Pl\"ucker relations, see \cite{Ful}.

The cyclic group acts on $\Gr(k,n)_{\geq 0}$ with generator $\chi$ acting by the map
$$
\chi: \left(v_1,v_2,\ldots,v_n\right) \to \left(v_2,\ldots,v_n, (-1)^{k-1}v_1\right)
$$
where $v_i$ are columns of some $k\times n$ matrix representing $X$.

\subsection{Matchings for bipartite graphs in a disk}
Let $N$ be a weighted bipartite network embedded in the disk with $n$ boundary vertices, labeled $1,2,\ldots,n$ in clockwise order.  Each vertex (including boundary vertices) is colored either black or white, and all edges join black vertices to white vertices.  We let $d$ be the number of interior white vertices minus the the number of interior black vertices.  Furthermore, we let $d' \in [n]$ be the number of white boundary vertices.  Finally, we assume that all boundary vertices have degree 1, and that edges cannot join boundary vertices to boundary vertices.

\begin{remark}
Since the graph is bipartite, this last condition ensures that the coloring of the boundary vertices is determined by the interior part of the graph.  So sometimes we will pretend that boundary vertices are not colored, and usually will omit the color of boundary vertices from pictures.
\end{remark}

An \defn{almost perfect matching} $\Pi$ is a subset of edges of $N$ such that 
\begin{enumerate}
\item
each interior vertex is used exactly once
\item
boundary vertices may or may not be used.
\end{enumerate}
The boundary subset $I(\Pi) \subset \{1,2,\ldots,n\}$ is the set of black boundary vertices that are used by $\Pi$ union the set of white boundary vertices that are not used.  By our assumptions we have $|I(\Pi)| =  k:=d' + d$.

Define the \defn{boundary measurement}, or \defn{dimer partition function} as follows.  For $I \subset [n]$ a $k$-element subset,
$$
\Delta_I(N) := \sum_{\Pi : I(\Pi) = I} \wt(\Pi)
$$
where $\wt(\Pi)$ is the product of the weight of the edges in $\Pi$.

\begin{theorem}\label{thm:matchingplucker}
Suppose $N$ has nonnegative real weights, and that almost perfect matchings of $N$ exist.  Then the homogeneous coordinates $(\Delta_I(N))_{I \in \binom{[n]}{k}})$ defines a point $M(N)$ in the Grassmannian $\Gr(k,n)_{\geq 0}$.  Furthermore, every $X \in \Gr(k,n)_{\geq 0}$ is realizable $X = M(N)$ by a planar bipartite graph.
\end{theorem}

The first statement of Theorem \ref{thm:matchingplucker} is essentially due to Kuo \cite{Kuo}.  

\subsection{Gauge equivalences and local moves}
The material of this subsection is not used essentially in the sequel.  We suggest the reader unfamiliar with this material to skip this section.

Let $N$ be a planar bipartite graph.  If $e_1,e_2,\ldots,e_r$ are adjacent to an interior vertex $v$, we can multiply all of their edge weights by the same constant $c \in \R_{> 0}$ to get a new graph $N'$, and we have $M(N') = M(N)$.  This is called a \defn{gauge equivalence}.

We also have the following \defn{local moves}:
\begin{enumerate}
\item[(M1)]
Spider move or square move \cite{Pos,GK}: assuming the leaf edges of the spider have been gauge fixed to 1, the transformation is
$$
a'=\frac{a}{ac+bd} \qquad b'=\frac{b}{ac+bd} \qquad c'=\frac{c}{ac+bd} \qquad d'=\frac{d}{ac+bd}
$$
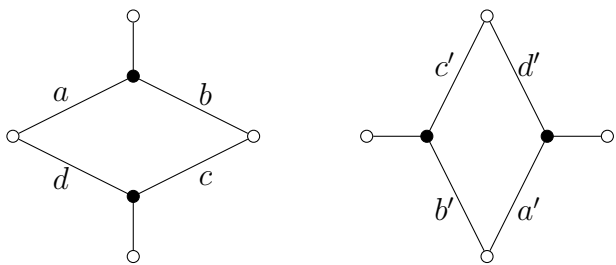
\begin{center}
\begin{tikzpicture}[scale=0.8]
\draw (-2,0) -- (0,1)--(2,0)--(0,-1)--  (-2,0);
\draw (0,1) -- (0,2);
\draw (0,-1) -- (0,-2);
\node at (-1.2,0.7) {$a$};
\node at (-1.2,-0.7) {$d$};
\node at (1.2,0.7) {$b$};
\node at (1.2,-0.7) {$c$};

\filldraw[black] (0,1) circle (0.1cm);
\filldraw[black] (0,-1) circle (0.1cm);
\filldraw[white] (-2,0) circle (0.1cm);
\draw (-2,0) circle (0.1cm);
\filldraw[white] (0,2) circle (0.1cm);
\draw (0,2) circle (0.1cm);
\filldraw[white] (2,0) circle (0.1cm);
\draw (2,0) circle (0.1cm);
\filldraw[white] (0,-2) circle (0.1cm);
\draw (0,-2) circle (0.1cm);
\end{tikzpicture}
\hspace{30pt}
\begin{tikzpicture}[scale=0.8]
\draw (0,-2) -- (1,0)-- (0,2)-- (-1,0)-- (0,-2);
\draw (1,0) -- (2,0);
\draw (-1,0) -- (-2,0);
\node at (0.7,-1.2) {$a'$};
\node at (-0.7,-1.2) {$b'$};
\node at (0.7,1.2) {$d'$};
\node at (-0.7,1.2) {$c'$};

\filldraw[black] (1,0) circle (0.1cm);
\filldraw[black] (-1,0) circle (0.1cm);
\filldraw[white] (-2,0) circle (0.1cm);
\draw (-2,0) circle (0.1cm);
\filldraw[white] (0,2) circle (0.1cm);
\draw (0,2) circle (0.1cm);
\filldraw[white] (2,0) circle (0.1cm);
\draw (2,0) circle (0.1cm);
\filldraw[white] (0,-2) circle (0.1cm);
\draw (0,-2) circle (0.1cm);

\end{tikzpicture}
\end{center}
\item[(M2)]
Valent two vertex removal.  If $v$ has degree two, we can gauge fix both incident edges $(v,u)$ and $(v,u')$ to have weight 1, then contract both edges (that is, we remove both edges, and identify $u$ with $u'$).  Note that if $v$ is a valent two-vertex adjacent to boundary vertex $b$, with edges $(v,b)$ and $(v,u)$, then removing $v$ produces an edge $(b,u)$, and the color of $b$ flips.

\item[(R1)]
Multiple edges with same endpoints is the same as one edge with sum of weights.
\item[(R2)]
Leaf removal.  Suppose $v$ is leaf, and $(v,u)$ the unique edge adjacent to it.  Then we can remove both $v$ and $u$, and all edges adjacent to $u$.  However, if there is a boundary edge $(b,u)$ where $b$ is a boundary vertex, then that edge is replaced by a boundary edge $(b,w)$ where $w$ is a new vertex with the same color as $v$.
\item[(R3)]
Dipoles (two degree one vertices joined by an edge) can be removed.
\end{enumerate}

\begin{proposition}
Each of these relations preserves $M(N)$.
\end{proposition}

A planar bipartite graph $N$ is \defn{reduced} if it has the minimal number of faces in its move-equivalence class, and in addition, there are no leaves in $N$ connected to interior vertices.

\begin{theorem}\label{thm:Pos}
Suppose $N$ and $N'$ are planar bipartite graphs with $M(N) = M(N')$.  Then $N$ and $N'$ are related by local moves and gauge equivalences.  Suppose $N$ and $N'$ are reduced planar bipartite graphs with $M(N) = M(N')$.  Then $N$ and $N'$ are related by square moves, valent two vertex removals/additions, and gauge equivalences.
\end{theorem}

\subsection{Bounded affine permutations and Grassmann necklaces}
A \defn{bounded affine permutation}, or \defn{bounded juggling pattern} of type $(k,n)$ is a bijection $f:\Z \to \Z$ satisfying:
\begin{enumerate}
\item
$i \leq f(i) \leq i+n$
\item
$f(i+n) = f(i) + n$ for all $i \in \Z$
\item
$\sum_{i=1}^n (f(i)-i) = kn$
\end{enumerate}
A bijection satisfying only (2) and (3) is called an affine permutation.  The affine permutations of type $(k,n)$ are denoted $\tS_n^k$.  The bounded affine permutations of type $(k,n)$ are denoted $\Bound(k,n)$.  Bounded affine permutations inherit a length function, and a Bruhat order from the set of all affine permutations.  See \cite{KLS,BB} for full details.  It is often convenient to think of $f$ as a juggling pattern: $f(i) = j$ says that the ball thrown at time $i$ lands at time $j$.  Since $f$ is determined by its values on $[n]$, we sometimes give $f$ in window notation: $f = [f(1),f(2),\ldots,f(n)]$.

A \defn{(k,n)-Grassmann necklace} is a collection of $k$-element subsets $\I = (I_1,I_2,\ldots,I_n)$ satisfying the following property: for each $a \in [n]$:
\begin{enumerate}
\item
$I_{a+1} = I_a$ if $a \notin I_a$
\item
$I_{a+1} = I_a -\{a\} \cup \{a'\}$ if $a \in I_a$.
\end{enumerate}

For each $f \in \Bound(k,n)$ we define $\I(f) = (I_1,I_2,\ldots,I_n)$ by
$$
I_a := \{f(b) \mid b < a \text{ and } f(b) \geq a\} \mod n.
$$
Here ``$\mod n$" means that we take representatives in $[n]$.

We write $\leq_a$ for the ordering $a <_a a+1 <_a \cdots <_a n <_a 1 <_a \cdots <_a a-1$ on $[n]$.  We define the dominance ordering on $\binom{[n]}{k}$ by $I =\{i_1<i_2<\cdots<i_k\} \leq J =\{j_1<j_2<\cdots <j_k\}$ if $i_r \leq j_r$ for all $r$.  We also have the cyclically rotated version $I \leq_a J$: order the elements of $I$ and $J$ using $\leq_a$, and then compare them (in order) using $\leq_a$.  Define a partial order on Grassmann necklaces by $\I \leq \I'$ if $I_a \leq_a I'_a$ for each $a$.

\begin{theorem}\label{thm:Grassorder}
The map $f \mapsto \I(f)$ is a bijection between bounded affine permutations of type $(k,n)$ and $(k,n)$-Grassmann necklaces.  We have $f \leq f'$ in Bruhat order if and only if $\I(f) \leq \I(f')$.
\end{theorem}

\begin{example}
Let $k =2$ and $n = 6$.  Suppose $f = [2,4,6,5,7,9]$.  Then $\I(f) = (13,23,34,46,56,16)$.
\end{example}

Suppose $N$ is a reduced planar bipartite graph.  We define a bounded affine permutation $f_N$ as follows.  At each boundary vertex $i \in [n]$ we produce a path starting at $i$, called the \defn{trip} $T_i$, as follows.  Travel along the unique outgoing edge at $i$ and at each interior black vertex, turn maximally towards the right, while at each interior white vertex turn maximally towards the left.  The trip $T_i$ ends when we reach a boundary vertex $j \in [n]$.  We then define $f_N$ by $f_N(i) = j \mod n$.  There is a special case: when $T_i$ ends at $i$.  This can only happen (for a reduced planar bipartite graph) if $i$ is connected to a leaf.  In this case, we declare $f_N(i) = i$ if $i$ is connected to a black leaf, and $f_N(i) = i+n$ if $i$ is connected to a white leaf.

\subsection{Positroids and positroid varieties}
Let $X \in \Gr(k,n)$.  Define $f_X$ by
\begin{equation}\label{eq:perm}
f_X(i) := \min\{j \geq i \mid v_i \in \spn\{v_{i+1},v_{i+2},\ldots,v_j\}\}
\end{equation}
where $v_i$ are the columns of a representative of $X$, and we extend these columns periodically by defining $v_{i+n}:=v_i$.

\begin{proposition}
Suppose $N$ is a reduced planar bipartite graph.  Then $f_{M(N)} = f_N$.
\end{proposition}

Define $\I(X) := (I_1,I_2,\ldots,I_n)$ by setting $I_a(X)$ to be the lexicographically minimal non-vanishing Pl\"ucker coordinate of $X$ with respect to the order $\leq_a$.

\begin{proposition}
Let $X \in \Gr(k,n)$.  Then $f_X$ is a bounded affine permutation of type $(k,n)$ and $\I(X)$ is a $(k,n)$-Grassmann necklace.  We have $\I(f_X) = \I(X)$.
\end{proposition}

We shall consider matroids of rank $k$ on $[n]$.  The matroid $\M(X)$ of a point $X \in \Gr(k,n)$ is the collection
$$
\left\{I \in \binom{[n]}{k} \mid \Delta_I(X) \neq 0\right\}.
$$
A \defn{positroid} is the matroid of a point $X \in \Gr(k,n)_{\geq 0}$.

Let $\S_I = \{J \in \binom{[n]}{k} \mid I \leq J\}$ be the Schubert matroid with minimal element $I$.  Let $\S_{I,a} = \{J \in \binom{[n]}{k} \mid I \leq_a J\}$.   For $f \in \Bound(k,n)$ define a matroid 
\begin{equation}\label{eq:Oh}
\M(f):= \bigcap_{a=1}^n \S_{I_a,a}.
\end{equation}

For $f \in \Bound(k,n)$, define the \defn{open positroid variety}
$$
\oPi_f := \{X \in \Gr(k,n) \mid f_X = f\}.
$$
Obviously, we have $\Gr(k,n) = \bigsqcup_{f \in \Bound(k,n)} \oPi_f$.

\begin{theorem}\label{thm:TNNmain} \
\begin{enumerate}
\item
The \defn{closed positroid varieties} are given by
$$\Pi_f:=\overline{\oPi_f} = \bigsqcup_{f' \geq  f} \oPi_{f'}$$
\item
The intersection
$(\oPi_f)_{>0} := \oPi_f \cap \Gr(k,n)_{\geq 0}
$
is a cell $\R_{>0}^{k(n-k) - \ell(f)}$.  Each $X \in (\oPi_f)_{>0}$ has the same positroid $\M(f)$.
\end{enumerate}
\end{theorem}

To summarize, there are bijections between bounded affine permutations, Grassmann necklaces, and positroids.

\subsection{A reduction result}\label{sec:xy}
The group $GL_{n}$ acts on $\Gr(k,n)$ by right multiplication.  For $a \in \R$ and $1 \leq i \leq n-1$, the element $x_i(a) \in GL_{2n}$ is the elementary matrix differing from the identity matrix by an entry $a$ in the $i$-th row and $(i+1)$-st column.  Similarly $y_i(a)$ has an entry $a$ in the $(i+1)$-st row and $i$-th column.  For example, for $n = 2$,
$$
y_2(a) = \left( \begin{array}{cccc} 1 & 0 & 0 & 0 \\ 
0 & 1 &0 & 0 \\
0 & a & 1 & 0 \\
0 & 0 & 0 & 1 
\end{array}\right).
$$
The generator $x_i(a)$ acts on $\Gr(k,n)$ by adding $a$ times the $i$-th column to the $(i+1)$-st, and we define the action of $x_{n}(a)$ by using the signed cyclic action: $x_{n}(a) \cdot X = (\chi \circ x_1(a) \circ \chi^{-1}) \cdot X$.  We may also think of $x_{n}(a) \in GL_{n}$ as the matrix that differs from the identity by an entry in the $n$-th row and first column.  Similarly, $y_i(a)$ acts by adding the $a$ times the $(i+1)$-st column to the $i$-th column. 

We shall need the following \defn{dual Grassmann necklace}.  Let $f \in \Bound(k,n)$.  Define $\J(f) = (J_1,J_2,\ldots,J_{n})$ by  
$$
J_b(f) := \{a < b \mid f(a) \geq b\} \mod n \subset [n].
$$
Thus $\J(f)$ keeps track of where balls are thrown instead of where they land.  The same arguments as for Grassmann necklaces show that $\J(f)$ consists of $k$-element subsets of $[n]$, and that $f \mapsto \J(f)$ is injective.

\begin{proposition}\label{prop:reduce} \
\begin{enumerate}
\item
Suppose $X \in \oPi_f$ and $i < f(i) < f(i+1) < i+n+1$.  If $$a = \Delta_{I_{i+1}}(X)/\Delta_{I_{i+1} \cup \{i\} - \{i+1\}}(X)$$ is well-defined (that is, $\Delta_{I_{i+1} \cup \{i\} - \{i+1\}}(X) \neq 0$) then $X' = X \cdot x_i(-a) \in \oPi_{f'}$ where $f' = f s_i > f$.  Furthermore, $a$ is always well-defined if $X \in \Gr(k,n)_{\geq 0}$.  In this case $X'$ also lies in $\Gr(k,n)_{\geq 0}$.
\item
Suppose $X \in \oPi_f$ and $i-n < f^{-1}(i) < f^{-1}(i+1) < i+1$.  If $$a = \Delta_{J_{i+1}}(X)/\Delta_{J_{i+1} \cup \{i+1\} - \{i\}}(X)$$ is well-defined (that is, $\Delta_{J_{i+1} \cup \{i+1\} - \{i\}}(X) \neq 0$) then $X' = X \cdot y_i(-a) \in \oPi_{f'}$ where $f' = s_i f  > f$.  Furthermore, $a$ is always well-defined if $X \in \Gr(k,n)_{\geq 0}$.  In this case $X'$ also lies in $\Gr(k,n)_{\geq 0}$.
\end{enumerate}
%
\end{proposition}
\begin{proof}
We only prove (1), as (2) is similar.  Let $v_i$ be the columns of a $k \times n$ matrix which represents $X$.

Suppose $a$ is well defined.  Now $X'$ is obtained from $X$ by adding $-a$ times $v_i$ to $v_{i+1}$.  So for any $J$ we have
\begin{equation}\label{eq:Delta}
\Delta_J(X') = \begin{cases} \Delta_J(X) -a \Delta_{J -\{i+1\} \cup \{i\}}(X) &\mbox{ if $i+1 \in J$ and $i \notin J$} \\
\Delta_J(X) & \mbox{otherwise.}
\end{cases}
\end{equation}
The formulae above are the minors of this specific representative of $X'$; the Pl\"ucker coordinates of the actual point in the Grassmannian are only determined up to a scalar. 

Let $v'_i$ be the columns for the matrix obtained from $v_i$ by right multiplication by $x_i(-a)$.  Then $\spn(v_i) = \spn(v'_i)$ and $\spn(v_i,v_{i+1}) = \spn(v'_i,v'_{i+1})$, so $f_{X'}(r) = f_X(r)$ unless $r \in \{i,i+1\} \mod n$.  But $f_{X'} \neq f_X$ since $\Delta_{I_{i+1}}(X') = 0$.  Thus $f_{X'}$ must be obtained from $f_X$ by swapping the values of $f(i)$ and $f(i+1)$, so $f' = fs_i > f$.

Now suppose that $X \in \Gr(k,n)_{\geq 0}$.  If $f(i) = i+1$, then by \eqref{eq:perm}, the columns $v_i$ and $v_{i+1}$ are parallel, and since $f(i+1) \neq i+1$ both $v_i$ and $v_{i+1}$ are non-zero.  In this case $a$ is just the ratio $v_{i+1}/v_i$, and $X'$ is what we get by changing the $(i+1)$-st column to $0$.  All the claims follow.

We now assume that $f(i) > i+1$.  Let $f(i) = j$ and $f(i+1) =  k$.  Since $f(i) \notin \{i,i+n\}$, we have $i \in I_i$ and $i \notin I_{i+1}$.  We also have $i+1 \in I_i \cap I_{i+1}$.  We let $I_i = \{i,i+1\} \cup I$, $I_{i+1} = (i+1) \cup I \cup \{j\}$, and $I_{i+2} = I \cup \{j,k\}$ for some $I \subset [n]-\{i,i+1\}$.  Note that if $k = n+i$, then $I_{i+2} = I \cup \{j,i\}$; this immediately gives $\Delta_{i \cup I \cup j} \neq 0$.

Suppose $k \neq n+i$.  Then we have a Pl\"ucker relation
$$
\Delta_{i \cup I\cup j} \Delta_{(i+1) \cup I \cup k} = \Delta_{i \cup I \cup k} \Delta_{(i+1) \cup I \cup j} + \Delta_{i \cup (i+1) \cup I} \Delta_{I \cup j \cup k}
$$
where all subsets are ordered according to $\leq_i$.  (The easiest way to see that the signs are correct is just to take $i = 1$.)  Since the RHS is positive, $\Delta_{i \cup I \cup j} \neq 0$.  We have shown that $a$ is well defined.

Using \eqref{eq:Delta} and Lemma \ref{lem:reduce} below, we see that $X' \in \Gr(k,n)_{\geq 0}$.
\end{proof}

\begin{lemma}\label{lem:reduce}
Let $X \in \Gr(k,n)_{\geq 0}$ be as in Proposition \ref{prop:reduce}, with $f(i) > i+1$.  For simplicity of notation suppose $i = 1$.  
Write $I_2 = 2 \cup I \cup j$.  Suppose $J \subset \{3,\ldots,n\}$ satisfies $1 \cup J \in \M_X$.  Then $\Delta_{1\cup I \cup j}(X)\Delta_{2 \cup J}(X) \geq \Delta_{1 \cup J}(X) \Delta_{2 \cup I \cup j}(X)$.
\end{lemma}
\begin{proof}
Let $\M$ be the positroid of $X$.  
We let $I_1 = \{1,2\} \cup I$, $I_2 = 2 \cup I \cup \{j\}$, and $I_3 = I \cup \{j,k\}$, as in the proof of Proposition \ref{prop:reduce}. We have already shown in the proof of Proposition \ref{prop:reduce} that $(1 \cup I \cup j) \in \M$.

We proceed by induction on the size of $r = |(I \cup j) \setminus J|$.  The case $r = 0$ is tautological.  So suppose $r \geq 1$.  We may assume that $1 \cup J \in \M$ for otherwise the claim is trivial.  Applying the exchange lemma to $1 \cup J$ the element $a = \max(J \setminus (I \cup j)) \in J$ and the other base $1 \cup I \cup j$, we obtain $L = J -\{a\} \cup \{b\}$ such that $1 \cup L \in \M$.  

We claim that $b < a$.  To see this, note that $I_1 \leq (1 \cup J)$, which implies that $a > I \setminus J$.  So the only way that $b$ could be greater than $a$ is if $b = j$, and $a < j$.  But by assumption we also have $I_3 = I \cup \{j,k\} \leq_3 (1 \cup J)$ with $k \geq_2 j$.  This is impossible since both $k$ and $j$ are greater than $a$, but we have $J \setminus I \subset [3,a]$ -- the only element of $(1 \cup J) \setminus I$ that is greater than $j$ or $k$ in $\leq_3$ order is $1$. Thus $b < a$. 

So by induction we have that $\Delta_{2\cup L}/\Delta_{1 \cup L} \geq \Delta_{2 \cup I}/\Delta_{1 \cup I}$, where in particular we have $(1 \cup L), (2 \cup L) \in \M$.  It suffices to show that $\Delta_{2\cup J}/\Delta_{1 \cup J} \geq \Delta_{2\cup L}/\Delta_{1 \cup L}$.

We apply the Pl\"ucker relation to $\Delta_{2\cup J} \Delta_{1 \cup L}$, swapping $L$ with $(k-1)$ of the indices in $2\cup J$ to get
$$
\Delta_{1\cup L}\Delta_{2 \cup J}= \Delta_{1 \cup J} \Delta_{2 \cup L} +  \Delta_{12 j_1 j_2 \cdots \hat a \cdots j_{k-1}} \Delta_{\ell_1 \ell_2 \cdots a \cdots \ell_{k-1}} .
$$
We note that  $\ell_1 < \ell_2 < \cdots < a < \cdots < \ell_{k-1}$ is actually correctly ordered, since $L$ is obtained from $J$ by changing $a$ to a smaller number.  So all factors in the above expression are nonnegative.  The claim follows.
\end{proof}

\subsection{Bridges in planar bipartite graphs}\label{sec:bridge}
By adding degree two vertices to a planar bipartite graph $N$, we can always assume that a boundary vertex is the color we want it to be.  If $i$ and $i+1$ are two adjacent boundary vertices, we can add a \defn{bridge} between the two edges leaving $i$ and $i+1$.  There are two different kinds of bridges depending on which color is assigned to which vertex of the added edge.  For simplicity, we for example just say we are adding ``a bridge with white at $i+1$ and black at $i$'' for the following picture.

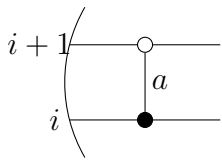
\begin{center}
\begin{tikzpicture}
\draw (0.2,3.5) arc (150:210:2);
\draw (0,2) -- (2,2);
\draw (0,3) -- (2,3);
\draw (1,2) -- (1,3);
\filldraw [black] (1,2) circle (0.1cm);
\filldraw [white] (1,3) circle (0.1cm);
\draw (1,3) circle (0.1cm);
\node at (-0.2,2) {$i$};
\node at (-0.4,3) {$i+1$};
\node at (1.2,2.5) {$a$};
\end{tikzpicture}
\end{center}

These are the network analogues of the Chevalley generator actions $x_i(a)$ and $y_i(b)$.

\begin{lemma}\label{lem:networkbridge}
Let $N$ be a network.  Now let $N'$ be obtained by adding a bridge with edge weight $a$ from $i$ to $i+1$ which is white at $i$ and black at $i+1$.  Then we have $M(N) = x_i(a) \cdot M(N)$, and the boundary measurements change as follows:
$$
\Delta_I(N') = \begin{cases} 
\Delta_I(N) + a\Delta_{I - \{i+1\} \cup \{i\}}(N) & \mbox{if $i+1 \in I$ but $i \notin I$} \\
\Delta_I(N) & \mbox{otherwise.}
\end{cases}
$$
If the bridge is black at $i$ and white at $i+1$, then we have $M(N) = y_i(a) \cdot M(N)$, and 
$$
\Delta_I(N') = \begin{cases} 
\Delta_I(N) + a\Delta_{I - \{i\} \cup \{i+1\}}(N) & \mbox{if $i \in I$ but $i+1 \notin I$} \\
\Delta_I(N) & \mbox{otherwise.}
\end{cases}
$$
\end{lemma}

\section{Compactifying the space of circular planar electrical networks}
\label{sec:cactus}
In this section we construct a compactification of the space of circular planar electrical networks from Section \ref{sec:elec}.  The section is organized as follows: in Section \ref{sec:cactus} we define cactus networks and in Section \ref{sec:grovecactus} we discuss grove coordinates on cactus networks.  The compactified space $E_n$ of circular planar electrical networks is defined in Section \ref{sec:compact}.  We discuss the $0$-dimensional cells of $E_n$ (a Catalan object) in Section \ref{sec:bottom}.  We introduce the uncrossing partial order on matchings in Section \ref{sec:poset} and prove that it is dual to an induced subposet of affine Bruhat order in Section \ref{sec:elecaffine}.  In Section \ref{sec:Catalan}, we observe that Grassmann necklaces of electrical affine permutations are special necklaces that we call Catalan necklaces.

\subsection{Cactus networks}\label{sec:cactus}
Let $S$ be a circle with $n$ boundary points labeled $[\bar n]$, as usual.  Let $\sigma$ be a non-crossing partition on $[\bar n]$.  Then identifying the boundary points according to the parts of $\sigma$ gives a \defn{hollow cactus} $S_\sigma$: it is a union of circles, glued together at the identified points.  (Note that it is possible for three or more circles to be glued together at the same point.)  We can still think of the interior of $S_\sigma$: this is a union of open disks.  The interior of $S_\sigma$, together with $S_\sigma$ itself will be called a \defn{cactus}.  We caution that our cacti are not identical to the similar notion in the theory of real stable curves.

A \defn{cactus network} is a weighted graph $\Gamma$ embedded into a cactus.  We may think of a cactus network as obtained from a usual circular planar network by declaring some boundary vertices (specified by $\sigma$) to have infinite conductance between them.  A cactus network also decomposes into a union of circular planar networks (with differing sets of boundary vertices) for each disk component of the cactus.  Any cactus network has a medial graph with the convention that medial strands/wires always stay completely within one disk.  Sometimes it is convenient to draw the medial strands of a cactus network in a disk, rather than in a cactus.

\begin{example}
The hollow cactus $S_\sigma$ where $\sigma = (\bar 1|\bar 2,\overline{13},\overline{14}|\bar 3,\bar 7|\bar 4,\bar 6|\bar 5|\bar 8|\bar9,\overline{12}|\overline{10}|\overline{11})$.
\begin{center}
\begin{tikzpicture}
\draw (0,0) circle (1cm);
\draw (2,0) circle (1cm);
\draw (2,2) circle (1cm);
\draw (4,0) circle (1cm);
\draw (6,0) circle (1cm);
\filldraw[black] (0:1) circle (0.1cm);
\filldraw[black] (120:1) circle (0.1cm);
\filldraw[black] (240:1) circle (0.1cm);
\filldraw[black] (2,3) circle (0.1cm);
\filldraw[black] (2,1) circle (0.1cm);
\filldraw[black] (3,0) circle (0.1cm);
\filldraw[black] (7,0) circle (0.1cm);
\filldraw[black] (2,-1) circle (0.1cm);
\filldraw[black] (5,0) circle (0.1cm);
\node at (2,3.3) {$\bar 1$};
\node at (2.3,1.3) {$\bar 2$};
\node at (3,0.65) {$\bar 3$};
\node at (5,0.65) {$\bar 4$};
\node at (7.3,0) {$\bar 5$};
\node at (5,-0.7) {$\bar 6$};
\node at (3,-0.7) {$\bar 7$};
\node at (2,-1.4) {$\bar 8$};
\node at (1,-0.7) {$\bar 9$};
\node at (240:1.4) {$\overline{10}$};
\node at (120:1.4) {$\overline{11}$};
\node at (1,0.7) {$\overline{12}$};
\node at (1.7,0.7) {$\overline{13}$};
\node at (1.7,1.3) {$\overline{14}$};
\end{tikzpicture}
\end{center}
\end{example}

There is a natural notion of the response matrix $\Lambda(\Gamma)$ of a cactus network: we specify voltages at boundary vertices such that vertices belonging to the same part of $\sigma$ are assigned the same voltage.  Two cactus networks are electrically equivalent if they have the same response matrices.  We shall call $\sigma$ the \defn{shape} of the cactus network $\Gamma$.

\begin{proposition}\label{prop:cactus}\
\begin{enumerate}
\item
Each cactus network $\Gamma$ is electrically equivalent to one whose medial graph is lensless.  Such cactus networks are called \defn{critical} or \defn{reduced}.
\item
Any two electrically equivalent reduced cactus networks are related by $Y-\Delta$ transformations.
\item
Any matching on $[2n]$ can be obtained as the medial pairing of some cactus network.
\end{enumerate}
\end{proposition}
\begin{proof}
(1) and (2) are proved in exactly the same way as the corresponding statement for circular planar electrical networks.  (3) is proved by construction: let $\tau$ be a matching on $[2n]$ and $G$ any lensless medial graph with that medial pairing.  The strands of $G$ cuts the disk up into regions.  We glue together boundary vertices $[\bar n]$ if they belong to the same such region.  This recovers the shape of the cactus network.  In each disk component of the cactus, we now have a matching (on a smaller number of vertices), and this matching arises from a critical electrical network in that disk.
\end{proof}

The number of medial pairings of cactus networks is just the number of matchings on $2n$ objects, that is $(2n-1) \cdot (2n-3)\cdots 3 \cdot 1$.
\begin{center}
\begin{tikzpicture}
\draw (0,0) circle (1cm);
\draw (2,0) circle (1cm);
\draw (-1,0) -- (1,0) -- (3,0);
\draw[dashed] (135:1) -- (-45:1);
\draw[dashed] (45:1) -- (-135:1);
\node at (135:1.4) {$t_2$};
\node at (45:1.2) {$t_3$};
\node at (-45:1.2) {$t_8$};
\node at (-135:1.4) {$t_1$};
\begin{scope}[shift={(2,0)}]
\draw[dashed] (135:1) -- (-45:1);
\draw[dashed] (45:1) -- (-135:1);
\node at (135:1.2) {$t_4$};
\node at (45:1.4) {$t_5$};
\node at (-45:1.4) {$t_6$};
\node at (-135:1.2) {$t_7$};
\end{scope}
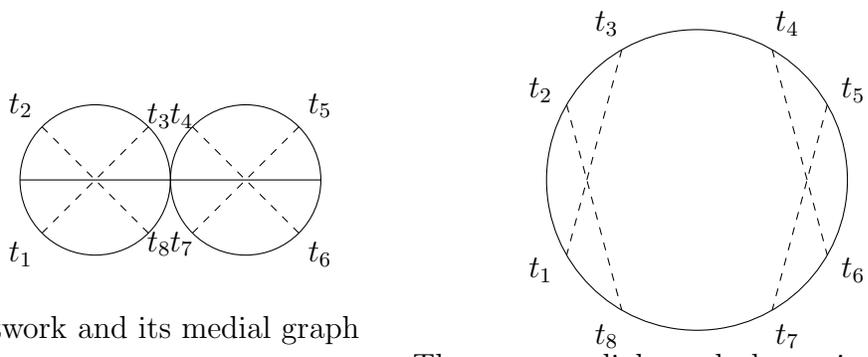
\node at (0,-2) {A cactus network and its medial graph};

\begin{scope}[shift={(8,0)}]
\draw (0,0) circle (2cm);
\draw[dashed] (120:2) -- (-150:2);
\draw[dashed] (150:2) -- (-120:2);
\draw[dashed] (60:2) -- (-30:2);
\draw[dashed] (30:2) -- (-60:2);
\node at (-150:2.4) {$t_1$};
\node at (150:2.4) {$t_2$};
\node at (120:2.4) {$t_3$};
\node at (60:2.4) {$t_4$};
\node at (30:2.4) {$t_5$};
\node at (-30:2.4) {$t_6$};
\node at (-60:2.4) {$t_7$};
\node at (-120:2.4) {$t_8$};
\node at (0,-2.5) {The same medial graph drawn in the disk};
\end{scope}
\end{tikzpicture}
\end{center}
\subsection{Grove measurements as projective coordinates}
\label{sec:grovecactus}
Recall that we have defined grove counting measurements $L_\sigma(\Gamma)$ for $\Gamma$ a planar electrical network and $\sigma$ a non-crossing partition.  The combinatorial definition of $L_\sigma(\Gamma)$ naturally extends to cactus networks $\Gamma$.  

Let $\P^{\NC_n}$ be the projective space with homogeneous coordinates indexed by non-crossing partitions.  The map
$$
\Gamma \longmapsto (L_\sigma(\Gamma))_{\sigma}
$$
sends a cactus network to a point in $\L(\Gamma) \in \P^{\NC_n}$.

Define the \defn{electroid} of $\L(\Gamma)$ by
$$
\E(\Gamma) := \E(\L(\Gamma)) := \{\sigma \mid L_\sigma \neq 0\} \subset \NC_n.
$$
\begin{remark}
Our electroids differ from Alman, Lian, and Tran's electrical positroids \cite{ALT} in two different ways: first, we study electroids for cactus networks, while \cite{ALT} only consider circular planar electrical networks.  Second, the elements of their electrical positroids are certain pairs of subsets of $[\bar n]$, which correspond to non-crossing partitions $\sigma$ with exactly two non-singleton parts of equal size.  (See \cite{KW} and \cite{Ken} for a discussion of the relationship between grove measurements and minors of the response matrix.)
\end{remark}
\begin{proposition}\label{prop:Lequiv}
If $\Gamma$ and $\Gamma'$ are electrically equivalent cactus networks, then $\L(\Gamma) = \L(\Gamma')$ in $\P^{\NC_n}$.  In particular, $\E(\Gamma)= \E(\Gamma')$.
\end{proposition}
\begin{proof}
Use the combinatorial definition of $L_\sigma(\Gamma)$, and do a check for each of the local electrical equivalences (series-parallel, star-triangle, and so on). 
\end{proof}
\begin{example}\
\label{ex:YDelta}
\begin{center}
\begin{tikzpicture}[scale = 0.5]
\draw (0,0) circle (4cm);
\draw (0:4) -- (0:0) -- (120:4);
\draw (0:0) -- (240:4);
\node at (120:4.3) {$\bar 1$};
\node at (0:4.3) {$\bar 2$};
\node at (240:4.4) {$\bar 3$};
\node at ($(120:2)+(0.3,0)$) {$a$};
\node at ($(0:2)+(0,0.3)$) {$b$};
\node at ($(240:2)+(0.3,0)$) {$c$};
\node at (0,-4.5) {$\Gamma$};
\begin{scope}[shift={(10,0)}]
\draw (0,0) circle (4cm);
\draw (0:4) -- (120:4) -- (240:4) -- (0:4);
\node at (120:4.3) {$\bar 1$};
\node at (0:4.3) {$\bar 2$};
\node at (240:4.4) {$\bar 3$};
\node at ($(240:4.4)!0.5!(0:4.4)$) {$A$};
\node at ($(240:4.6)!0.5!(120:4.6)$) {$B$};
\node at ($(120:4.6)!0.5!(0:4.6)$) {$C$};
\node at (0,-4.5) {$\Gamma'$};
  \end{scope}
\end{tikzpicture}
\end{center}
Let $L = L(\Gamma)$ where $\Gamma$ is the star, and $L' = L(\Gamma')$ where $\Gamma'$ is the triangle, as illustrated.  We have 
\begin{align*}
L_{\bar 1|\bar 2|\bar 3|} &= a + b + c, \qquad 
L_{\bar 1\bar 2|\bar 3} = a b, \qquad
L_{\bar 1|\bar 2\bar 3} = bc, \qquad
L_{\bar 1\bar 3|\bar 2} = ac, \qquad
L_{\bar 1 \bar 2 \bar 3} = abc 
\end{align*}
and
\begin{align*}
L'_{\bar 1|\bar 2|\bar 3|} &= 1, \qquad
L'_{\bar 1\bar 2|\bar 3} = C ,\qquad
L'_{\bar 1|\bar 2\bar 3} = A, \qquad
L'_{\bar 1\bar 3|\bar 2} = B, \qquad
L'_{\bar 1 \bar 2 \bar 3} = AB + BC + AC.
\end{align*}
If $a,b,c$ and $A,B,C$ are related as in Theorem \ref{thm:YDelta}, then $\L(\Gamma) = \L(\Gamma')$ in $\P^{\NC_n}$.
\end{example}

\begin{remark}
For a circular planar electrical network $\Gamma$, Proposition \ref{prop:Lequiv} follows immediately from the following theorem of Kenyon and Wilson \cite{KW}: for any planar partition $\sigma$, the ratio $L_\sigma/L_{\uncrossed}$ is an integer coefficient polynomial in the $\Lambda_{i,j}$ of degree equal to $n - \#$ parts of $\sigma$.\end{remark}

\subsection{The compactified space of circular planar electrical networks}
\label{sec:compact}
Let $E'_n$ denote the space of electrical networks modulo electrical equivalence, or equivalently, the space of response matrices characterized in Theorem \ref{thm:CIM}.

\begin{lemma}
The map $\Gamma \to \L(\Gamma)$ descends to an injection $E'_n \hookrightarrow \P^{\NC_n}$.
\end{lemma}
\begin{proof}
Proposition \ref{prop:Lequiv} says that $\L(\Gamma)$ is invariant under electrical equivalence.   For $\Gamma$ a planar electrical network, we have that $L_{\uncrossed}$ is always non-zero, so the entries of the response matrix $\Lambda(\Gamma)$ can be computed using Proposition \ref{prop:Lijresponse}.
\end{proof}

\begin{definition}
The \defn{compactified space of circular planar electrical networks} $E_n:= \overline{E'_n} \subset \P^{\NC_n}$ is defined to be the closure of $E'_n$ (in the Hausdorff topology).
\end{definition}

\begin{theorem}\label{thm:cactus}
The space $E_n$ is exactly the set of grove measurements of cactus networks.  A cactus network $\Gamma$ is determined, up to electrical equivalence, by $\L(\Gamma) \in E_n$.
\end{theorem}
\begin{proof}
By definition any point $\L$ in $E_n$ is the limit of points in $E'_n$ which are representable by usual circular planar networks.  Since the top cell of $E'_n$ is dense in $E_n$, using Theorem \ref{thm:CIM}(5), we can assume that $\L$ is the limit $\lim_{i \to \infty} \L(\Gamma_i)$ of circular planar networks $\Gamma_i$, all with the same underlying graph.  Since the edge weights of such graphs depend continuously on the point $\L(\Gamma_i)$, the point $\L$ can be obtained from a usual circular planar network by sending some of the edge weights to $\infty$.  Such a limit is just a cactus network $\Gamma$.  So $\L = \L(\Gamma)$.

Now let $\Gamma$ be a cactus network, and let $\Gamma = \bigcup_r \Gamma^{(r)}$ be the decomposition of $\Gamma$ into a union of circular planar electrical networks, each embedded into a disk.  It is clear that the shape $\sigma(\Gamma)$ is determined by $\L(\Gamma)$.  To see that $\L(\Gamma)$ uniquely determines $\Gamma$, it suffices to recover the response matrices of each $\Gamma^{(r)}$.  Suppose $i,j$ belong to the same disk of the cactus, so that $i$, $j$ are distinct vertices of $\Gamma^{(r)}$.  Let $\sigma_{ij}$ be obtained from $\sigma$ by gluing the parts containing $i$ and $j$ together.  Then by Proposition \ref{prop:Lijresponse} we have
$$
\Lambda(\Gamma^{(r)})_{ij} = -\dfrac{L_{\sigma_{ij}}(\Gamma)}{L_\sigma(\Gamma)}.
$$
We have used that a grove for $\Gamma$ is just a union of groves for each $\Gamma^{(r')}$: in the above ratio, the contribution of groves from components $\Gamma^{(r')}$ for $r' \neq r$ cancel out.  So the response matrix of each $\Gamma^{(r)}$ can be recovered from $\L(\Gamma)$, and hence $\Gamma$ is determined by $\L(\Gamma)$ up to electrical equivalence.
\end{proof}

Define
$$
E_\tau := \{\L(\Gamma) \mid \tau(\Gamma) = \tau\} \subset E
$$
to be those points representable by critical cactus networks with medial pairing $\tau$. 

\begin{proposition}\label{prop:cactusparam}
Each stratum $E_\tau$ is parametrized by choosing a criticial cactus network $\Gamma$ with $\tau(\Gamma) = \tau$, and letting the edge weights vary, so that we have $E_\tau \simeq \R_{>0}^{c(\tau)}$.  Furthermore,
$$
E = \bigsqcup_{\tau \in P_n} E_\tau.
$$
\end{proposition}
\begin{proof}
The first statement just follows from applying Theorem \ref{thm:CIM} to each circular planar electrical network $\Gamma^{(r)}$ in the decomposition $\Gamma = \bigcup_r \Gamma^{(r)}$ in the proof of Theorem \ref{thm:cactus}.  The second statement follows from Theorem \ref{thm:cactus} and Proposition \ref{prop:cactus}.
\end{proof}

Write $\E(\tau)$ for the electroid of any $\L \in E_\tau$.  This does not depend on the choice of $\L$, by Proposition \ref{prop:Lequiv}.  In this paper we will focus on the topological spaces $E_\tau$.  In the future we hope to consider the algebraic geometry of their Zariski closures in $\P^{\NC_n}$.

\subsection{The bottom cells of $E_n$}
\label{sec:bottom}
We have defined a decomposition of $E_n$ into cells.  Let $P_n$ be the set of medial pairings, or matchings, on $[2n]$.   For a medial pairing $\tau$, let $E_\tau \subset E_n$ be the corresponding electrical cell, so that $E_n = \bigcup_{\tau \in P_n} E_\tau$.  
There is a unique top cell $E_{\tau_{\rm top}}$, where $\tau_{\rm top}$ is given by the involution $\tau_{\rm top}(i) = i + n \mod 2n$.  There are Catalan number of 0-dimensional cells, corresponding to medial pairings $\tau$ that are non-crossing matchings.  

Let $p_\sigma$ denote the point in $E_n$ with 
$$
L_{\sigma'}(p_\sigma) = \begin{cases} 1 & \mbox{if $\sigma' = \sigma$} \\
0 & \mbox{otherwise.} \end{cases}
$$
(Recall that the coordinates $L_\sigma$ are projective coordinates, so the value $1$ is not important.)  

\begin{proposition}
The points $p_\sigma$ are exactly the $0$-dimensional cells of $E_n$.  The point $p_\sigma$ is the grove measurement of the cactus network with boundary vertices identified according to $\sigma$.
\end{proposition}

\subsection{Uncrossing partial order on matchings}
\label{sec:poset}
Define a partial order on $P_n$ as follows.  Let $\tau$ be a medial pairing and take any lensless medial graph $G$ representing $\tau$.  Now uncross any crossing in $G$ in either of two ways:  

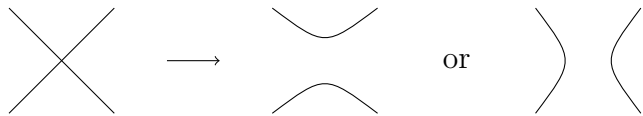
\begin{center}
\begin{tikzpicture}[scale=0.7]
\node at (7.5,0) {or};
\draw (-1,-1) -- (1,1);
\draw (-1,1) -- (1,-1);

\draw[->] (2,0) -- (3,0);

\draw (4,-1) .. controls (5,-0.25) .. (6,-1);
\draw (4,1) .. controls (5,0.25) .. (6,1);
\draw (9,-1) .. controls (9.75,0) .. (9,1);
\draw (11,-1) .. controls (10.25,0) .. (11,1);
\end{tikzpicture}
\end{center}

This gives a new medial graph $G'$.  Suppose $G'$ is also lensless.  Then we declare that $\tau(G') \lessdot \tau(G)$ is a cover relation in $P_n$.  The partial order $P_n$ is the transitive closure of these relations.    This partial order was studied by Alman, Lian, and Tran \cite{ALT}, by Kenyon \cite{Ken}, and also by Huang, Wen and Xie \cite{HWX}.

\begin{lemma}\label{lem:covers}
Let $G$ be a medial graph with $\tau(G) = \tau$.  Suppose $(a,b,c,d)$ are in cyclic order, and $\tau$ has strands from $a$ to $c$ and from $b$ to $d$.  Let $G'$ be obtained by uncrossing the intersection point of these strands so that in $G'$ we have that $a$ is joined to $d$ and $b$ is joined to $c$.  Then $G'$ is lensless if and only if no other medial strand goes from the arc $(a,b)$ to the arc $(c,d)$.
%
\end{lemma}

In particular, the definition of the set of covering relations $\tau' \lessdot \tau$ involving a fixed $\tau$ can be defined using any lensless medial graph $G$ with $\tau(G) = \tau$.

As before, the \defn{crossing number} $c(\tau)$ of a medial pairing $\tau$ of a cactus network is defined to be the number of crossings in a reduced/lensless representative medial graph.  

\begin{lemma}
$P_n$ is a graded poset with grading given by $c(\tau)$.
\end{lemma}

We shall show later that $P_n$ is the closure partial order on the stratification $\{E_\tau \mid \tau \in P_n\}$ of $E_n$.  Here is a picture of $P_3$:
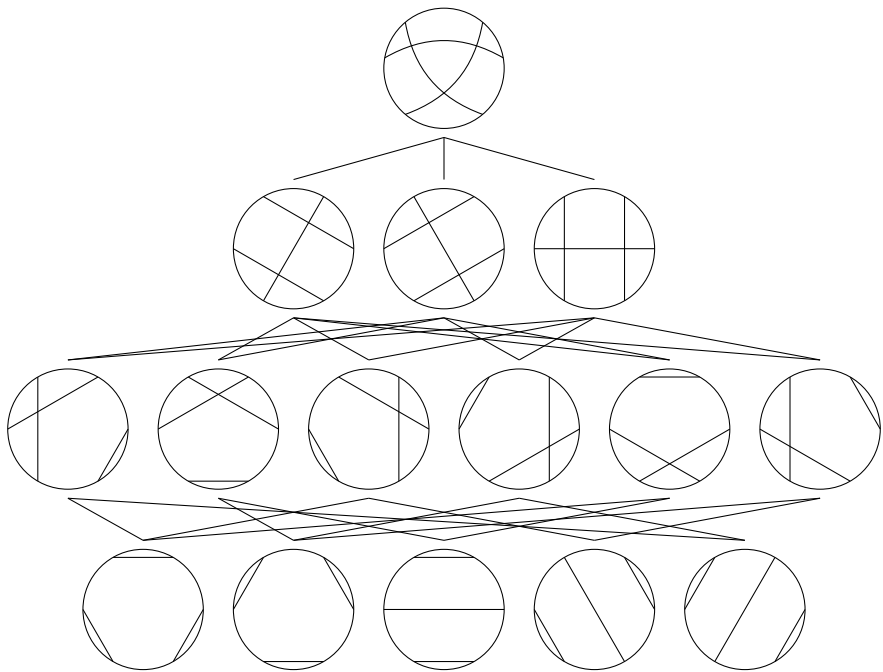
\begin{center}
\begin{tikzpicture}[scale = 0.4]
\draw (0,0) circle (2cm);
\coordinate (a1) at (130:2);
\coordinate (a2) at (50:2);
\coordinate (a3) at (10:2);
\coordinate (a4) at (-50:2);
\coordinate (a5) at (-130:2);
\coordinate (a6) at (170:2);
\draw (a1) to [bend right] (a4);
\draw (a2) to [bend left] (a5);
\draw (a3) to [bend right] (a6);

\draw (0,-2.3) -- (0,-3.7);
\draw (0,-2.3) -- (5,-3.7);
\draw (0,-2.3) -- (-5,-3.7);
\begin{scope}[shift = {(0,-6)}]
\draw (0,0) circle (2cm);
\coordinate (a1) at (120:2);
\coordinate (a2) at (60:2);
\coordinate (a3) at (0:2);
\coordinate (a4) at (-60:2);
\coordinate (a5) at (-120:2);
\coordinate (a6) at (180:2);
\draw (a1) -- (a4);
\draw (a2) -- (a6);
\draw (a3) -- (a5);
\draw (0,-2.3) -- (-17.5+1*5,-6+2.3);
\draw (0,-2.3) -- (-17.5+4*5,-6+2.3);
\draw (0,-2.3) -- (-17.5+5*5,-6+2.3);
\draw (0,-2.3) -- (-17.5+2*5,-6+2.3);
\end{scope}

\begin{scope}[shift = {(-5,-6)}]
\draw (0,0) circle (2cm);
\coordinate (a1) at (120:2);
\coordinate (a2) at (60:2);
\coordinate (a3) at (0:2);
\coordinate (a4) at (-60:2);
\coordinate (a5) at (-120:2);
\coordinate (a6) at (180:2);
\draw (a1) -- (a3);
\draw (a2) -- (a5);
\draw (a4) -- (a6);
\end{scope}
\begin{scope}[shift={(0,-6)}]
\draw (-5,-2.3) -- (-17.5+2*5,-6+2.3);
\draw (-5,-2.3) -- (-17.5+3*5,-6+2.3);
\draw (-5,-2.3) -- (-17.5+5*5,-6+2.3);
\draw (-5,-2.3) -- (-17.5+6*5,-6+2.3);
\end{scope}

\begin{scope}[shift = {(5,-6)}]
\draw (0,0) circle (2cm);
\coordinate (a1) at (120:2);
\coordinate (a2) at (60:2);
\coordinate (a3) at (0:2);
\coordinate (a4) at (-60:2);
\coordinate (a5) at (-120:2);
\coordinate (a6) at (180:2);
\draw (a2) -- (a4);
\draw (a3) -- (a6);
\draw (a1) -- (a5);
\end{scope}
\begin{scope}[shift={(0,-6)}]
\draw (5,-2.3) -- (-17.5+1*5,-6+2.3);
\draw (5,-2.3) -- (-17.5+3*5,-6+2.3);
\draw (5,-2.3) -- (-17.5+4*5,-6+2.3);
\draw (5,-2.3) -- (-17.5+6*5,-6+2.3);
\end{scope}

\begin{scope}[shift = {(-12.5,-12)}]
\draw (0,0) circle (2cm);
\coordinate (a1) at (120:2);
\coordinate (a2) at (60:2);
\coordinate (a3) at (0:2);
\coordinate (a4) at (-60:2);
\coordinate (a5) at (-120:2);
\coordinate (a6) at (180:2);
\draw (a3) -- (a4);
\draw (a5) -- (a1);
\draw (a6) -- (a2);
\end{scope}
\begin{scope}[shift={(0,-12)}]
\draw (-12.5,-2.3) -- (-15+1*5,-6+2.3);
\draw (-12.5,-2.3) -- (-15+5*5,-6+2.3);
\end{scope}

\begin{scope}[shift = {(-7.5,-12)}]
\draw (0,0) circle (2cm);
\coordinate (a1) at (120:2);
\coordinate (a2) at (60:2);
\coordinate (a3) at (0:2);
\coordinate (a4) at (-60:2);
\coordinate (a5) at (-120:2);
\coordinate (a6) at (180:2);
\draw (a4) -- (a5);
\draw (a6) -- (a2);
\draw (a3) -- (a1);
\end{scope}
\begin{scope}[shift={(0,-12)}]
\draw (-7.5,-2.3) -- (-15+2*5,-6+2.3);
\draw (-7.5,-2.3) -- (-15+3*5,-6+2.3);
\end{scope}

\begin{scope}[shift = {(-2.5,-12)}]
\draw (0,0) circle (2cm);
\coordinate (a1) at (120:2);
\coordinate (a2) at (60:2);
\coordinate (a3) at (0:2);
\coordinate (a4) at (-60:2);
\coordinate (a5) at (-120:2);
\coordinate (a6) at (180:2);
\draw (a5) -- (a6);
\draw (a1) -- (a3);
\draw (a2) -- (a4);
\end{scope}
\begin{scope}[shift={(0,-12)}]
\draw (-2.5,-2.3) -- (-15+1*5,-6+2.3);
\draw (-2.5,-2.3) -- (-15+4*5,-6+2.3);
\end{scope}

\begin{scope}[shift = {(2.5,-12)}]
\draw (0,0) circle (2cm);
\coordinate (a1) at (120:2);
\coordinate (a2) at (60:2);
\coordinate (a3) at (0:2);
\coordinate (a4) at (-60:2);
\coordinate (a5) at (-120:2);
\coordinate (a6) at (180:2);
\draw (a1) -- (a6);
\draw (a2) -- (a4);
\draw (a3) -- (a5);
\end{scope}
\begin{scope}[shift={(0,-12)}]
\draw (2.5,-2.3) -- (-15+2*5,-6+2.3);
\draw (2.5,-2.3) -- (-15+5*5,-6+2.3);
\end{scope}

\begin{scope}[shift = {(7.5,-12)}]
\draw (0,0) circle (2cm);
\coordinate (a1) at (120:2);
\coordinate (a2) at (60:2);
\coordinate (a3) at (0:2);
\coordinate (a4) at (-60:2);
\coordinate (a5) at (-120:2);
\coordinate (a6) at (180:2);
\draw (a1) -- (a2);
\draw (a3) -- (a5);
\draw (a4) -- (a6);
\end{scope}
\begin{scope}[shift={(0,-12)}]
\draw (7.5,-2.3) -- (-15+1*5,-6+2.3);
\draw (7.5,-2.3) -- (-15+3*5,-6+2.3);
\end{scope}

\begin{scope}[shift = {(12.5,-12)}]
\draw (0,0) circle (2cm);
\coordinate (a1) at (120:2);
\coordinate (a2) at (60:2);
\coordinate (a3) at (0:2);
\coordinate (a4) at (-60:2);
\coordinate (a5) at (-120:2);
\coordinate (a6) at (180:2);
\draw (a2) -- (a3);
\draw (a4) -- (a6);
\draw (a1) -- (a5);
\end{scope}
\begin{scope}[shift={(0,-12)}]
\draw (12.5,-2.3) -- (-15+2*5,-6+2.3);
\draw (12.5,-2.3) -- (-15+4*5,-6+2.3);
\end{scope}

\begin{scope}[shift = {(-10,-18)}]
\draw (0,0) circle (2cm);
\coordinate (a1) at (120:2);
\coordinate (a2) at (60:2);
\coordinate (a3) at (0:2);
\coordinate (a4) at (-60:2);
\coordinate (a5) at (-120:2);
\coordinate (a6) at (180:2);
\draw (a1) -- (a2);
\draw (a3) -- (a4);
\draw (a5) -- (a6);
\end{scope}
\begin{scope}[shift = {(-5,-18)}]
\draw (0,0) circle (2cm);
\coordinate (a1) at (120:2);
\coordinate (a2) at (60:2);
\coordinate (a3) at (0:2);
\coordinate (a4) at (-60:2);
\coordinate (a5) at (-120:2);
\coordinate (a6) at (180:2);
\draw (a2) -- (a3);
\draw (a4) -- (a5);
\draw (a1) -- (a6);
\end{scope}
\begin{scope}[shift = {(0,-18)}]
\draw (0,0) circle (2cm);
\coordinate (a1) at (120:2);
\coordinate (a2) at (60:2);
\coordinate (a3) at (0:2);
\coordinate (a4) at (-60:2);
\coordinate (a5) at (-120:2);
\coordinate (a6) at (180:2);
\draw (a1) -- (a2);
\draw (a3) -- (a6);
\draw (a4) -- (a5);
\end{scope}
\begin{scope}[shift = {(5,-18)}]
\draw (0,0) circle (2cm);
\coordinate (a1) at (120:2);
\coordinate (a2) at (60:2);
\coordinate (a3) at (0:2);
\coordinate (a4) at (-60:2);
\coordinate (a5) at (-120:2);
\coordinate (a6) at (180:2);
\draw (a2) -- (a3);
\draw (a4) -- (a1);
\draw (a6) -- (a5);
\end{scope}
\begin{scope}[shift = {(10,-18)}]
\draw (0,0) circle (2cm);
\coordinate (a1) at (120:2);
\coordinate (a2) at (60:2);
\coordinate (a3) at (0:2);
\coordinate (a4) at (-60:2);
\coordinate (a5) at (-120:2);
\coordinate (a6) at (180:2);
\draw (a3) -- (a4);
\draw (a2) -- (a5);
\draw (a1) -- (a6);
\end{scope}

\end{tikzpicture}
\end{center}

\subsection{Matching partial order and Bruhat order}
\label{sec:elecaffine}
The partial order $P_n$ on matchings is induced from the partial order of the affine symmetric group, as we now explain.  To a medial pairing $\tau$, we associate a bounded affine permutation $g_\tau$ by 
$$
g_\tau(i) := \begin{cases} \tau(i) & \mbox{if $i < \tau(i)$} \\
\tau(i) + 2n & \mbox{if $i > \tau(i)$}
\end{cases}
$$ where $\tau$ is thought of as a fixed-point free involution on $[2n]$.  Note that $g_\tau$ is a bounded affine permutation of type $(n,2n)$.  We have $g_0:=g_{\tau_\top}$ is given by $g_0(i) = i+n$, which has length $0$.  The bounded affine permutation $g_0$ plays the role of the identity permutation.

Define the (infinite by infinite) \defn{affine rank matrix} of an affine permutation $f$ by
$$
r(i,j) := |\{ a \leq i \mid f(a) \geq j \}|.
$$
This matrix satisfies the periodicity $r(i+2n,j+2n) = r(i,j)$.

\begin{theorem}[{\cite[Theorem 8.3.7]{BB}}] \label{thm:rank}
We have $f \leq f'$ in Bruhat order if and only if $r_f(i,j) \leq r_{f'}(i,j)$ for all $i,j \in \Z$.
\end{theorem}

Let $\tS_{2n}^0$ denote the affine permutations $f: \Z \to \Z$ satisfying $\sum_{i=1}^{2n} (f(i)-i) = 0$.  Let $t_{a,b} \in \tS_{2n}^0$ be the transposition swapping $a$ and $b$, and $s_i = t_{i,i+1}$.  We note that $s_i g_0 = g_0 s_{i+n}$.

\begin{lemma}\label{lem:wgw}
Let $\tau \in P_n$.  Then there exists $w \in \tS_{2n}^0$ such that
\begin{equation}\label{eq:inv}
g_\tau = w g_0 w^{-1}
\end{equation}
where $\ell(w) = \binom{n}{2}- c(\tau)$, and $\ell(g_\tau) = 2 \ell(w)$.
\end{lemma}
\begin{proof}
The claims are trivially true when $\tau=\tau_\top$ and $g_\tau = g_0$.  Suppose $\tau \in P_n$ is not the top element.  Then there exists some $i$ such that $g_\tau(i) > g_\tau(i+1)$.  Let $\tau'$ be obtained from $\tau$ by swapping $i, i+1$ and $g_\tau(i), g_\tau(i+1)$ (all taken modulo $2n$).  It is clear that $\tau' \gtrdot \tau$.  But we have
$$
g_{\tau} = s_i g_{\tau'} s_i
$$
and $\ell(g_{\tau'}) = \ell(g_{\tau}) + 2$, and the claim follows by induction.
\end{proof}

Note that the factorization in Lemma \ref{lem:wgw} is not unique.  For example, if $n = 3$, then $s_1 g_0 s_1 = s_3 g_0 s_3$, so there are multiple choices for $w$.

\begin{theorem}\label{thm:poset}
We have $\ell(g_\tau) = 2(\binom{n}{2} - c(\tau))$.  The map $\tau \mapsto g_\tau$ identifies $P_n$ with an induced subposet of the dual Bruhat order of bounded affine permutations.  In other words, $g_\tau \leq g_{\tau'}$ in Bruhat order if and only if $\tau' \leq \tau$.
\end{theorem}

%

\begin{proof}
The first claim follows from Lemma \ref{lem:wgw}.

Suppose $\tau' \lessdot \tau$.  By Lemma \ref{lem:covers}, $\tau'$ is obtained from $\tau$ by uncrossing the intersection point of strands $a \leftrightarrow c$ and $b \leftrightarrow d$ so that after uncrossing $a$ is joined to $d$ and $b$ is joined to $c$, where $(a,b,c,d)$ are in cyclic order, and no other medial strand goes from the arc $(a,b)$ to the arc $(c,d)$.  For simplicity, we suppose that $a < b < c < d$.  Then $g_{\tau'} = t_{a,b} t_{c,d}g_{\tau} = t_{a,b} g_\tau t_{a,b}$, and it is easy to see that $g_{\tau'} > g_\tau$.  

Now suppose that $g_{\tau'} > g_\tau$.  We know that there exists $a < b$ such that $g_{\tau'} > t_{a,b} g_\tau \gtrdot g_\tau$.  Let $c:= g_\tau(a)$ and $d:= g_\tau(b)$.  It is clear that we also have $t_{a,b} g_\tau t_{a,b} \gtrdot t_{c,d} g_\tau = g_\tau t_{a,b} \gtrdot g_\tau$.  We claim that $g_{\tau'} > g_\tau t_{a,b}$ as well.  To see this use Lemma \ref{lem:wgw} to write $g_\tau = wg_0w^{-1}$ and $g_{\tau'} = v g_0 v^{-1}$.  Define the group isomorphism $\iota: \tS_n^0 \to \tS_n^0$ by $s_i \mapsto s_{i+n}$ for all $i$.  Then $g_{\tau'} > t_{a,b} g_\tau$ implies $v \iota(v^{-1}) > t_{a,b} w \iota(w^{-1})$.  Taking inverses we get $\iota(v) v^{-1} > \iota(w) w^{-1} t_{a,b}$, and left multiplying by $g_0$, we get $g_{\tau'} > g_\tau t_{a,b}$.

So $g_{\tau'}$ is greater than both $g_\tau t_{a,b}$ and $t_{a,b} g_\tau$.  We now show that $g_{\tau'}$ greater than $t_{a,b} g_\tau t_{a,b}$.  We have
\begin{equation}\label{eq:order}
a < b < c < d < a+2n.
\end{equation}
Let $R$ be the rectangular region with corners at $(a,c), (a,d), (b,c), (b,d)$.  Similarly, let $R'$ be the rectangular region with corners at $(c,a+2n), (c,b+2n), (d,a+2n), (d,b+2n)$.  Then the affine rank matrices of $g_\tau t_{a,b}$ (resp. $t_{a,b} g_\tau$) and $g_\tau$ differ only in $R$ (resp. $R'$) and the periodic shifts of $R$ (resp. $R'$).  The inequalities \eqref{eq:order} imply that the periodic shifts of $R$ and the periodic shifts of $R'$ never intersect.  Applying Theorem \ref{thm:rank}, we see that $$r_{g_{\tau'}}(i,j) \geq \max\left(r_{g_\tau t_{a,b}}(i,j), r_{t_{a,b}g_\tau}(i,j)\right) = r_{t_{a,b}g_\tau t_{a,b}}(i,j)
$$ 
for each $i,j \in \Z$, and thus $g_{\tau'} > t_{a,b}g_\tau t_{a,b}$.

But $t_{a,b} g_\tau t_{a,b} = g_{\tau''}$ for some $\tau'' \lessdot \tau$.  By induction on $\ell(g_{\tau'}) - \ell (g_\tau)$ we conclude that $\tau' < \tau$.
\end{proof}

Let us also define a bounded affine permutation $f_\tau$ by $f_\tau(i) = g_\tau(i) - 1$.  Since $i < g_\tau(i) < i + 2n$, we have $i \leq f_\tau (i) \leq i+2n-2$.  Note that $f_\tau$ is of type $(n-1,2n)$.  Define an \defn{electrical affine permutation} to be a bounded affine permutation $f$ with period $2n$ satisfying: 
\begin{enumerate}
\item 
$i \leq f(i) \leq i+2n-2$
\item
if $j = f(i)$ then $f(j+1) \equiv (i-1) \mod 2n$.  
\end{enumerate}
Note that (2) and boundedness determines $f(j+1)$ unless $i-1=j+1 + 2n$, in which case (1) forces $f(j+1) = j+1$.  Denote the set of electrical affine permutations by $\Elec(n)$.  
We have
\begin{lemma}
The set of electrical affine permutations is exactly the set $\{f_\tau \mid \tau \in P_n\}$.
\end{lemma}

We also have the following corollary of Theorem \ref{thm:poset}.

\begin{corollary}\label{cor:order}\
The following are equivalent:
\begin{enumerate}
\item
$\tau' \leq \tau$
\item
$f_\tau \leq f_{\tau'}$ in Bruhat order
\item
$\I(f_\tau) \leq \I(f_{\tau'})$
\end{enumerate}
\end{corollary}
\begin{proof}
Let $\tau \in P_n$.  And let $\I(g) = (J_1,J_2,\ldots,J_{2n})$ be the Grassmann necklace of $g_\tau$.  Then $a \in J_a$ for each $a$.  Let $\I(f_\tau) = (I_1,I_2,\ldots,I_{2n})$.  We have that $I_a = \{b-1 \mid b \in J_a \setminus \{a\}\}$.  It follows that $\I(f_\tau) \leq \I(f_{\tau'})$ if and only if $\I(g_\tau) \leq \I(g_{\tau'})$.  The claim then follows from Theorem \ref{thm:Grassorder} and Theorem \ref{thm:poset}.
\end{proof}

Corollary \ref{cor:order} gives a new non-recursive criterion for the partial order $P_n$, even when we consider only matchings that arise from circular planar electrical networks (instead of the more general cactus networks).

\subsection{Catalan subsets and Catalan necklaces}\label{sec:Catalan}
The Grassmann necklace $\I(g_\tau) = (J_1,\ldots,J_{2n})$ can be read off from $\tau$ as follows.  Draw $\tau$ as a medial pairing in the disk.  For each $a \in [2n]$, let us traverse the circle clockwise starting at the boundary vertex $a$.  At each boundary vertex, write ``$U$" if the vertex is at the end of a strand we have not previously encountered, and write ``$D$" if the vertex is at the end of a strand we have previously encountered.  The boundary vertices marked with $U$ gives the set $J_a$.  To get $I_a(f_\tau)$, we remove $a$ (which is always present in $J_a$) and shift everything by 1.

Let $P$ be a Dyck path of length $2n$, thought of as a sequence of U-s and D-s.  A sequence of $n$ U-s and $n$ D-s is a Dyck path if it satisfies the property that in any initial subsequence there is at least as many U-s as there are D-s.  Note that a Dyck path must start with a U.

A \defn{Catalan subset} $I$ is an $(n-1)$-element subset of $[2n]$ such that
$$
\{1\} \cup \{a+1 \mid a \in I\} = \{\text{positions of up steps in } P(I)\}
$$
for some Dyck path $P(I)$.  More generally, we call $I$ a Catalan subset with respect to $\leq_a$ if the subset $I-a \mod 2n$ is a Catalan subset.

\begin{lemma}
Let $I, J$ be two Catalan subsets.  Then $I \leq J$ in dominance order if and only if the Dyck path $P(I)$ never goes below $P(J)$. 
\end{lemma}

The Catalan subset $\{1,2,4,5,9\}$ corresponds to the UD-sequence $UUUDUUDDDUDD$ and the following Dyck path:
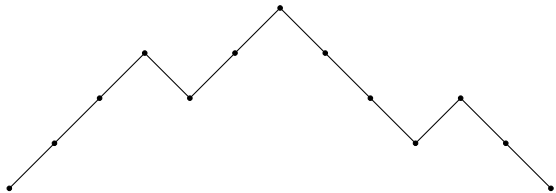
\begin{center}
\begin{tikzpicture}[scale=0.6]
\draw (0,0) -- (1,1) -- (2,2) -- (3,3) -- (4,2) --(5,3) -- (6,4) -- (7,3) -- (8,2) -- (9,1) -- (10,2) -- (11,1) -- (12,0);
\filldraw[black] (0,0) circle (0.05cm);
\filldraw[black] (1,1) circle (0.05cm);
\filldraw[black] (2,2) circle (0.05cm);
\filldraw[black] (3,3) circle (0.05cm);
\filldraw[black] (4,2) circle (0.05cm);
\filldraw[black] (5,3) circle (0.05cm);
\filldraw[black] (6,4) circle (0.05cm);
\filldraw[black] (7,3) circle (0.05cm);
\filldraw[black] (8,2) circle (0.05cm);
\filldraw[black] (9,1) circle (0.05cm);
\filldraw[black] (10,2) circle (0.05cm);
\filldraw[black] (11,1) circle (0.05cm);
\filldraw[black] (12,0) circle (0.05cm);
\end{tikzpicture}
\end{center}

Define $\I(\tau) = (I_1,I_2,\ldots,I_{2n})$ by $\I(\tau) = \I(f_\tau)$.   The previous discussion gives:

\begin{lemma}
Let $\tau \in P_n$.  For each $a \in [2n]$, the set $I_a(\tau)$ is a Catalan subset with respect to $\leq_a$.
\end{lemma}

A Grassmann necklace of type $(n-1,2n)$ is a \defn{Catalan necklace} if each $I_a$ subset is an $a$-shifted Catalan subset.  Thus $\I(\tau)$ is a Catalan necklace for $\tau \in P_n$.  We shall characterize Catalan necklaces in terms of non-crossing partitions in Proposition \ref{prop:Catalan}.

\begin{example}
Let $n = 5$ and $\tau = \{(1,7),(2,9),(3,8),(4,10),(5,6)\} \in P_5$.  Then $g_\tau = [7,9,8,10,6,15,11,13,12,14]$ and $f_\tau = [6,8,7,9,5,14,10,12,11,13]$.  We have $\I(\tau) = (1234,2346,3468,4678,6789,6789,7894,89(10)4,9(10)24,(10)124)$.
\end{example}

For a non-crossing partition $\sigma$, we shall also write $\I(\sigma)$ for $\I(\tau(\sigma))$.  The Grassmann necklaces $\I(\sigma)$ have the property that not only does it consist of Catalan subsets, but each subset $I_a(\sigma)$ determines the whole $\I(\sigma)$.    The subset $I_a(\sigma)$ has the following explicit description.  Suppose $\sigma = (\sigma_1 | \sigma_2 | \cdots | \sigma_r)$ so that $\sigma_i$ are the parts of $\sigma$, and similarly let $\tsigma=(\sigma_1 | \sigma_2 | \cdots | \sigma_{\tilde r})$.  Then
\begin{equation}\label{eq:Ia}
[2n] \setminus I_a(\sigma) = \{\max_{\leq_a} \sigma_1,  \max_{\leq_a} \sigma_2, \ldots, \max_{\leq_a} \sigma_r \} \bigcup \{\max_{\leq_a} \tsigma_1,  \max_{\leq_a} \sigma_2, \ldots, \max_{\leq_a} \tsigma_{\tilde r} \}.
\end{equation}
Note that in this equation we have identified $[\bar n] \cup [\tilde n]$ with $[2n]$.  Thus the order $\leq_a$ on $[2n]$ induces an order on $[\bar n]$ and an order on $[\tilde n]$ and $\max_{\leq a}$ is the maximum with respect to this order.  Also note that the right hand side has cardinality $n+1$ by Lemma \ref{lem:dualpart}.  The following result is straightforward.

\begin{lemma}
For each $a \in [2n]$, the map $\sigma \mapsto I_a(\sigma)$ is a bijection between $\NC_n$ and Catalan subsets.
\end{lemma}

The $a$-shifted dominance order on subsets can be transferred to non-crossing partitions via the bijection $\sigma \mapsto I_a(\sigma)$: we define $\sigma \leq_a \sigma'$ if and only if $\I_a(\sigma) \leq_a I_a(\sigma')$.  Similarly there is an $a$-shifted lexicographic ordering on $\NC_n$.   

\section{Electroid varieties}
In this section we define and study an embedding of $E_n$ into the totally nonnegative Grassmannian.  The section is organized as follows. In Section \ref{sec:Temperley} we explain the construction of a bipartite graph $N(\Gamma)$ from an electrical network $\Gamma$.  In Section \ref{sec:concordant}, we define a linear slice $\X$ of the totally nonnegative Grassmannian and state our main theorems.  Theorem \ref{thm:stratification} states that the only nonempty intersections of $\X$ with the positroid stratification are with the strata labeled by electrical affine permutations.  Theorem \ref{thm:realizability} states that $E_n$ can be identified with the totally nonnegative part $\X_{\geq 0}$ of $\X$, compatibly with the stratifications of all spaces.  In Section \ref{sec:grovematch} we give linear relations between grove coordinates on $E_n$ and boundary measurements on $\Gr(n-1,2n)_{\geq 0}$.  In Section \ref{sec:generators}, we discuss the action of the generators of the electrical Lie group on the Grassmannian.  In Section \ref{sec:strata}, we compare the stratifications of $E_n$ and $\X_{\geq 0}$.  In Section \ref{sec:closure} we discuss the closure partial order on electroid strata.  Sections \ref{sec:reductions} and \ref{sec:stratification} contain the proof of Theorem \ref{thm:stratification}.  Sections \ref{sec:ncorder} and \ref{sec:realizability} contain the rest of the proof of Theorem \ref{thm:realizability}.  In Section \ref{sec:partnecklace}, we give further characterizations of electroids, and define non-crossing partition necklaces.  In Section \ref{sec:quadratic} we discuss quadratic relations for grove coordinates.

\subsection{From electrical networks to bipartite graphs}\label{sec:Temperley}
We produce a planar bipartite network $N = N(\Gamma)$, embedded into the disk, for each electrical network $N$.  Our construction is a modified version (to take into account boundary vertices) of the \defn{generalized Temperley's trick}; see \cite{GK, KPW}.  If $\Gamma$ has boundary vertices $\bar 1, \bar 2,\ldots, \bar n$, then $N$ will have boundary vertices $1,2,\ldots,2n$, where boundary vertex $\bar i$ is identified with $2i-1$, and a boundary vertex $2i$ in $N$ lies between $\bar i$ and $\overline{i+1}$.  The boundary vertex $2i$ can be identified with the vertex $\tilde i$ used to label dual non-crossing partitions.  The planar bipartite network $N$ always has boundary vertices of degree $1$.

The interior vertices of $N$ are as follows: we have a black interior vertex $b_v$ for each interior vertex $v$ of $\Gamma$, and a black interior vertex $b_F$ for each interior face $F$ of $\Gamma$; we have a white interior vertex $w_e$ placed at the midpoint of each interior edge $e$ of $\Gamma$.  For each vertex $\bar i$, we also make a black interior vertex $b_i$.  The edges of $N$ are as follows: (1) if $v$ is a vertex of an edge $e$ in $\Gamma$, then $b_v$ and $w_e$ are joined, and the weight of this edge is equal to the weight $w(e)$ of $e$ in $\Gamma$, (2) if $e$ borders $F$, then $w_e$ is joined to $b_F$ by an edge with weight 1, (3) the vertex $b_i$ is joined (by an edge with weight $1$) to the boundary vertex $2i-1$ in $N$, and $b_i$ is also joined (by an edge with weight $1$) to $w_e$ for any edge $e$ incident to $\bar i$ in $\Gamma$, and (4) even boundary vertices $2i$ in $N$ are joined (by an edge with weight $1$) to the face vertex $w_F$ of the face $F$ that they lie in.  

The construction is extended to cactus networks as follows: if boundary vertices $\overline{a_1},\overline{a_2},\ldots,\overline{a_r}$ are glued together in $\Gamma$, then in $N(\Gamma)$ the vertices $b_{a_1}, b_{a_2},\ldots, b_{a_r}$ are identified.

\begin{example}
Consider the $Y$ electrical network from Example \ref{ex:YDelta}.  Using the computations in Example \ref{ex:YDelta} and directly counting almost perfect matchings, we can compute that for $N(\Gamma)$, we have
\begin{align*}
\Delta_{12} = \Delta_{45} &=ac = L_{\bar 1\bar 3|\bar 2} \\
\Delta_{23} = \Delta_{56} &=bc = L_{\bar 2\bar 3|\bar 1} \\
\Delta_{34} = \Delta_{16} &=ab = L_{\bar 1\bar 2|\bar 3} \\
\Delta_{13} = \Delta_{35}=\Delta_{15} &=abc = L_{\bar 1\bar 2\bar 3} \\
\Delta_{24} = \Delta_{46}=\Delta_{26} &=a+b+c = L_{\bar 1|\bar 2|\bar 3} \\
\Delta_{14} &=ab + ac = L_{\bar 1\bar 2|\bar 3} + L_{\bar 1 \bar 3| \bar 2}\\
\Delta_{25} &=ab + bc = L_{\bar 1\bar 2|\bar 3} + L_{\bar 2 \bar 3| \bar 1}\\
\Delta_{36} &=bc + ac = L_{\bar 1\bar 3|\bar 2} + L_{\bar 1 \bar 3| \bar 2}
\end{align*}
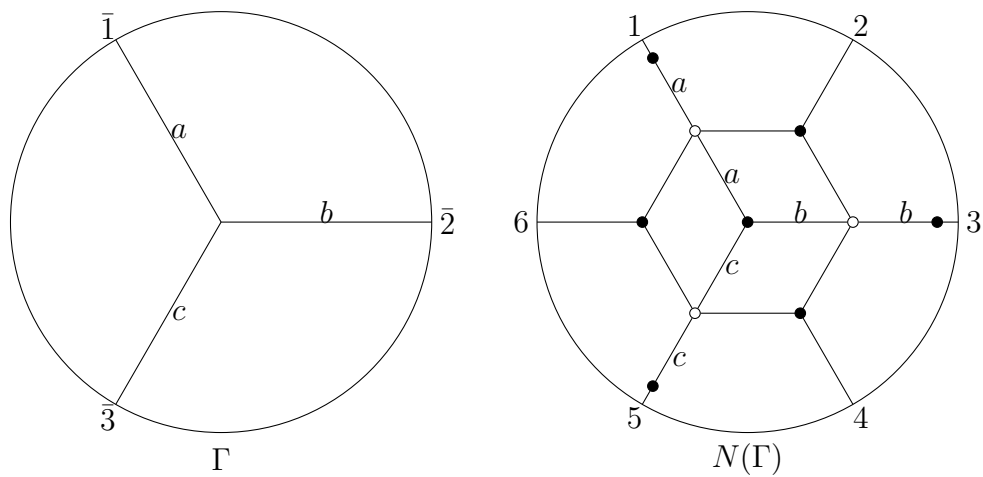
\begin{center}
\begin{tikzpicture}[scale = 0.7]
\draw (0,0) circle (4cm);
\draw (0:4) -- (0:0) -- (120:4);
\draw (0:0) -- (240:4);
\node at (120:4.3) {$\bar 1$};
\node at (0:4.3) {$\bar 2$};
\node at (240:4.3) {$\bar 3$};
\node at ($(120:2)+(0.2,0)$) {$a$};
\node at ($(0:2)+(0,0.2)$) {$b$};
\node at ($(240:2)+(0.2,0)$) {$c$};
\node at (0,-4.5) {$\Gamma$};
\begin{scope}[shift={(10,0)}]
\draw (0,0) circle (4cm);
\draw (0:4) -- (0:0) -- (120:4);
\draw (0:0) -- (240:4);
\draw (0:2) -- (60:2) -- (120:2) -- (180:2) -- (240:2) -- (300:2) -- (0:2);
\draw (60:2) -- (60:4);
\draw (180:2) -- (180:4);
\draw (300:2) -- (300:4);
\node at ($(120:1)+(0.2,0)$) {$a$};
\node at ($(120:3)+(0.2,0)$) {$a$};
\node at ($(0:1)+(0,0.2)$) {$b$};
\node at ($(0:3)+(0,0.2)$) {$b$};
\node at ($(240:1)+(0.2,0)$) {$c$};
\node at ($(240:3)+(0.2,0)$) {$c$};
\filldraw[black] (60:2) circle (0.1cm);
\filldraw[black] (0:0) circle (0.1cm);
\filldraw[black] (300:2) circle (0.1cm);
\filldraw[black] (180:2) circle (0.1cm);
\filldraw[black] (0:3.6) circle (0.1cm);
\filldraw[black] (120:3.6) circle (0.1cm);
\filldraw[black] (240:3.6) circle (0.1cm);
\filldraw[white] (0:2) circle (0.1cm);
\draw (0:2) circle (0.1cm);
\filldraw[white] (120:2) circle (0.1cm);
\draw (120:2) circle (0.1cm);
\filldraw[white] (240:2) circle (0.1cm);
\draw (240:2) circle (0.1cm);
\node at (120:4.3) {$1$};
\node at (60:4.3) {$2$};
\node at (0:4.3) {$3$};
\node at (300:4.3) {$4$};
\node at (240:4.3) {$5$};
\node at (180:4.3) {$6$};
\node at (0,-4.5) {$N(\Gamma)$};
  \end{scope}
\end{tikzpicture}
\end{center}

\end{example}

Let us first observe that the map $\Gamma \to N(\Gamma)$ is compatible with sending conductances to $0$ or $\infty$.  For example, if $\Gamma'$ is obtained from $\Gamma$ by deleting an edge $e$ then modulo a valent two vertex removal $N(\Gamma')$ is obtained from $N(\Gamma)$ by sending the corresponding weight to $0$.

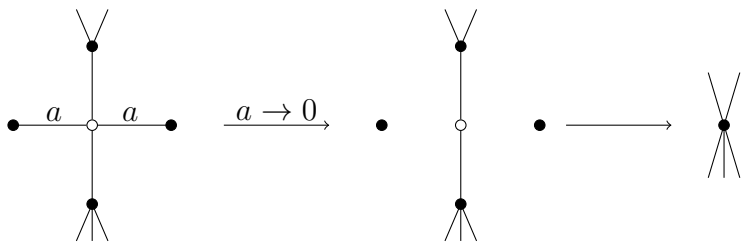
\begin{center}
\begin{tikzpicture}[scale=0.7]
\draw(0,1.5) -- (0.3,2.2);
\draw(0,1.5) -- (-0.3,2.2);
\draw(0,-1.5) -- (-0.3,-2.2);
\draw (0,-1.5) -- (0,-2.2);
\draw(0,-1.5) -- (0.3,-2.2);
\draw (0,0) -- (1.5,0);
\draw (0,0) -- (-1.5,0);
\draw (0,0) -- (0,1.5);
\draw (0,0) -- (0,-1.5);
\node at (0.725,0.2) {$a$};
\node at (-0.725,0.2) {$a$};
\filldraw[black] (1.5,0) circle (0.1cm);
\filldraw[black] (-1.5,0) circle (0.1cm);
\filldraw[black] (0,1.5) circle (0.1cm);
\filldraw[black] (0,-1.5) circle (0.1cm);
\filldraw[white] (0,0) circle (0.1cm);
\draw (0,0) circle (0.1cm);
\draw[->] (2.5,0) -- (4.5,0);
\node at (3.5,0.3) {$a \to 0$};
\begin{scope}[shift={(7,0)}]
\draw(0,1.5) -- (0.3,2.2);
\draw(0,1.5) -- (-0.3,2.2);
\draw(0,-1.5) -- (-0.3,-2.2);
\draw (0,-1.5) -- (0,-2.2);
\draw(0,-1.5) -- (0.3,-2.2);
\draw (0,0) -- (0,1.5);
\draw (0,0) -- (0,-1.5);
\filldraw[black] (1.5,0) circle (0.1cm);
\filldraw[black] (-1.5,0) circle (0.1cm);
\filldraw[black] (0,1.5) circle (0.1cm);
\filldraw[black] (0,-1.5) circle (0.1cm);
\filldraw[white] (0,0) circle (0.1cm);
\draw (0,0) circle (0.1cm);
\end{scope}
\draw[->] (9,0) -- (11,0);
\begin{scope}[shift={(12,0)}]
\draw(0,0) -- (0.3,1);
\draw(0,0) -- (-0.3,1);
\draw(0,0) -- (-0.3,-1);
\draw (0,0) -- (0,-1);
\draw(0,0) -- (0.3,-1);
\filldraw[black] (0,0) circle (0.1cm);
\end{scope}
\end{tikzpicture}
\end{center}

\subsection{Electroid varieties}\label{sec:concordant}
In the following, we will often identify $[2n]$ with $\{\bar 1, \tilde 1, \bar 2, \tilde 2, \ldots, \bar n, \tilde n\}$.  This is the identification we will use when we compare subsets of $[2n]$ with the vertices of partitions $\sigma$ and dual partitions $\tsigma$.  Call an $(n-1)$-element subset $I \subset [2n]$ \defn{concordant} with a non-crossing partition $\sigma$ if each part of $\sigma$, and each part of the dual partition $\tsigma$, contains exactly one element not in $I$.  In this situation we also say that $\sigma$ is concordant with $I$, or that $(\sigma,\tsigma)$ is concordant with $I$.  For $I \in \binom{[2n]}{n-1}$, let $\E(I) \subset \NC_n$ denote the set of non-crossing partitions concordant with $I$.  For $\sigma \in \NC_n$, let $\M(\sigma) \subset \binom{[2n]}{n-1}$ denote the collection of subsets concordant with $\sigma$. In Proposition \ref{prop:Msigma}, we will show that $\M(\sigma)$ is actually a positroid.

\begin{example}
Let $\sigma = (\bar 1, \bar 4, \bar 6| \bar 2, \bar 3| \bar 5)$ so that $\tsigma = (\tilde 1, \tilde 3| \tilde 2| \tilde 4, \tilde 5| \tilde 6)$.  Then $\sigma$ is concordant with $\{2,5,7,8,11\}$ but not concordant with $\{2,5,7,8,12\}$.  In the diagram we use black and white vertices to indicate elements of $\sigma$ versus $\tsigma$.  This color should not be confused with the bipartite coloring of $N(\Gamma)$.
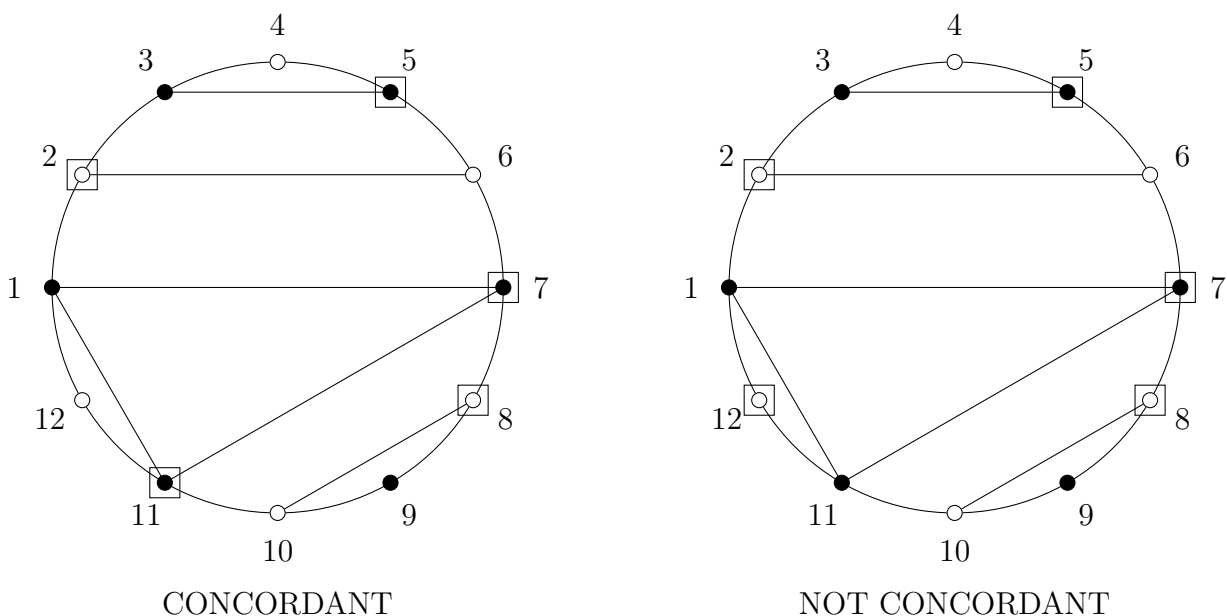
\begin{center}
\begin{tikzpicture}
\draw (0,0) circle (3cm);
\node at (180:3.5) {$1$};
\node at (120:3.5) {$3$};
\node at (60:3.5) {$5$};
\node at (0:3.5) {$7$};
\node at (300:3.5) {$9$};
\node at (240:3.5) {$11$};
\node at (150:3.5) {$2$};
\node at (90:3.5) {$4$};
\node at (30:3.5) {$6$};
\node at (-30:3.5) {$8$};
\node at (-90:3.5) {$10$};
\node at (-150:3.5) {$12$};

\coordinate (a11) at (240:3);
\coordinate (a7) at (0:3);
\coordinate (a8) at (-30:3);
\coordinate (a5) at (60:3);
\coordinate (a2) at (150:3);

\draw ($(a11)+(0.2,0.2)$) rectangle ($(a11)+(-0.2,-0.2)$);
\draw ($(a7)+(0.2,0.2)$) rectangle ($(a7)+(-0.2,-0.2)$);
\draw ($(a8)+(0.2,0.2)$) rectangle ($(a8)+(-0.2,-0.2)$);
\draw ($(a5)+(0.2,0.2)$) rectangle ($(a5)+(-0.2,-0.2)$);
\draw ($(a2)+(0.2,0.2)$) rectangle ($(a2)+(-0.2,-0.2)$);

\draw (180:3) -- (0:3) -- (240:3) -- (180:3);
\draw (120:3) -- (60:3);
\draw (150:3) -- (30:3);
\draw (-30:3) -- (-90:3);
\filldraw[black] (180:3) circle (0.1cm);
\filldraw[black] (120:3) circle (0.1cm);
\filldraw[black] (60:3) circle (0.1cm);
\filldraw[black] (0:3) circle (0.1cm);
\filldraw[black] (240:3) circle (0.1cm);
\filldraw[black] (300:3) circle (0.1cm);

\filldraw[white] (90:3) circle (0.1cm);
\filldraw[white] (150:3) circle (0.1cm);
\filldraw[white] (30:3) circle (0.1cm);
\filldraw[white] (-30:3) circle (0.1cm);
\filldraw[white] (-90:3) circle (0.1cm);
\filldraw[white] (-150:3) circle (0.1cm);
\draw (90:3) circle (0.1cm);
\draw (-90:3) circle (0.1cm);
\draw (150:3) circle (0.1cm);
\draw (-150:3) circle (0.1cm);
\draw (30:3) circle (0.1cm);
\draw (-30:3) circle (0.1cm);
\node at (-0,-4.2) {CONCORDANT};
\begin{scope}[shift={(9,0)}]
\draw (0,0) circle (3cm);
\node at (180:3.5) {$1$};
\node at (120:3.5) {$3$};
\node at (60:3.5) {$5$};
\node at (0:3.5) {$7$};
\node at (300:3.5) {$9$};
\node at (240:3.5) {$11$};
\node at (150:3.5) {$2$};
\node at (90:3.5) {$4$};
\node at (30:3.5) {$6$};
\node at (-30:3.5) {$8$};
\node at (-90:3.5) {$10$};
\node at (-150:3.5) {$12$};

\coordinate (a12) at (210:3);
\coordinate (a7) at (0:3);
\coordinate (a8) at (-30:3);
\coordinate (a5) at (60:3);
\coordinate (a2) at (150:3);

\draw ($(a12)+(0.2,0.2)$) rectangle ($(a12)+(-0.2,-0.2)$);
\draw ($(a7)+(0.2,0.2)$) rectangle ($(a7)+(-0.2,-0.2)$);
\draw ($(a8)+(0.2,0.2)$) rectangle ($(a8)+(-0.2,-0.2)$);
\draw ($(a5)+(0.2,0.2)$) rectangle ($(a5)+(-0.2,-0.2)$);
\draw ($(a2)+(0.2,0.2)$) rectangle ($(a2)+(-0.2,-0.2)$);

\draw (180:3) -- (0:3) -- (240:3) -- (180:3);
\draw (120:3) -- (60:3);
\draw (150:3) -- (30:3);
\draw (-30:3) -- (-90:3);
\filldraw[black] (180:3) circle (0.1cm);
\filldraw[black] (120:3) circle (0.1cm);
\filldraw[black] (60:3) circle (0.1cm);
\filldraw[black] (0:3) circle (0.1cm);
\filldraw[black] (240:3) circle (0.1cm);
\filldraw[black] (300:3) circle (0.1cm);

\filldraw[white] (90:3) circle (0.1cm);
\filldraw[white] (150:3) circle (0.1cm);
\filldraw[white] (30:3) circle (0.1cm);
\filldraw[white] (-30:3) circle (0.1cm);
\filldraw[white] (-90:3) circle (0.1cm);
\filldraw[white] (-150:3) circle (0.1cm);
\draw (90:3) circle (0.1cm);
\draw (-90:3) circle (0.1cm);
\draw (150:3) circle (0.1cm);
\draw (-150:3) circle (0.1cm);
\draw (30:3) circle (0.1cm);
\draw (-30:3) circle (0.1cm);
\node at (0,-4.2) {NOT CONCORDANT};
\end{scope}
\end{tikzpicture}
\end{center}
\end{example}

\begin{remark}
It is easy to see that $\E(I)$ and $\M(\sigma)$ are always non-empty.  In an earlier version of this work we observed that the number of $I \in \binom{[2n]}{n-1}$ satisfying $|\E(I)| = 1$ appeared to be the sequence $1,4,12,32,80,...$, $n2^{n-1}$.  David Speyer has recently proved this numerology.
\end{remark}
\begin{remark}
It is not true that for each $\tau$, there exists some $I$ such that $\tau$ is the only non-crossing matching concordant with $I$.  For example, take $n = 6$, and $\tau$ the non-crossing matching $\{(1,2),(3,12),(4,5),(6,9),(7,8),(10,11)\}$.  Then for each $I$ concordant with $\tau$, we have that $I$ is concordant with at least two non-crossing matchings. 
\end{remark}

Let $\P^{\binom{[2n]}{n-1}}$ be the Pl\"ucker projective space in which $\Gr(n-1,2n)$ is embedded.  Define a matrix $A = (a_{I \sigma})$ with columns labeled by non-crossing partitions and rows labeled by $(n-1)$-element subsets of $[2n]$ by
$$
a_{I \sigma} = \begin{cases} 1 & \mbox{if $\sigma$ is concordant with $I$} \\
0 & \mbox{otherwise.}
\end{cases}
$$
Let $\H' \subset \R^{\binom{[2n]}{n-1}}$ be the column space of the matrix $A$ (that is, the image of the corresponding linear transformation).  Let $\H$ be the image of $\H'$ in $\P^{\binom{[2n]}{n-1}}$.  Define 
$$\X = \X_n := \Gr(n-1,2n) \cap \H \subset \Gr(n-1,2n)$$ 
to be the intersection of the Grassmannian with the linear subspace $\H$.  

\begin{definition}
Let $f \in \Bound(n-1,2n)$.  If $\X \cap \oPi_f$ is non-empty, we define $\X_f := \X \cap \Pi_f$ to be the \defn{electroid variety} indexed by $f$, and $\oX_f:= \X \cap \oPi_f$ to be the \defn{open electroid variety} indexed by $f$. 
\end{definition}

\begin{theorem}\label{thm:stratification}
The intersection $\X \cap \oPi_f$ is non-empty exactly when $f  \in \Elec(n)$.  We have 
$$
\X_{f_\tau} = \bigsqcup_{\tau' \leq \tau} \oX_{f_{\tau'}}.
$$
\end{theorem}

\begin{remark}
The definitions and Theorem \ref{thm:stratification} also make sense and hold over $\C$.
\end{remark}

Let $\X_{\geq 0} = \X \cap \Gr(n-1,2n)_{\geq 0}$, and let $(\oX_{f})_{\geq 0} = \oX_f \cap \Gr(n-1,2n)_{\geq 0}$.

\begin{theorem}\label{thm:realizability}
The construction $\Gamma \mapsto N(\Gamma)$ gives an injection $\iota: E_n \hookrightarrow \Gr(n-1,2n)$ that induces a bijection $E_n \simeq \X_{\geq 0}$.  Thus every point in $\X_{\geq 0}$ is realizable by an electrical network.  Furthermore, 
$\iota(E_\tau) =  (\oX_{f_\tau})_{\geq 0}$, and we have
$$
\overline{\oX_{f_\tau}} = \bigsqcup_{\tau' \leq \tau} \oX_{f_{\tau'}} \qquad \text{ and } \qquad \overline{E_\tau} = \bigsqcup_{\tau' \leq \tau} E_{\tau'}
$$
where the closures are taken in the Hausdorff topologies on $\X_{\geq 0}$ and $E_n$ respectively.
\end{theorem}

We shall also call $\X_{\geq 0}$ the \defn{compactified space of electrical networks}, or the \defn{space of cactus networks}.  The proofs of Theorems \ref{thm:stratification} and \ref{thm:realizability} will be given in Sections \ref{sec:grovematch}--\ref{sec:realizability}.  Theorem \ref{thm:stratification} is proved in Sections \ref{sec:reductions} and \ref{sec:stratification}.  Theorem \ref{thm:realizability} is proved in Section \ref{sec:grovematch} ($\iota(E_n) \subseteq \X_{\geq 0}$), Section \ref{sec:strata} ($\iota(E_\tau) =  (\oX_{f_\tau})_{\geq 0}$), Section \ref{sec:ncorder} ($\iota$ is injective), Section \ref{sec:closure} (closure order), and Section \ref{sec:realizability} ($\iota$ is surjective).

\subsection{From groves to matchings}\label{sec:grovematch}
Let $\Gamma$ be an electrical network and $N(\Gamma)$ the corresponding bipartite graph.  Suppose $F \subset \Gamma$ is a grove in $\Gamma$.  Let $\tGamma$ be the planar dual of $\Gamma$, with vertices given by the faces of $\Gamma$, and edges for adjacent faces.  Also $\tGamma$ has boundary vertices $[\tilde n]$ arranged in the same way the vertices of $\tsigma$ are.  The spanning forest $F$ induces a dual spanning forest $\tF$ in $\tGamma$, determined by the condition: an edge $e \in \Gamma$ is present in $F$ if and only if the unique dual edge $\tilde e \in \tGamma$ intersecting $e$ is absent in $\tF$.

Let $\sigma$ be the boundary partition for $F$ and $\tsigma$ the boundary partition of $\tF$.  (Note that $\tsigma$ depends only on $\sigma$.)  As usual we may think of $\tsigma$ as a boundary partition on the even boundary vertices $\{2,4,\ldots,2n\}$ of $N(\Gamma)$.  That is, the boundary vertex $\bar i$ of $\sigma$ is identified with vertex $(2i-1)$ of $N(\Gamma)$, and the boundary vertex $\tilde i$ of $\tsigma$ is identified with vertex $2i$ of $N(\Gamma)$.

A \defn{rooting} $\xi$ of $(\sigma,\tsigma)$ is a choice of a boundary vertex, called the \defn{root}, for each component of $\sigma$, and each component of $\tsigma$.  Given $(F,\tF)$ and a rooting $\xi$ of $(\sigma(F),\tsigma(\tF))$, we define an almost perfect matching $\Pi = \Pi(F,\xi)$ in $N(\Gamma)$. This is a variant of a construction in work of Kenyon, Propp and Wilson \cite{KPW}.  Orient each component of $F$ and of $\tF$ towards the root vertex.  We match each interior white vertex $w_e$ in $F$ or $\tF$ with the black vertex which is at the source of $e$ for this orientation.  Any remaining unmatched interior vertex is matched with the marked boundary vertex.

\begin{lemma}
Let $\Pi = \Pi(F,\xi)$. Then the boundary partition $I(\Pi)$ is equal to the set of vertices in $N$ that are not roots.  Furthermore, $|I(\Pi)| = n-1$.
\end{lemma}
\begin{proof}
Recall that by convention each boundary vertex is joined to a black interior vertex, so that the boundary vertices should be considered to be white.  The boundary vertices that are used in the matching $\Pi(F,\xi)$ are exactly the boundary vertices that are roots.  Since each boundary vertex is white, by definition $I(\Pi)$ consists of the boundary vertices that are not roots.  The statement $|I(\Pi)| = n-1$ follows from Lemma \ref{lem:dualpart}.
\end{proof}

\begin{center}
\begin{tikzpicture}
\tikzset{->-/.style={decoration={
  markings,
  mark=at position .5 with {\arrow{>}}},postaction={decorate}}}
\draw (0,0) circle (3cm);
\coordinate (a) at (135:3);
\coordinate (b) at (45:3);
\coordinate (c) at (-45:3);
\coordinate (d) at (-135:3);
\coordinate (e) at (135:1.3);
\coordinate (f) at (45:1.3);
\coordinate (g) at (-45:1.3);
\coordinate (h) at (-135:1.3);
\draw[thick] (a) -- (e) -- (f) -- (b);
\draw[thick] (d) -- (h) -- (g) -- (c);
\draw[thick] (e) -- (h);
\draw[thick] (f) -- (g);
\coordinate (F) at (90:3);
\coordinate (G) at (0:3);
\coordinate (H) at (-90:3);
\coordinate (I) at (180:3);
\coordinate (A) at (90:2.3);
\coordinate (B) at (0:2.3);
\coordinate (C) at (-90:2.3);
\coordinate (D) at (180:2.3);
\coordinate (E) at (0,0);
\draw[dashed] (F) -- (A) -- (E) -- (B) -- (G);
\draw[dashed] (I) -- (D) -- (E) -- (C) -- (H);
\draw[dashed] (C) -- (B) --(A) -- (D) -- (C);
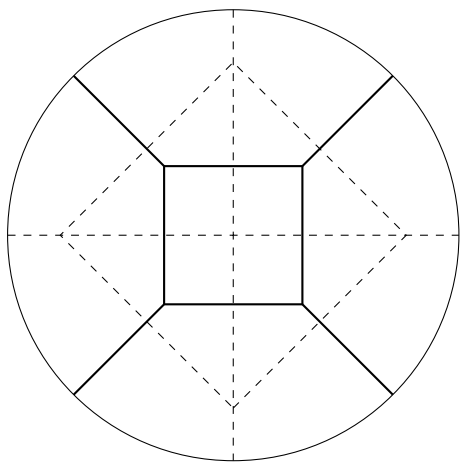
\node at (0,-3.5) {$\Gamma$ in thick lines and $\tGamma$ in dashed lines};

\begin{scope}[shift={(8,0)}]
\draw (0,0) circle (3cm);
\coordinate (a) at (135:3);
\coordinate (b) at (45:3);
\coordinate (c) at (-45:3);
\coordinate (d) at (-135:3);
\coordinate (e) at (135:1.3);
\coordinate (f) at (45:1.3);
\coordinate (g) at (-45:1.3);
\coordinate (h) at (-135:1.3);
\coordinate (F) at (90:3);
\coordinate (G) at (0:3);
\coordinate (H) at (-90:3);
\coordinate (I) at (180:3);
\coordinate (A) at (90:2.3);
\coordinate (B) at (0:2.3);
\coordinate (C) at (-90:2.3);
\coordinate (D) at (180:2.3);
\coordinate (E) at (0,0);
\filldraw[black] (b) circle (0.05cm);
\filldraw[black] (d) circle (0.05cm);
\filldraw[black] (G) circle (0.05cm);
\filldraw[black] (I) circle (0.05cm);
\draw[thick,->-] (a) -- (e);
\draw[thick,->-] (f) -- (e);
\draw[thick,->-] (e) -- (h);
\draw[thick,->-] (h) -- (d);
\draw[thick,->-] (g) -- (c);
\draw[dashed,->-] (F) -- (A);
\draw[dashed,->-] (A) -- (B);
\draw[dashed,->-] (B) -- (G);
\draw[dashed,->-] (D) -- (I);
\draw[dashed,->-] (H) -- (C);
\draw[dashed,->-] (C) -- (E);
\draw[dashed,->-] (E) -- (B);
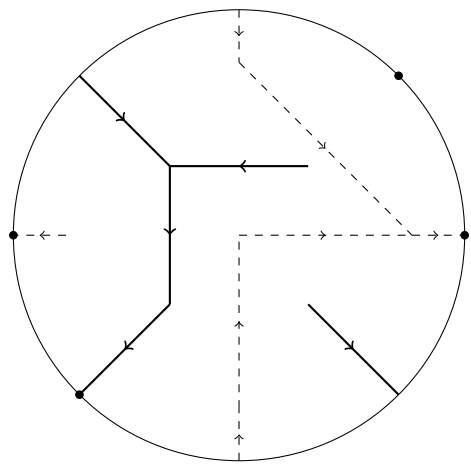
\node at (0,-3.5) {$F$ and a dual forest $\tF$ with roots chosen};
\end{scope}
\end{tikzpicture}
\end{center}

\begin{center}
\begin{tikzpicture}
\draw (0,0) circle (3cm);
\coordinate (a) at (135:3);
\coordinate (b) at (45:3);
\coordinate (c) at (-45:3);
\coordinate (d) at (-135:3);
\coordinate (e) at (135:1.3);
\coordinate (f) at (45:1.3);
\coordinate (g) at (-45:1.3);
\coordinate (h) at (-135:1.3);

\coordinate (aa) at (135:2.7);
\coordinate (bb) at (45:2.7);
\coordinate (cc) at (-45:2.7);
\coordinate (dd) at (-135:2.7);
\coordinate (ae) at ($(aa)!0.5!(e)$);
\coordinate (bf) at ($(bb)!0.5!(f)$);
\coordinate (cg) at ($(cc)!0.5!(g)$);
\coordinate (dh) at ($(dd)!0.5!(h)$);
\coordinate (ef) at ($(e)!0.5!(f)$);
\coordinate (eh) at ($(e)!0.5!(h)$);
\coordinate (fg) at ($(f)!0.5!(g)$);
\coordinate (gh) at ($(g)!0.5!(h)$);

\draw (a) -- (e) -- (f) -- (b);
\draw (d) -- (h) -- (g) -- (c);
\draw (e) -- (h);
\draw (f) -- (g);
\coordinate (F) at (90:3);
\coordinate (G) at (0:3);
\coordinate (H) at (-90:3);
\coordinate (I) at (180:3);
\coordinate (A) at (90:2.3);
\coordinate (B) at (0:2.3);
\coordinate (C) at (-90:2.3);
\coordinate (D) at (180:2.3);
\coordinate (E) at (0,0);
\draw[dashed] (F) -- (A) -- (E) -- (B) -- (G);
\draw[dashed] (I) -- (D) -- (E) -- (C) -- (H);
\draw[dashed] (C) -- (cg) -- (B) -- (bf) -- (A) -- (ae) -- (D) -- (dh) -- (C);

\filldraw[black] (aa) circle (0.1cm);
\filldraw[black] (bb) circle (0.1cm);
\filldraw[black] (cc) circle (0.1cm);
\filldraw[black] (dd) circle (0.1cm);
\filldraw[black] (e) circle (0.1cm);
\filldraw[black] (f) circle (0.1cm);
\filldraw[black] (g) circle (0.1cm);
\filldraw[black]  (h) circle (0.1cm);
\filldraw[black] (A) circle (0.1cm);
\filldraw[black] (B) circle (0.1cm);
\filldraw[black] (C) circle (0.1cm);
\filldraw[black] (D) circle (0.1cm);
\filldraw[black] (aa) circle (0.1cm);

\filldraw[white] (ae) circle (0.1cm);
\draw (ae) circle (0.1cm);
\filldraw[white] (bf) circle (0.1cm);
\draw (bf) circle (0.1cm);
\filldraw[white] (cg) circle (0.1cm);
\draw (cg) circle (0.1cm);
\filldraw[white] (dh) circle (0.1cm);
\draw (dh) circle (0.1cm);
\filldraw[white] (ef) circle (0.1cm);
\draw (ef) circle (0.1cm);
\filldraw[white] (eh) circle (0.1cm);
\draw (eh) circle (0.1cm);
\filldraw[white] (fg) circle (0.1cm);
\draw (fg) circle (0.1cm);
\filldraw[white] (gh) circle (0.1cm);
\draw (gh) circle (0.1cm);

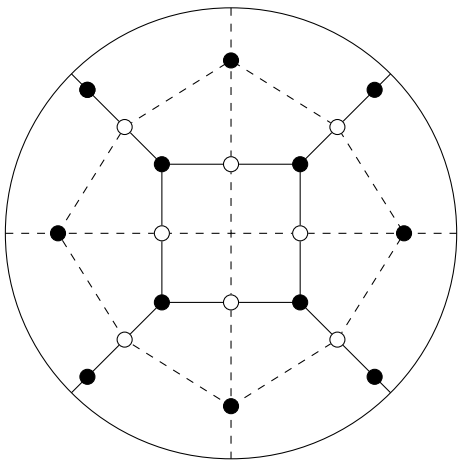
\node at (0,-3.5) {The planar bipartite graph $N(\Gamma)$};

\begin{scope}[shift={(8,0)}]
\draw (0,0) circle (3cm);
\coordinate (a) at (135:3);
\coordinate (b) at (45:3);
\coordinate (c) at (-45:3);
\coordinate (d) at (-135:3);
\coordinate (e) at (135:1.3);
\coordinate (f) at (45:1.3);
\coordinate (g) at (-45:1.3);
\coordinate (h) at (-135:1.3);

\coordinate (aa) at (135:2.7);
\coordinate (bb) at (45:2.7);
\coordinate (cc) at (-45:2.7);
\coordinate (dd) at (-135:2.7);
\coordinate (ae) at ($(aa)!0.5!(e)$);
\coordinate (bf) at ($(bb)!0.5!(f)$);
\coordinate (cg) at ($(cc)!0.5!(g)$);
\coordinate (dh) at ($(dd)!0.5!(h)$);
\coordinate (ef) at ($(e)!0.5!(f)$);
\coordinate (eh) at ($(e)!0.5!(h)$);
\coordinate (fg) at ($(f)!0.5!(g)$);
\coordinate (gh) at ($(g)!0.5!(h)$);

\draw (a) -- (e) -- (f) -- (b);
\draw (d) -- (h) -- (g) -- (c);
\draw (e) -- (h);
\draw (f) -- (g);
\coordinate (F) at (90:3);
\coordinate (G) at (0:3);
\coordinate (H) at (-90:3);
\coordinate (I) at (180:3);
\coordinate (A) at (90:2.3);
\coordinate (B) at (0:2.3);
\coordinate (C) at (-90:2.3);
\coordinate (D) at (180:2.3);
\coordinate (E) at (0,0);
\draw[dashed] (F) -- (A) -- (E) -- (B) -- (G);
\draw[dashed] (I) -- (D) -- (E) -- (C) -- (H);
\draw[dashed] (C) -- (cg) -- (B) -- (bf) -- (A) -- (ae) -- (D) -- (dh) -- (C);

\draw[line width=0.1cm] (aa) -- (ae);
\draw[line width=0.1cm] (f) -- (ef);
\draw[line width=0.1cm] (e) -- (eh);
\draw[line width=0.1cm] (h) -- (dh);
\draw[line width=0.1cm] (dd) -- (d);
\draw[line width=0.1cm] (g) -- (cg);
\draw[line width=0.1cm] (cc) -- (c);
\draw[line width=0.1cm] (bb) -- (b);
\draw[line width=0.1cm] (A) -- (bf);
\draw[line width=0.1cm] (B) -- (G);
\draw[line width=0.1cm] (E) -- (fg);
\draw[line width=0.1cm] (C) -- (gh);
\draw[line width=0.1cm] (D) -- (I);

\filldraw[black] (aa) circle (0.1cm);
\filldraw[black] (bb) circle (0.1cm);
\filldraw[black] (cc) circle (0.1cm);
\filldraw[black] (dd) circle (0.1cm);
\filldraw[black] (e) circle (0.1cm);
\filldraw[black] (f) circle (0.1cm);
\filldraw[black] (g) circle (0.1cm);
\filldraw[black]  (h) circle (0.1cm);
\filldraw[black] (A) circle (0.1cm);
\filldraw[black] (B) circle (0.1cm);
\filldraw[black] (C) circle (0.1cm);
\filldraw[black] (D) circle (0.1cm);
\filldraw[black] (aa) circle (0.1cm);

\filldraw[white] (ae) circle (0.1cm);
\draw (ae) circle (0.1cm);
\filldraw[white] (bf) circle (0.1cm);
\draw (bf) circle (0.1cm);
\filldraw[white] (cg) circle (0.1cm);
\draw (cg) circle (0.1cm);
\filldraw[white] (dh) circle (0.1cm);
\draw (dh) circle (0.1cm);
\filldraw[white] (ef) circle (0.1cm);
\draw (ef) circle (0.1cm);
\filldraw[white] (eh) circle (0.1cm);
\draw (eh) circle (0.1cm);
\filldraw[white] (fg) circle (0.1cm);
\draw (fg) circle (0.1cm);
\filldraw[white] (gh) circle (0.1cm);
\draw (gh) circle (0.1cm);
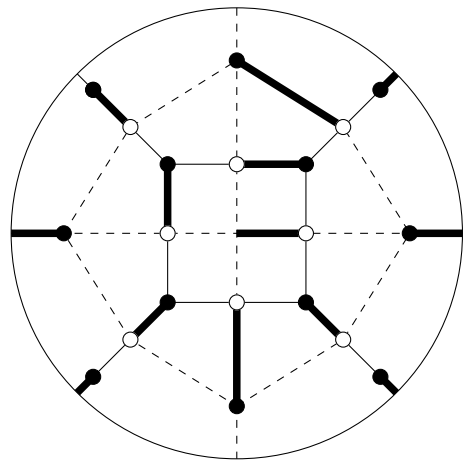
\node at (0,-3.5) {The almost perfect matching $\Pi(F,\xi)$};
\end{scope}

\end{tikzpicture}
\end{center}

\begin{theorem}\label{thm:concordant}
We have a bijection between matchings in $N(\Pi)$ with boundary partition $I$ and groves $F$ in $\Gamma$ with boundary partition $\sigma$ concordant with $I$.  Therefore $M(N(\Pi)) \in \Gr(n-1,2n)_{\geq 0}$ and for $I \in \binom{[2n]}{n-1}$
$$
\Delta_I(N(\Gamma)) = \sum_{\sigma} a_{I \sigma} L_{\sigma}(\Gamma).
$$
 In other words, $M(N(\Gamma)) \in \X_{\geq 0}$. 
\end{theorem}
\begin{proof}
We describe the map inverse to $(F,\xi) \mapsto \Pi(F,\xi)$.  Fix $I \in \binom{[2n]}{n-1}$ and let $\Pi$ be a matching with boundary partition $I$.  We construct $F \subset \Gamma$ and $\tF \subset \tGamma$ as follows: if the interior vertex $w_e$ is matched to a black vertex corresponding to a vertex of $\Gamma$, then we set $e \in F$, otherwise we set $e' \in \tF$, where $e' \in \tGamma$ is the edge dual to $e$.  Note that all interior white vertices in $N(\Gamma)$ are matched with interior black vertices.  The remaining edges of the matching $\Pi$ involve boundary vertices, and that determines a the roots $\xi$ (which is the same information as the subset $I$).  It only remains to argue that $F$ and $\tF$ defined in this way are trees.  If not, then let us suppose $F$ has a cycle, and let $R$ be the region inside this cycle.  It follows by an induction on the total number of edges and vertices inside $R$ that $N(\Gamma)$ has an odd number of vertices strictly inside $R$.  It is not possible for these vertices to be perfectly matched with each other, and hence this situation can never arise starting from an almost perfect matching $\Pi$.

This gives a bijection
$$
\{\Pi \mid I(\Pi) = I\} \leftrightarrow \{(\sigma(F),\tsigma(\tF)) \text{ concordant  with } I\}.
$$
The stated identity follows by taking weight generating functions.
\end{proof}


\begin{corollary}\label{cor:equiv}
Suppose $\Gamma$ and $\Gamma'$ are electrically equivalent cactus networks.  Then $N(\Gamma)$ and $N(\Gamma')$ are equivalent via local moves (including gauge equivalences) of planar bipartite graphs.
\end{corollary}
\begin{proof}
Suppose $\Gamma$ and $\Gamma'$ are electrically equivalent.  Then by Proposition \ref{prop:Lequiv}, we have $\L(\Gamma) = \L(\Gamma')$.  By Theorem \ref{thm:concordant}, we have $M(N(\Gamma)) = M(N(\Gamma'))$.  By Theorem \ref{thm:Pos}, $N(\Gamma)$ and $N(\Gamma')$ are equivalent via local moves.
\end{proof}

The claim of Corollary \ref{cor:equiv} could also be checked case-by-case (see Goncharov and Kenyon \cite{GK} for a discussion of this check in the absence of boundary vertices).   


\subsection{Electrical generators acting on the Grassmannian}
\label{sec:generators}
The operation $\Gamma \mapsto N(\Gamma)$ is also compatible with the operations of adding boundary spikes and boundary edges.  Note that $x_i(a) y_{i-1}(a) = y_{i-1}(a) x_i(a)$ as matrices.

\begin{proposition}\label{prop:elecaction}
The planar bipartite graphs $N(v_i(a) \cdot \Gamma)$ and $(x_i(a) y_{i-1}(a)) \cdot N(\Gamma) = (y_{i-1}(a) x_i(a)) \cdot N(\Gamma)$ are equivalent up to valent two vertex removals or additions.
\end{proposition}
\begin{proof}
Checked directly using the definition of $N(\Gamma)$.
\end{proof}

This suggests the following simple representation of the electrical braid relations studied in \cite{LP} (see Theorem \ref{thm:LP}).  This result was obtained jointly with Alex Postnikov.

\begin{proposition}\label{prop:elecbraid}
For $1 \leq i \leq 2n$, let $u_i(a) = x_i(a) y_{i-1}(a)=y_{i-1}(a) x_i(a) \in GL_{2n}$.  Then $u_i(a)$ satisfy the relations
\begin{enumerate}
\item
$u_i(a)u_i(b) = u_i(a+b)$
\item
$u_i(a) u_j(b) = u_j(b) u_i(a)$ for $|i-j|\geq 2$
\item
$$
u_i(a) u_{i \pm 1}(b) u_i(c) = u_{i \pm 1}({bc}/({a+c+abc})) u_i(a+c+abc) u_{i \pm 1}({ab}/({a+c+abc})).
$$
\end{enumerate}
\end{proposition}

\begin{remark}
In \cite{LP} it is established that the relations of Proposition \ref{prop:elecbraid} essentially generate the symplectic group.  This is however not clear from our current perspective.
\end{remark}

\begin{proposition}\label{prop:closed}
$\X_n$ is closed under the actions $\{u_i(a)\}$ for $i \in [2n]$.
\end{proposition} 
\begin{proof}
Let $X \in \X_n$ and let $X' = u_i(a) \cdot X$.  By definition, there exists a point $\L = (L_\sigma) \in \P^{{\NC_n}}$, such that $\Delta_I(X) = \sum a_{I \sigma} L_\sigma$.  Assume that $i = 2k-1$ is odd; the case $i$ even is similar.  If $k$ is isolated in $\sigma$, define $L'_\sigma$ by 
$$
L'_\sigma =  L_\sigma +a\sum_{\kappa} L_\kappa 
$$
where the summation is over non-crossing partitions $\kappa$ obtained from $\sigma$ by merging $k$ with any of the parts of $\sigma$.  If $k$ is not isolated in $\sigma$, then define 
$$
L'_\sigma  =  L_\sigma.
$$
We claim that $\Delta_I(X') = \sum a_{I \sigma} L'_\sigma$.  One way to see this is by directly using the combinatorial interpretation of $a_{I \sigma}$.  Another way to see this is to note that the above formulae for $\L'$ are what we would get if $X$ is of the form $M(N(\Gamma))$ for some electrical network, and $\L' = \L(\Gamma')$, where $\Gamma' = v_i(a) \cdot \Gamma$.  
\end{proof}

\begin{remark}
It follows from Proposition \ref{prop:closed} that the electrically nonnegative part $(EL_{2n})_{\geq 0}$ of the electrical Lie group of \cite{LP} acts on the compactified space $E_n$ (or $\X_{\geq 0}$) of electrical networks.
\end{remark}

\subsection{Electrical strata to positroid strata}
\label{sec:strata}
The following proposition proves part of Theorem \ref{thm:realizability}.
\begin{proposition}\label{prop:ftau}
Suppose that $\Gamma$ is an electrical network on $[\bar n]$, and that $\L(\Gamma) \in E_\tau$.  Then $M(N(\Gamma)) \in \oPi_{f_\tau}$.
\end{proposition}

\begin{proof}
Suppose $\Gamma$ is not critical.  Then it is electrically equivalent to a critical graph $\Gamma'$ and by Theorem \ref{thm:concordant}, we have $M(N(\Gamma)) = M(N(\Gamma'))$.  Thus we may, and will suppose that $\Gamma$ is critical.

So it is enough to prove the claim for critical electrical networks $\Gamma$, and we may pick the electrically equivalent representative that we like.  We shall proceed by induction on $n$, followed by induction on the number of edges in $\Gamma$.  The base case $n = 1$ is trivial.

Now suppose $\Gamma$ has no edges.  Then it is a hollow cactus with some boundary points identified.  In this case, the claim can be checked directly.   For example, take $n = 5$ and the hollow cactus $\Gamma$ with boundary points $\bar 2, \bar 3, \bar 5$ identified.  Then we have $N(\Gamma)$ and $\tau(\Gamma)$ as illustrated.

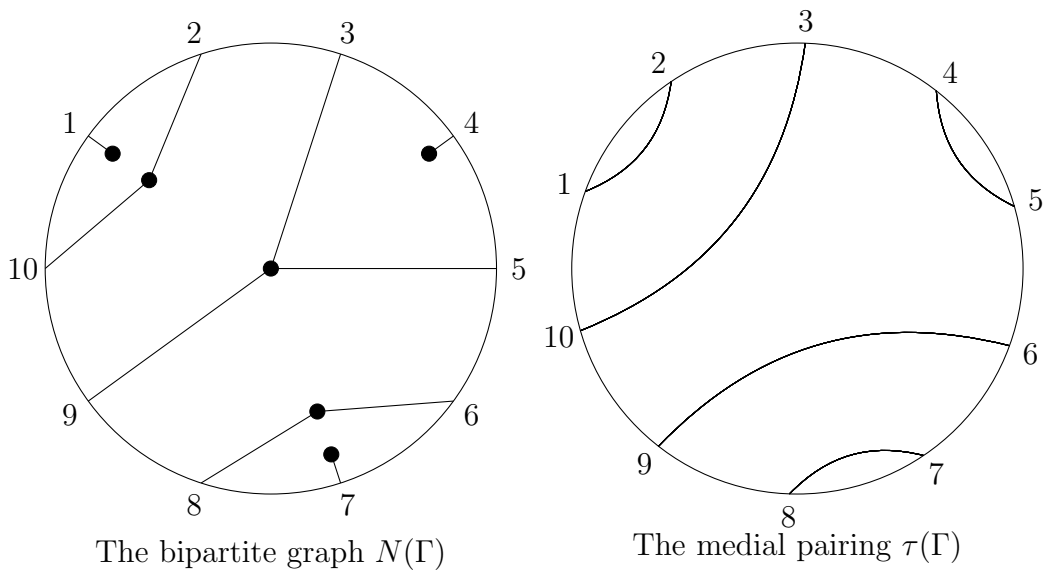
\begin{tikzpicture}
\draw (0,0) circle (3cm);
\foreach \i in {1,...,10}
{
\node at (180-\i*36:3.3) {$\i$};
}
\draw (180-36:3) -- (180-36:2.6);
\filldraw[black] (180-36:2.6) circle (0.1cm);
\draw (180-2*36:3) -- (180-36:2) -- (180-10*36:3);
\filldraw[black] (180-36:2) circle (0.1cm);
\draw (0,0) -- (180-3*36:3);
\draw (0,0) -- (180-5*36:3);
\draw (0,0) -- (180-9*36:3);
\filldraw[black] (0,0) circle (0.1cm);
\draw (180-4*36:3) -- (180-4*36:2.6);
\filldraw[black] (180-4*36:2.6) circle (0.1cm);
\draw (180-7*36:3) -- (180-7*36:2.6);
\filldraw[black] (180-7*36:2.6) circle (0.1cm);
\draw (180-6*36:3) -- (180-7*36:2) -- (180-8*36:3);
\filldraw[black] (180-7*36:2) circle (0.1cm);
\node at (0,-3.8) {The bipartite graph $N(\Gamma)$};
\begin{scope}[shift={(7,0)}]
\draw (0,0) circle (3cm);
\foreach \i in {1,...,10}
{
\node at (196-\i*36:3.3) {$\i$};
\draw (196-36:3) to [bend right] (196-2*36:3);
\draw (196-3*36:3) to [bend left] (196-10*36:3);
\draw (196-4*36:3) to [bend right] (196-5*36:3);
\draw (196-6*36:3) to [bend right] (196-9*36:3);
\draw (196-7*36:3) to [bend right] (196-8*36:3);
}
\node at (0,-3.7) {The medial pairing $\tau(\Gamma)$};
\end{scope}
\end{tikzpicture}

The reader is encouraged to verify that the bounded affine permutation $f_{N(\Gamma)} = f_{M(N)}$ is equal to $f_{\tau(\Gamma)}$.  For example, the trip $T_9$ turns right at the black vertex and ends at $5$, so $f_{N(\Gamma)}(9) = 15$.  (Alternatively, one can compute $M(N)$ using matchings, and then use \eqref{eq:perm}.)

Suppose $\Gamma$ has an isolated boundary vertex $\bar k$.  Then $\tau(2k-1) = 2k$ while $f_\tau(2k-1) = 2k-1$ and $f_\tau(2k) = 2k+2n-2$.  Let $\Gamma'$ be the critical electrical network on $\bar 1, \ldots, \overline{k-1},\overline{k+1}, \ldots, \bar n$ obtained from $\Gamma$ by removing $\bar k$.  By induction, the claim is true for $\Gamma'$.  It is straightforward to check that the claim also holds for $\Gamma$.

In a similar manner, one deals with the case that boundary vertices $\bar k$ and $\overline{k+1}$ are glued in $\Gamma$.

Otherwise, it is easy to see by considering medial graphs that some critical representative of the electrical equivalence class of $\Gamma$ will have a boundary spike at some boundary vertex $\bar k$, or a boundary edge between $\bar k$ and $\overline{k+1}$.  Let us assume we are in the boundary spike case, the other case being similar.  Let $\tau$ be the medial pairing of $\Gamma$ and $\tau'$ be the medial pairing of the electrical network $\Gamma'$ where this boundary spike is removed.  Adding a boundary spike at $\bar k$ introduces a crossing between the strands $T_{2k-1}$ and $T_{2k}$ of $\tau$, so $f_\tau = s_{i-1} f_{\tau'} s_i$, where $i = 2k-1$.

According to Proposition \ref{prop:elecaction}, adding a boundary spike to $\Gamma'$ corresponds to adding two bridges to $N(\Gamma')$: one which is white at $i$ and black at $i+1$, and another one which is black at $i-1$ and white at $i$.  By \eqref{eq:perm} and Lemma \ref{lem:networkbridge} (see also the argument in Proposition \ref{prop:reduce}) we see that $f_{M(N(\Gamma))} = s_{i-1} f_{M(N(\Gamma'))} s_i$, so by induction we conclude that $M(N(\Gamma)) \in \oPi_{f_\tau}$.
\end{proof}

Using this we obtain another characterization of $\M(\sigma)$.

\begin{proposition}\label{prop:Msigma}
Let $\sigma \in \NC_n$.  Then 
$$
\M(\sigma) = \M(f_{\tau(\sigma)})
$$
is the positroid of $f_{\tau(\sigma)}$.  Equivalently,
$$
\M(\sigma) = \{I \mid I \geq_a I_a\text{ for all } a\}
$$
where $\I(\sigma) = (I_1,\ldots, I_{2n})$ is the Catalan necklace of $\sigma$.  In particular, $I_a(\sigma)$ is concordant with $\sigma$ for any $a$.
\end{proposition}
\begin{proof}
The point $p_\sigma \in E_n$ has one non-vanishing grove coordinate $L_\sigma$.  By Proposition \ref{prop:ftau}, $\iota(p_\sigma) \in \oPi_{f_{\tau(\sigma)}}$.  The claim then follows from Theorem \ref{thm:concordant} and Theorem \ref{thm:TNNmain}.
\end{proof}

\subsection{Proof of the injectivity part of Theorem \ref{thm:realizability}}
\label{sec:ncorder}
Recall from Section \ref{sec:Catalan} that there is an $a$-shifted dominance ordering (and a $a$-shifted lexicographic order) on non-crossing partitions obtained via the bijection $\sigma \mapsto I_a(\sigma)$.
The following proposition establishes the injectivity part of Theorem \ref{thm:realizability}.

\begin{proposition}\label{prop:injection}
The map $\Gamma \mapsto N(\Gamma)$ induces an injection $\iota: E_n \hookrightarrow \Gr(n-1,2n)_{\geq 0}$.
\end{proposition}
\begin{proof}
The map is given in coordinates by $\Delta_I = \sum_{\sigma} L_{\sigma}$, where the sum is over $\sigma$ concordant with $I$.  Let us restrict our attention to Catalan subsets $I$, that is $I = I_1(\sigma)$ for some $\sigma \in \NC_n$.  Recall that the map $\sigma \mapsto I_1(\sigma)$ is injective, and by Proposition \ref{prop:Msigma} $\sigma$ is concordant with $I_1(\sigma)$.  But $\sigma'$ can be concordant with $I_1(\sigma)$ only if $I_1(\sigma') \leq I_1(\sigma)$ in dominance order.  It follows that when this transition formula is restricted to Catalan subsets, we obtain an invertible triangular system.  
\end{proof}
%

\subsection{Reductions used in proof of Theorem \ref{thm:stratification}}\label{sec:reductions}
This section contains technical results used in the proof of Theorem \ref{thm:stratification}.  Let $X \in \X_n$, and suppose $X \in \oPi_f$, where $f = f_\tau \in \Elec(f)$ is an electrical affine permutation.  

\begin{lemma}\label{lem:fixedpoint}
Suppose $f(i) = i$ for some $i$.  Then $\Delta_I(X) = 0$ if $i \in I$ or if $I \cap \{i-1,i+1\} = \emptyset$.  Furthermore, we have $\Delta_I = \Delta_{I - \{i-1\} \cup \{i+1\}}$ whenever $i-1 \in I$ but $i+1 \notin I$.
\end{lemma}
\begin{proof}
The first claim only uses that $X \in \oPi_f$.  By \eqref{eq:perm} the column $v_i$ of (an $(n-1) \times 2n$ matrix representative of) $X$ is the 0-column.  Thus $\Delta_I(X) = 0$ if $i \in I$.  Similarly, $v_{i+1}$ is not in the span of $v_{i+2},\ldots, v_{i-2}$, so $\Delta_I(X) = 0$ if $I \cap \{i-1,i+1\} = \emptyset$.

The final claim requires that $X \in \H$.  For simplicity, assume that $i = 2k-1$ is odd.  We note that $\sigma$ is concordant with some $I$ satisfying $i \in I$ if and only if $\bar k$ is not isolated in $\sigma$.  Furthermore $\sigma \mapsto I_i(\sigma)$ is a bijection between $\{\sigma \mid \bar k \text{ is not isolated } \}$ and $i$-shifted Catalan subsets  $I$ satisfying $i \in I$.  Since $\Delta_I(X) = 0$ if $i \in I$ we deduce in the same way as Proposition \ref{prop:injection} that $L_\sigma(X) = 0$ whenever $\bar k$ is not isolated in $\sigma$.

Let $\sigma$ be a non-crossing partition where $\bar k$ is isolated.  Then $\widetilde{(k-1)}$ and $\tilde k$ must belong to the same part of $\tsigma$.  Suppose $i-1 \in I$ and $i,i+1 \notin I$.  Then $\sigma$ and $I$ are concordant if and only if $\sigma$ and $I - \{i-1\} \cup \{i+1\}$ are concordant.  This proves the claim.\end{proof}

\begin{proposition}\label{prop:fixedpoint}
Suppose $X \in \oPi_f \cap \X_n$, where $f(i) = i$ and $f(i+1) = i+2n-1$.  Then as $J$ varies over $(n-2)$-element subsets of $\{1,2,\ldots,i-1,i+2,\ldots,2n\}$, the collection of Pl\"ucker coordinates
$$
\Delta_J(Y) = \Delta_{J \cup \{i+1\}}(X)  
$$
defines a point $Y \in \X_{n-1} \subset \Gr(n-2,2n-2)$.  Furthermore, $X$ can be recovered from $Y$.
\end{proposition}
\begin{proof}
As $f(i+1) = i+2n-1$, we have $\Delta_J(Y) \neq 0$ for some $J$.  Since $\Delta_I(X)$ satisfy the Pl\"ucker relations, it is clear that $\Delta_J(Y)$ satisfy the Pl\"ucker relations as well, so $Y \in \Gr(n-2,2n-2)$.  To see that $Y \in \X_{n-1}$, let us assume for simplicity that $i = 2k-1$.  Then we define $L'_\kappa = L_{\kappa \cup \{\bar k\}}$, where $\kappa$ is a non-crossing partition of $\{\bar 1,\bar 2,\ldots,\overline{k-1},\overline{k+1},\ldots,\bar n\}$ and $\bar k$ is isolated in $\kappa \cup \{\bar k\}$.  Then $\Delta_J(Y) = \sum_{\kappa} a_{J \kappa} L'_\kappa$.

Finally, Lemma \ref{lem:fixedpoint} shows that all $\Delta_{I}(X)$ can be recovered from $\Delta_{J \cup \{i+1\}}(X)$ for  $J$ varying over $(n-2)$-element subsets of $\{1,2,\ldots,i-1,i+2,\ldots,2n\}$.
\end{proof}


\begin{lemma}\label{lem:union}
Suppose $f \in \Elec(n)$.  Then 
$$
\{a+1 \mod 2n \mid a \in I_{i+1}(f)\}  \bigsqcup J_i(f) \bigsqcup \{i,i+1\} = [2n].
$$
\end{lemma}
\begin{proof}
Set $I = I_{i+1}(f)$ and $J = J_i(f)$.  Suppose $a \in I$ and $f(b) = a$ where $b < i+1 \leq a$.  Then  $f(a+1) = b+2n-1 < i+2n$, so $a+1 \mod n \notin J$.  Thus $\{a+1 \mod n \mid a \in I\}  \cap J = \emptyset$.  Also $i-1,i \notin I$ because $f(a) < a+2n-1$ for all $a \in \Z$.  Similarly $\{i,i+1\} \notin J$.  So the stated union is disjoint, and counting shows that we the union is $[2n]$.
\end{proof}

\begin{lemma}\label{lem:IJ}
Let $f \in \Elec(n)$ be an electrical affine permutation.   Suppose there exists $i$ such that $i < f(i) < f(i+1)$.  Then $i+1 \in I_{i+1}(f)$ and $(i-1) \in J_i(f)$.  Also $f(i+1) < i+2n-1$.
\end{lemma}
\begin{proof}
Since $f \in \Elec(n)$, we have $f(i+1) \leq i+2n-1$.  If $f(i+1) = i+2n-1$ we would get $f(i) = i$, contradicting the asusmption.

Set $I = I_{i+1}(f)$ and $J = J_i(f)$.  We have $i+1 \in I$ since $f(i+1) > i+1$.  We repeatedly use the definition (Section \ref{sec:elecaffine}) of electrical affine permutation in the following.  Suppose $f(i-1) = i-1$.  Then $f(i) = i+2n-2$.  This is impossible since $f(i) < f(i+1)$ and $f(i+1) < i+2n-1$.  So $f(i-1) > i-1$.  This gives $(i-1) \in J$.
\end{proof}

In the situation of Lemma \ref{lem:IJ}, define 
\begin{equation}\label{eq:IJ}
\begin{aligned}
I &=I_{i+1}(f) &\qquad  I' &= I - \{i+1\} \cup \{i\} \\ 
J&=J_{i}(f) & \qquad J' &= J - \{i-1\} \cup \{i\}.
\end{aligned}
\end{equation}
Recall that $\M(\sigma)$ denotes the collection of subsets concordant with $\sigma$, and $\I(\sigma)$ is the Grassmann necklace associated with $\sigma$.  Also recall that $\E(I)$ is the set of non-crossing partitions concordant with $I$.  If $f = f_\tau \in \Elec(n)$, we write $\E(f):=\E(\tau)$ to be the electroid of $E_\tau$.


\begin{lemma}\label{lem:CIJ}
In the above situation, we have 
$$
\E(I) \cap \E(f) = \E(J) \cap \E(f) \qquad \text{and} \qquad \E(I') \cap \E(f) = \E(J') \cap \E(f).
$$
\end{lemma}
\begin{proof}
Suppose $\sigma \in \E(I) \cap \E(f)$.  Then $I \in \M(\sigma)$, so $I_{i+1}(\sigma) \leq_{i+1} I$.  But we also have $\I(f) \geq \I(\sigma)$, so $I = I_{i+1}(f) \leq_{i+1} I_{i+1}(\sigma) \leq_{i+1} I$ implies that $I_{i+1}(\sigma) = I$.  Thus $|\E(I) \cap \E(f)| \leq 1$.  Similarly, $|\E(J) \cap \E(f)| \leq 1$.  On the other hand, $\E(I) \cap \E(f)$ is non-empty because any $X \in \X \cap \oPi_f$ satisfies $\Delta_I(X) \neq 0$, so $L_\sigma \neq 0$ for some $\sigma \in \E(I) \cap \E(f) \neq 0$.  Similarly $|\E(J) \cap \E(f)| = 1$.

Let $\sigma$ be given by $I_{i+1}(\sigma) = I$.  We claim that $\sigma \in \E(I) \cap \E(f)$ (and a similar statement holds for $J$).  We show that $\sigma$ given by $I_{i+1}(\sigma) = I$ also satisfies $J_i(\sigma) = J$.  Let $\tau$ be the non-crossing matching corresponding to $\sigma$, thought of as a sequence of $U$-s and $D$-s forming a Dyck path, where $i+1 \in [2n]$ is taken to be the start (and always a $U$).  Then $\tau$ has $U$-s at positions
$$
\{i+1\} \cup \{a+1 \mid a \in I\}.
$$
Let $\tau'$ be the non-crossing matching such that $J_i(\tau') : = J_i(f_{\tau'}) = J$.  Again think of $\tau'$ as a sequence of $U$-s and $D$-s starting at $i+1$.  Then $\tau'$ has $D$-s in the positions specified by $J \cup \{i\}$.  By Lemma \ref{lem:union}, $\tau = \tau'$.  So we have shown that $\E(I) \cap \E(f) = \E(J) \cap \E(f)$.

We now claim that the assumption $i < f(i) < f(i+1) \leq i+2n$ implies that $\tau$ has the property that $i$ is always joined to $i+1$.    To see this suppose $f = f_\eta$.  Then the $U$, $D$-sequence of $\tau$ is obtained from $\eta$ as follows: starting from $i+1$ and going clockwise, we write a $U$ whenever we encounter an endpoint of a strand in $\eta$ the first time, and a $D$ whenever we encounter the endpoint the second time.  The inequalities imply that the strands $T_i$ and $T_{i+1}$ in $\eta$ starting at $i$ and $i+1$ intersect, from which the claim follows.

Now suppose $\kappa \in \E(I') \cap \E(f)$.  For simplicity, we assume for the rest of the proof that $i = 2k$ is even, with the odd case being analogous. The assumptions imply that in the electrical cell corresponding to $f$, we can find a critical electrical network $\Gamma$ with an edge $e$ joining $\bar k$ to $\overline{k+1}$.  Since $i+1, i-1 \notin I'$, but $i \in I'$ we deduce that $\bar k$ and $\overline{k+1}$ do not belong to the same part of $\kappa$.  Let $\kappa'$ be obtained by gluing the parts of $\kappa$ containing $\bar k$ and $\overline{k+1}$.  We claim that $\kappa' = \sigma$.  Any grove in $\Gamma$ with boundary partition $\kappa$ does not use the edge $e$, and adding the edge $e$ gives a grove with boundary partition $\kappa'$.  It follows that $\kappa' \in \E(f)$.  Furthermore, $\kappa' \in \E(I)$: the new part in $\kappa'$ containing $\bar k$ and $\overline{k+1}$, when considered as a subset of $[2n]$, intersects $[2n]\setminus I$ in exactly $i-1$, while in the dual partition $\tilde{\kappa}'$ the part containing $\tilde k$, considered as a subset of $[2n]$, intersects $[2n]\setminus I$ in exactly $i$.  Thus $\sigma$ is obtained from $\kappa$ by gluing the parts containing $\bar k$ and $\overline{k+1}$.  But $\sigma \in \E(J)$, and it is clear from this description of $\kappa$ that $\kappa \in \E(J')$ as well.  The other inclusion is proved in an identical manner, so we deduce that $\E(I') \cap \E(f) = \E(J') \cap \E(f)$.

(In fact, $\tau(\kappa)$ is obtained from $\tau$ above by replacing edges $(i,i+1), (a < b)$ of $\tau$ by edges $(i,b), (i+1,a)$).  
\end{proof}

\begin{proposition}\label{prop:elecreduce}
Suppose $f \in \Elec(n)$ and $i < f(i) < f(i+1)$, and define $I,I',J,J'$ as in \eqref{eq:IJ}.  Then $s_{i-1} f s_i > s_{i-1}f, fs_i > f$ and $f' = s_{i-1} f s_i \in \Elec(n)$. 
For $X \in \X \cap \oPi_f$, we have
$$
\Delta_I(X) = \Delta_J(X) \qquad \text{ and } \qquad \Delta_{I'}(X) = \Delta_{J'}(X).
$$
If $a = \frac{\Delta_I(X)}{\Delta_{I'}(X)} = \frac{\Delta_J(X)}{\Delta_{J'}(X)}$ is well defined then
$$
X' := u_i(-a) \cdot X \in \X \cap \oPi_{f'}.
$$
Furthermore, if $X \in \X_{\geq 0}$, then $a$ is always well-defined, and $X' \in \X_{\geq 0}$ as well.

\end{proposition}
\begin{proof}
Let $j = f(i)$ and $j' = f(i+1)$.  Then $i < i+1 < j+1 < j'+1 < i+2n$ by Lemma \ref{lem:IJ}.  We have $f(j+1) = i+2n-1$ and $f(j'+1) = i+2n$.  So $s_{i-1}f = f t_{j+1,j'+1} > f$.  Also, by $f(i) < f(i+1)$, we have $fs_i > f$.  Since $t_{j+1,j'+1}$ and $s_i$ commute, we conclude that $s_{i-1} f s_i > s_{i-1}f, fs_i > f$.

The well-definedness of $I'$ and $J'$ follow from Lemma \ref{lem:IJ}.  The equalities $\Delta_I(X) = \Delta_J(X)$ and  $\Delta_{I'}(X) = \Delta_{J'}(X)$ follow from Lemma \ref{lem:CIJ} and the definition of $\X$.  Since $u_i(-a) = x_i(-a)y_{i-1}(-a)$, the statement $X' \in \X \cap \oPi_{f'}$ follows from (two applications of) Proposition \ref{prop:reduce} and Proposition \ref{prop:closed}. Note that by Lemma \ref{lem:union}, $i+1 \notin J$ and $i+1 \notin J'$ so $\Delta_J(X) = \Delta_J(x_i(-a) \cdot X)$ and $\Delta_{J'}(X) = \Delta_{J'}(x_i(-a) \cdot X)$.

Finally, the last statement just from Proposition \ref{prop:reduce}.
\end{proof}

\subsection{Proof of Theorem \ref{thm:stratification}}
\label{sec:stratification}
Suppose $X \in \X \cap \oPi_f$.  We shall show that $f = f_\tau$ for some $\tau$.  We proceed by induction on the length of $f$.  If $\ell(f) = 0$, then $\oPi_f$ is the top positroid cell, and we know that $f = f_{\tau_\top}$.  

Now suppose $\ell(f) > 0$.  Then there is some $i$ such that $f(i) > f(i+1)$.  Then $f s_i < f$.  Let 
$$f' = \begin{cases} s_{i-1}fs_i & \mbox{ if $s_{i-1} f s_i < f s_i$} \\
f s_i. & \mbox{otherwise.}
\end{cases}
$$
Let $X' = u_i(a) \cdot X$ for a generic value of $a$.  It follows from \eqref{eq:perm} that $X' \in \oPi_{f'}$.  By Proposition \ref{prop:closed}, we have $X'\in \X \cap \oPi_{f'}$.  By the inductive hypothesis $f' = f_{\tau'}$ for some $\tau' \in P_n$.  

It is easy to check that $f'$ satisfies $i < f'(i) < f'(i+1)$.
By Proposition \ref{prop:elecreduce} we have that $\Delta_I(X') = \Delta_J(X')$ and $\Delta_{I'}(X') = \Delta_{J'}(X')$, where the subsets $I,J,I',J'$ are as in Proposition \ref{prop:elecreduce} but for $f'$.  But then (essentially by our construction of $X'$) we must have $a = \frac{\Delta_I(X)}{\Delta_{I'}(X)}$.  We conclude that $X$ lies in $\oPi_{s_{i-1}f' s_i}$ where $s_{i-1} f' s_i > s_{i-1} f', s_i f'> f'$.  Thus $f = s_{i-1} f' s_i$, and $f = f_{\tau}$ for some $\tau \in P_n$.

Thus we have a decomposition
$$
\X = \bigsqcup_{\tau \in P_n} \oX_{f_\tau}.
$$

\subsection{Closure partial order on electroid strata}
\label{sec:closure}
The following establishes the closure partial order claim of Theorem \ref{thm:realizability}. 

\begin{proposition}\label{prop:Pnorder}
We have
$$
\overline{E_\tau} = \bigsqcup_{\tau' \leq \tau} E_{\tau'}.
$$
\end{proposition}
\begin{proof}
Suppose $\tau' \leq \tau$.  Uncrossing a crossing in a medial graph $G$ corresponds to either contracting (that is, gluing the endpoints of) an edge $e$, or deleting that edge.  This corresponds respectively to taking the edge weight $w(e)$ to $\infty$, or to $0$.  It follows easily that from the definition of $P_n$ that if $\tau' \leq \tau$ then $E_{\tau'} \subset \overline{E_\tau}$.

Now for the converse suppose that $L \in \overline{E_\tau}$.  Since $\iota$ is injective (Proposition \ref{prop:injection}) and continuous, we have $\iota(L) \in \overline{\iota(E_\tau)}$.  By Proposition \ref{prop:ftau}, we have $\iota(E_\tau) \subset \oX_{f_\tau}$, and $\iota(L) \in \oX_{f_{\tau'}}$ for some $\tau'$.  By Theorem \ref{thm:TNNmain} we have $f_{\tau'} \geq f_\tau$.  By Corollary \ref{cor:order}, we have $\tau' \leq \tau$.
\end{proof}

Recall that $P_n$ is defined using uncrossings of lensless medial graphs that result in a lensless medial graph. As a consequence of Proposition \ref{prop:Pnorder}, we can show that more general uncrossings of medial graphs also lead to relations in $P_n$.  This result was first established by Alman, Lian, and Tran \cite{ALT}.  

\begin{cor}
Let $G$ be a lensless medial graph.  Let $G'$ be obtained from $G$ by uncrossing any number of crossings in $G$ in an arbitrary manner, and then removing lenses using loop removals and lens removals.  Then $\tau(G') \leq \tau(G)$ in $P_n$.
\end{cor}
\begin{proof}
Uncrossing crossings in $G$ corresponds to sending conductances of edges in $\Gamma$ to $0$ or to $\infty$.  Removing loops and lenses corresponds to using series/parallel reductions, loop removals, and pendant removals.  Using Proposition \ref{prop:cactusparam} we see that $E_{\tau(G')} \subset \overline{E_{\tau(G)}}$.  By Proposition \ref{prop:Pnorder} we have $\tau(G') \leq \tau(G)$.
\end{proof}

Once we prove that $\iota(E_\tau) = (\oX_{f_\tau})_{\geq 0}$, we shall immediately obtain that $\overline{\oX_{f_\tau}} = \bigsqcup_{\tau' \leq \tau} \oX_{f_{\tau'}}$ as well.

\subsection{Proof of realizability part of Theorem \ref{thm:realizability}}\label{sec:realizability}
Suppose $X \in \X_{\geq 0}$.  Suppose that $X \in \oPi_f$, where $f = f_\tau$.  We need to show that $X$ is realizable by a cactus network.  We proceed by induction first on $n$, and then on the codimension of $\oPi_f$.

First suppose that $f$ has a fixed point $f(i) =i$.  Again for simplicity we assume $i = 2k-1$ is odd.  Consider the point $Y \in (\X_{n-1})_{\geq 0}$ from Proposition \ref{prop:fixedpoint}.  By the inductive hypothesis, $Y$ is represented by a cactus network $\Gamma'$ on $\bar 1, \ldots, \overline{k-1}, \overline{k+1},\ldots, \bar n$.  Define $\Gamma$ to be the cactus network obtained from $\Gamma'$ by adding an isolated new vertex labeled $\bar k$ (between $\overline{k-1}$ and $\overline{k+1}$) to $\Gamma'$.  Note that if $\overline{k-1}$ and $\overline{k+1}$ are glued together in $\Gamma'$, then to obtain $\Gamma$ we have to make a new cactus disk and place $\bar k$ on the boundary of that disk away from $\overline{k-1}$ and $\overline{k+1}$.  It is straightforward to check that $\Gamma$ represents $X$.

Now suppose that $f$ has no fixed points.  Then we must be able to find $i$ such that $i < f(i) < f(i+1)$.  By Proposition \ref{prop:elecreduce}, we have a point $X' = u_i(-a) \cdot X$ which lies in $\oX_{s_{i-1}f s_i} \cap \Gr(n-1,2n)_{\geq 0}$, where $s_{i-1} f s_i \in P_n$ is greater than $f$.  By the inductive hypothesis, $X'$ is representable by a cactus network $\Gamma'$.  By Proposition \ref{prop:elecaction}, $\Gamma =v_i(a) \cdot \Gamma'$ represents $X$.  This completes the proof of the realizability statement in Theorem \ref{thm:realizability}.

\subsection{Electroids and partition necklaces}\label{sec:partnecklace}

By Proposition \ref{prop:Lequiv}, the electroid of a critical cactus network $\Gamma$ depends only on the medial pairing $\tau(\Gamma)$.  Recall that $\E(\tau)$ denotes the electroid of any cactus network $\Gamma \in E_\tau$.  

\begin{theorem}\label{thm:electroid}
Let $\eta \in P_n$.  Then 
$$
\E(\eta) = \{\sigma \mid \tau(\sigma) \leq \eta\}.
$$
\end{theorem}
\begin{proof}
Let $\Gamma$ be a cactus network representing $\L \in E_\eta$.  Suppose $L_\sigma(\Gamma) \neq 0$.  Let $F$ be a grove in $\Gamma$ with boundary partition $\sigma$.  Removing the edges in in $\Gamma \setminus F$, and contracting the edges in $F$ we obtain a cactus network $\Gamma'$ with no edges.  It is clear that $\L(\Gamma') = p_\sigma$.  It follows from Proposition \ref{prop:Pnorder} that $\tau(\sigma) \leq \eta$.  Thus $\E(\L) \subseteq \{\sigma \mid \tau(\sigma) \leq \eta\}$.

Conversely, suppose $\tau(\sigma) \leq \eta$.  Then by definition of the partial order $P_n$ there is a way to contract and delete edges in $\Gamma$ to get a cactus network $\Gamma'$ such that $\L(\Gamma') \in E_{\tau(\sigma)} = p_\sigma$.  We can always find a grove $F$ with boundary partition $\sigma$ among the contracted edges, and it follows that $\sigma \in \E(\L)$.
\end{proof}

For example, if $\eta \in P_n$ is minimal, and $\tau(\sigma) = \eta$, then $E_\eta = p_\sigma$ and the theorem says $\E(p_\sigma) = \{\sigma\}$.

We now define an analogue of Grassmann necklaces where subsets are replaced by non-crossing partitions.   Let $\sigma \in \NC_n$ with parts $\sigma_1,\sigma_2,\ldots,\sigma_r$.  Let $\sigma_i$ and $\sigma_j$ be two parts.  We say that a third part $\sigma_k$ \defn{separates} $\sigma_i$ from $\sigma_j$ if any straight line in the disk from the convex hull of $\sigma_i$ to the convex hull of $\sigma_j$ intersects the convex hull of $\sigma_k$.  This condition can also be formulated as follows: let $(\bar a,\bar b,\bar c,\bar d) \subset [\bar n]$ be in circular order such that $\sigma_i \subset [\bar d,\bar a]$ and $\sigma_j \subset [\bar b,\bar c]$ with we have $\bar a,\bar d \in \sigma_i$ and $\bar b,\bar c \in \sigma_j$.  (Note that we may have $\bar a=\bar d$ or $\bar b=\bar c$.)  We define the circular arcs $(\sigma_i,\sigma_j) := (\bar a, \bar b)$ and $(\sigma_j,\sigma_i):= (\bar c, \bar d)$.  We say that a third part $\sigma_k$ of $\sigma$ separates $\sigma_i$ from $\sigma_j$ if $\sigma_k \cap (\sigma_i,\sigma_j) \neq \emptyset$ and $\sigma_k \cap (\sigma_j,\sigma_i) \neq \emptyset$.   In a similar manner, we define what it means for $\sigma_k$ to separate $\sigma_i$ and $\tsigma_j$, where $\tsigma_j$ is a part of the dual non-crossing partition $\tsigma$.  We also use the notations $(\sigma_i,\tsigma_j)$ and $(\tsigma_j,\tsigma_i)$ for the corresponding circular arcs. 

Denote the possibly empty set of parts of $\sigma$ that separate $\sigma_i$ from $\sigma_j$ (resp. $\tsigma_j$) by $\Sep_\sigma(\sigma_i,\sigma_j)$ (resp. $\Sep_\sigma(\sigma_i,\tsigma_j)$).  Note that $\Sep_\sigma(\sigma_i,\sigma_j)$ (resp. $\Sep_\sigma(\sigma_i,\tsigma_j)$) is linearly ordered: there is a part that is closest to $\sigma_i$, and one that is second closest, and so on.  Note that we will still consider $\Sep_\sigma(\sigma_i,\tsigma_j)$ as a collection of parts of $\sigma$ (there are analogous separating sets $\Sep_{\tsigma}(\tsigma_j,\sigma_i)$).

A pair $(\bar a, \bar b)$ is a \defn{legal transition} of $\sigma$ if either 
\begin{enumerate}
\item
$\bar a, \bar b$ belong to the same part $\sigma_i$, and $\bar b = \max_{\leq_{\bar a}} \sigma_i$,or
\item
 $\bar a \in \sigma_i$ and $\bar b \in \sigma_j$ belong to different parts of $\sigma$ and we have:
\begin{enumerate}
\item $|\sigma_i| > 1$ and $\sigma_j$ is contained in the circular interval $(\bar a, \max_{\leq_{\bar a}}\sigma_i)$
\item $\bar b = \max_{\leq_{\bar a}} \sigma_j$.
\end{enumerate}
\end{enumerate}
If $(\bar a, \bar b)$ is legal, we define a new non-crossing partition $t_{\bar a \bar b}(\sigma) := \sigma' = \in \NC_n$, as follows.  If $\bar a$ and $\bar b$ belong to the same part, then $\sigma' = \sigma$.  Otherwise, let $\sigma_{k_1}, \sigma_{k_2}, \ldots, \sigma_{k_c} \in \Sep_\sigma(\sigma_i, \sigma_j)$ be listed in order, with $\sigma_{k_1}$ closest to $\sigma_i$.  For each $\ell$, let
$$
A_\ell = \sigma_{k_\ell} \cap (\sigma_i, \sigma_j) \qquad \text{and} \qquad B_\ell = \sigma_{k_\ell} \cap (\sigma_j, \sigma_i)
$$
so that $\sigma_{k_\ell} = A_\ell \sqcup B_\ell$ and both $A_\ell$ and $B_\ell$ are non-empty.  Also let $A_0 = \sigma_i \cap [\bar a,\bar b)$, $B_0 = \sigma_i \cap (\bar b,\bar a)$.  Define $\sigma'$ by replacing the parts $\sigma_i, \sigma_{k_1}, \sigma_{k_2}, \ldots, \sigma_{k_c}, \sigma_j$ by the parts
$$
A_0, A_1 \cup B_0, A_2 \cup B_1,\ldots, A_c \cup B_{c-1}, \sigma_j \cup B_c. 
$$

Similarly, define a pair $(\bar a, \tilde b)$ to be a \defn{legal transition} of $\sigma$ if the part $\tsigma_j \subset [\tilde n]$ containing $\tilde b$ satisfies:
\begin{enumerate}
\item
$\tsigma_j$ is contained in the circular interval $(\bar a, \max_{\leq_{\bar a}}\sigma_i)$
\item
$\bar b = \max_{\leq_{\bar a}} \tsigma_j$.
\end{enumerate}
For a legal transition $(\bar a, \tilde b)$, we define $t_{\bar a \tilde b}(\sigma) := \sigma' = \in \NC_n$, as follows.  Let $\sigma_{k_1}, \sigma_{k_2}, \ldots, \sigma_{k_c} \in \Sep_\sigma(\sigma_i, \tsigma_j)$ be listed in order, with $\sigma_{k_1}$ closest to $\sigma_i$.  For each $\ell$, let
$$
A_\ell = \sigma_{k_\ell} \cap (\sigma_i, \tsigma_j) \qquad \text{and} \qquad B_\ell = \sigma_{k_\ell} \cap (\tsigma_j, \sigma_i)
$$
so that $\sigma_{k_\ell} = A_\ell \sqcup B_\ell$ and both $A_\ell$ and $B_\ell$ are non-empty.  Also let $A_0 = \sigma_i \cap [\bar a,\tilde b)$, $B_0 = \sigma_i \cap (\tilde b,\bar a)$.  Define $\sigma'$ by replacing the parts $\sigma_i, \sigma_{k_1}, \sigma_{k_2}, \ldots, \sigma_{k_c}$ by the parts
$$
A_0, A_1 \cup B_0, A_2 \cup B_1,\ldots, A_c \cup B_{c-1},  B_c. 
$$

Informally, $t_{\bar a \bar b}(\sigma)$ (and similarly and $t_{\bar a \tilde b}(\sigma)$) is obtained by drawing a line from $\bar a$ to $\bar b$, cutting up the parts $\Sep_\sigma$ using this line, and then reattaching the parts by a shift.  

\begin{lemma}\label{lem:sswap} \ 
\begin{enumerate}
\item
For a legal transition $(\bar a, \bar b)$ ,the set partition $\sigma' = t_{\bar a \bar b}(\sigma)$ is non-crossing satisfying $|\sigma'|=|\sigma|$.  Furthermore, if $I_{\bar a}(\sigma) = I$ then $I_{\tilde a}(\sigma') = I - \{2a-1\} \cup \{2b-1\}$.
\item
For a legal transition $(\bar a, \tilde b)$ ,the set partition $\sigma' = t_{\bar a \tilde b}(\sigma)$ is non-crossing satisfying $|\sigma'|=|\sigma|+1$.  Furthermore, if $I_{\bar a}(\sigma) = I$ then $I_{\tilde a}(\sigma') = I - \{2a-1\} \cup \{2b\}$.
\end{enumerate}
\end{lemma}
\begin{proof}
We prove (1); the other claim is similar.  The claim concerning the number of parts follows immediately from the definitions.  If $\bar a$ and $\bar b$ belong to the same part all the claims is trivial, so we assume otherwise.  The sets $A_0, A_1, \ldots, A_c, \sigma_j, B_c, B_{c-1},\ldots, B_1, B_0$ are in circular order.  It follows from this that $\sigma'$ is non-crossing.  For the second statement, we use \eqref{eq:Ia}.  Set $I' = I - \{2a-1\} \cup \{2b-1\}$.  We have that 
$$
\max_{\leq_{\tilde a}} A_0 = \bar a, \qquad \max_{\leq_{\tilde a}} A_1 \cup B_0 = \max_{\leq_{\bar a}} \sigma_i, \cdots, \max_{\leq_{\tilde a}} A_\ell \cup B_{\ell-1} = \max_{\leq_{\bar a}} \sigma_{k_\ell}, \cdots
$$
$$
\max_{\leq_{\tilde a}} A_c \cup B_{c-1} = \max_{\leq_{\bar a}} \sigma_{k_{c-1}}, \qquad \max_{\leq_{\tilde a}} \sigma_j \cup B_c = \max_{\leq_{\bar a}} \sigma_{k_c}.
$$
Thus $([2n]\setminus I') \cap [\bar n]$ agrees via \eqref{eq:Ia} with $([2n]\setminus I_{\tilde a}(\sigma')) \cap [\bar n]$.  A similar computation for the dual partition shows that  $([2n]\setminus I') \cap [\tilde n]$ equals $([2n]\setminus I_{\tilde a}(\sigma')) \cap [\tilde n]$.
\end{proof}

\begin{example}
Let $\sigma = (\bar 1 \bar 8| \bar 2 \bar 6 \bar 7| \bar 3 \bar 5| \bar 4 | \bar 9)$.  Then $\Sep_\sigma(\bar1 \bar 8, \bar 4) = \{ \bar 2 \bar 6 \bar7,\bar 3 \bar 5\}$.  The pair $(\bar 1,\bar 4)$ is a legal transition.   We have
$$
A_0 = \bar 1, A_1 = \bar 2, A_2 = \bar 3, B_0 = \bar 8, B_1 = \bar 6 \bar 7, B_2 = \bar 5
$$
so that $\sigma' = t_{\bar 1\bar 4}(\sigma)= (\bar 1|\bar 2 \bar 8| \bar 3 \bar 6 \bar 7|\bar 4 \bar 5|\bar 9)$.

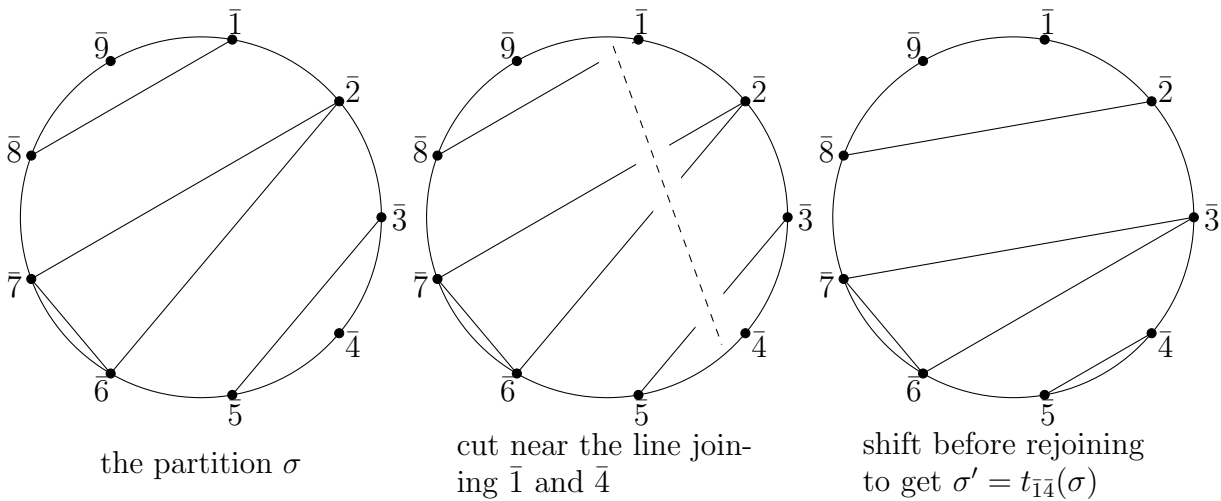
\begin{center}
\begin{tikzpicture}[scale = 0.6]
\draw (0,0) circle (4cm);
\foreach \i in {1,...,9}
{
\filldraw[black] (120-40*\i:4) circle (0.1cm);
\node at (120-40*\i:4.4) {$\bar \i$};
\coordinate (a\i) at (120-40*\i:4);
}
\draw (a1) -- (a8);
\draw (a2) -- (a6) -- (a7) -- (a2);
\draw (a3) --(a5);
\node at (0,-5.5) {the partition $\sigma$};
\begin{scope}[shift={(9,0)}]
\draw (0,0) circle (4cm);
\foreach \i in {1,...,9}
{
\filldraw[black] (120-40*\i:4) circle (0.1cm);
\node at (120-40*\i:4.4) {$\bar \i$};
\coordinate (a\i) at (120-40*\i:4);
}
\draw (a1) -- (a8);
\draw (a2) -- (a6) -- (a7) -- (a2);
\draw (a3) --(a5);
\draw[line width=0.5cm,white] (88:3.8) -- (120-40*4-8:3.8);
\draw[dashed] (88:3.8) -- (120-40*4-8:3.8);
\node at (0,-5.5) [text width=4cm] {cut near the line joining $\bar 1$ and $\bar 4$};
\end{scope}
\begin{scope}[shift={(18,0)}]
\draw (0,0) circle (4cm);
\foreach \i in {1,...,9}
{
\filldraw[black] (120-40*\i:4) circle (0.1cm);
\node at (120-40*\i:4.4) {$\bar \i$};
\coordinate (a\i) at (120-40*\i:4);
}
\draw (a2) -- (a8);
\draw (a3) -- (a6) -- (a7) -- (a3);
\draw (a4) --(a5);
\node at (0,-5.5)  [text width=4cm] {shift before rejoining to get $\sigma' = t_{\bar 1\bar 4}(\sigma)$};
\end{scope}
\end{tikzpicture}

\end{center}

Now consider the legal transition $(\bar 1, \tilde 4)$.  Then as before we have $\Sep_\sigma(\bar1 \bar 8, \tilde 4) = \{ \bar 2 \bar 6 \bar7,\bar 3 \bar 5\}$, and the same $A$- and $B$-sets.   But this time
that $\sigma' = t_{\bar 1 \tilde 4}(\sigma) = (\bar 1|\bar 2 \bar 8| \bar 3 \bar 6 \bar 7|\bar 4 |\bar 5|\bar 9)$.
\end{example}

For $s \in [2n]$, we will define a relation $\sigma \rightarrow_s \sigma'$ called an \defn{$s$-swap}, as follows.  First suppose $s = 2a-1$ is odd.  
\begin{enumerate}
\item
If $\bar a$ is a singleton in $\sigma$, then we must have $\sigma' = \sigma$.
\item
If not, then we must have $\sigma' = t_{\bar a \bar b}(\sigma)$ or $\sigma' = t_{\bar a \tilde b}(\sigma)$, where $(\bar a, \bar b)$ or $(\bar a, \tilde b)$ is a legal transition.
\end{enumerate}
Otherwise, $s = 2a$ is even.  We then ask for the same condition between $\tsigma$ and $\tsigma'$, with $\tilde a$ replacing $\bar a$ (and using $\Sep_{\tsigma}$ in the definitions instead).

A \defn{partition necklace} on $[\bar n]$ is a sequence $\Sigma =(\sigma_1,\sigma_2,\ldots,\sigma_{2n})$ of non-crossing partitions on $[\bar n]$ such that each $\sigma_s \rightarrow_s \sigma_{s+1}$ is a $s$-swap for each $s$.  

Let $\I = (I_1,I_2,\ldots,I_{2n})$ be a Catalan necklace.  Then we may define a necklace $\Sigma(\I) =(\sigma^{(1)},\sigma^{(2)},\ldots,\sigma^{(2n)})$ of non-crossing partitions by the condition that $\sigma^{(s)}$ satisfies $I_s(\sigma^{(s)}) = I_s$.

\begin{proposition}\label{prop:Catalan}
The map $\I \longmapsto \Sigma(\I)$ is a bijection between Catalan necklaces and partition necklaces.
\end{proposition}
\begin{proof}
The inverse map is given by sending $\Sigma = (\sigma_1,\ldots,\sigma_{2n})$ to the Grassmann necklace $\I(\Sigma) = (I_1(\sigma^{(1)}),I_2(\sigma^{(2)}),\ldots,I_{2n}(\sigma^{(2n)}))$.  By Lemma \ref{lem:sswap}, $\I(\Sigma)$ is a Catalan necklace if $\Sigma$ is a partition necklace.

It suffices to show that if $\I = (I_1,\ldots,I_{2n})$ is a Catalan necklace then $\Sigma(\I) = (\sigma^{(1)},\ldots,\sigma^{(2n)})$ is a partition necklace.  Let us consider the $s$-th subset $I_s$, and set $\sigma = \sigma^{(s)}$ and $\sigma' = \sigma^{(s+1)}$.  If $s \notin I_s$, then $I_{s+1} = I_s$ and in this case we have $\sigma = \sigma'$, and $\sigma \to_s \sigma'$.  Otherwise, let us assume for simplicity that $s = 2a-1$ is odd.  

Suppose first that $I_{s+1} = I_s - \{s\} \cup \{s'\}$ where $s' = 2a'-1$.  Obviously $\bar a' \notin I_s$ so by \eqref{eq:Ia} we have that $\bar a'$ is maximal (with respect to $\leq_{\bar a}$) in its part of $\sigma$.  If $\bar a'$ and $\bar a$ are in the same part, then $\sigma' = \sigma$, and $\sigma \to_{s} \sigma'$ is indeed a $s$-swap.  If $\bar a' \in \sigma_j$ while $\bar a \in \sigma_i$ are in different parts, and $\sigma_j$ is contained in the circular interval $(\bar a, \max_{\leq_{\bar a}}\sigma_i)$, then $(\bar a, \bar a')$ is a legal transition, so again $\sigma \to_s \sigma'$.  Let $\bar b = \max_{\leq_{\bar a}}\sigma_i$.  Finally, we show that it is impossible for $\bar a'$ to be contained in the circular interval $(\bar b, \bar a)$.  We claim that if this is the case that $I_{s+1}$ is not a $(s+1)$-shifted Catalan subset.  A counting argument similar to Lemma \ref{lem:dualpart} gives that $|I_s \cap [2a-1,2b-1]| =(b-a)$.  If $\bar a' \in (\bar b, \bar a)$, then $|I_{s+1} \cap [2a,2b-1]| = (b-a-1)$ but $([2n]\setminus I_{s+1}) \cap [2a,2b-1] = 2(b-a) = |I_{s+1} \cap [2a,2b-1]| + 2$.  This would mean that $I_{s+1}$ is not a $(s+1)$-shifted Catalan subset.

Next, suppose that $I_{s+1} = I_s - \{s\} \cup \{s'\}$ where $s' = 2a'$.  Then $\tilde a'$ belongs to a part $\tsigma_j$ of $\tsigma$.  By \eqref{eq:Ia} we have that $\tilde a'$ is maximal (with respect to $\leq_{\bar a}$) in its part of $\tsigma$.  The same argument as above shows that $\tsigma_j \subset (\bar a, \bar b)$ where $\bar b = \max_{\leq_{\bar a}}\sigma_i$.  So $(\bar a, \tilde a')$ is a legal transition, and the claims follow.
\end{proof}

The partition necklace $\Sigma(\tau) = (\sigma^{(1)},\sigma^{(2)},\ldots,\sigma^{(2n)})$ of $\tau \in P_n$ is defined as follows.   For each $s$, $\sigma^{(s)}$ is chosen so that $I_s(\sigma_s(\tau))$ is the $s$-shifted lexicographically minimal subset in $\{I_s(\sigma) \mid \sigma \in \E(\tau))$.  By definition, $\sigma_s(\tau) \in \E(\tau)$ for each $\tau$.  The following characterization of electroids is analogous to a theorem of Oh \cite{Oh}.

\begin{theorem}\label{thm:Ohelectroid}
For each $s$, we have $I_s(\sigma_s(\tau)) = I_s(\tau)$.  Thus $\Sigma(\tau) = \Sigma(\I(\tau))$.  The electroid of $\tau$ is given by 
$$
\E(\tau) = \{\sigma \mid \sigma \geq_s \sigma_s(\tau) \text{ for all } s\}
$$
where $\geq_s$ is the $s$-shifted dominance order on $\NC_n$ from Section \ref{sec:ncorder}.
\end{theorem}
\begin{proof}
Fix $s$.  It is enough to show that if $I_s(\sigma) = I_s(\tau)$ then $\sigma \in \E(\tau)$.  By Theorem \ref{thm:electroid} and Corollary \ref{cor:order}, it is enough to show that $ I_t(\tau) \leq_t I_t(\sigma)$ for all $t \in [2n]$.  

For a circular arc $[t,r]$, let $\tau[t,r]$ be the number of strands of $\tau$ with both endpoints on the arc.  Let $\tau' = \tau(\sigma)$.  Then the stated inequalities are equivalent to: $\tau[t,r] \leq \tau'[t,r]$ for every circular arc $[t,r]$.  Note that by considering the complement circular interval, we can assume that $[t,r]$ does not contain $s$.  

We think of $\tau'$ as a sequence of U-s and D-s (with the sequence starting at $s$).  Then $\tau'[t,r]$ is a consecutive subsequence of this, which we call $S$.  By matching U-s and D-s like parentheses, we can count the number of matched pairs in any sequence of U-s and D-s (not necessarily a sequence corresponding to a Dyck path).  Each strand in $\tau$ with both endpoints in $[t,r]$ is a pair consisting of a U followed (not immediately) by a D in $S$.  So the claim follows from the following statement: if we add extra letters to a sequence of U-s and D-s, the number of matched pairs cannot decrease.  This is easy to see directly.
\end{proof}

Let $\Sigma =(\sigma^{(1)},\ldots,\sigma^{(2n)})$ and $\Theta = (\theta^{(1)},\ldots, \theta^{(2n)})$ be two partition necklaces.  Then we define $\Sigma \leq \Theta$ if $\sigma^{(s)} \leq_s \theta^{(s)}$ for each $s \in [2n]$.

\begin{corollary}
Let $\tau, \tau' \in P_n$.  The following are equivalent:
\begin{enumerate}
\item
$\tau' \leq \tau$
\item
$\E(\tau') \subseteq \E(\tau)$
\item
$\Sigma(\tau')  \geq \Sigma(\tau)$.
\end{enumerate}
\end{corollary}
\begin{proof}
The equivalence of (1) and (3) follows from Corollary \ref{cor:order} and the first statement of Theorem \ref{thm:Ohelectroid}.  The equivalence of (2) and (3) follows from the second statement of Theorem \ref{thm:Ohelectroid}.
\end{proof}

Not all partition necklaces are of the form $\Sigma(\tau)$.  The partition necklaces that arise in this manner are satisfy the following additional conditions:
\begin{enumerate}
\item
if $\bar a$ is a singleton in  $\sigma_{\bar a}$ then $\sigma^{(\overline{a+1})} = t_{\tilde a \widetilde{a-1}}\sigma^{(\tilde a)}$, and
\item
if $\sigma^{(\bar a)} \to_{\bar a} \sigma^{(\tilde a)}$ is the swap given by the transition $(\bar a, \bar a')$, then $\sigma^{(\tilde a')} \to_{\tilde a'} \sigma^{\overline{a'+1}}$ is given by the transition $(\tilde a', \widetilde{a-1})$, and
\item
if $\sigma^{(\bar a)} \to_{\bar a} \sigma^{(\tilde a)}$ is the swap given by the transition $(\bar a, \tilde a')$, then $\sigma^{(\overline{a'+1})} \to_{\overline{a'+1}} \sigma^{\widetilde{a'+1}}$ is given by the transition $(\overline{a'+1}, \widetilde{a-1})$, and
\item
similar conditions with $\sigma^{(\bar a)}$ replaced by $\sigma^{(\tilde a)}$.
\end{enumerate}

\begin{example}
Let $\tau$ be the matching $\{(1,4),(2,6),(3,7),(5,8)\}$.  Then $f_\tau = [3,5,6,8,7,9,10,12]$, and $\I(\tau)=(124,234,435,456,568,678,781,812)$ and 
$$
\sigma(\tau) = \left((14|2|3),(24|1|3),(12|34),(13|2|4),(23|1|4),(24|1|3),(12|34),(13|2|4)\right).
$$
We can check that each consecutive pair is given by a $s$-swap.  For example $(12|34) \to_{\bar 2} (13|2|4)$ is given by the legal transition $(\bar 2, \tilde 3)$. In this case the separating set consists of only $(\bar 3 \bar 4)$.
\end{example}

\subsection{Quadratic relations for grove measurements}\label{sec:quadratic}
The well known Pl\"ucker relations \cite{Ful} for the Grassmannian gives rise to certain quadratic relations for the grove measurements $L_\sigma$.  

\begin{proposition}\label{prop:quadratic}
Suppose $\L(\Gamma) \in \P^{\NC_n}$ is the grove measurement point of a cactus network $\Gamma$.  Then for each $1 \leq k < n-1$ and each $I = \{i_1 < i_2 < \cdots < i_{n-1}\}$, $J = \{j_1 < j_2 < \cdots < j_{n-1}\}$, we have
$$
 \sum_{\sigma \in \E(I), \kappa \in \E(J)} L_\sigma L_\kappa = \sum_{I',J'} (-1)^{a(I',J')}  \sum_{\sigma \in \E(I'), \kappa \in \E(J')} L_\sigma L_\kappa
$$
where:
\begin{enumerate}
\item
 the summation is over $I', J'$ obtained from swapping the first $k$ indices in $J$ with with any $k$ indices in $I$, keeping the order in both;
\item
$a(I',J')$ is the total number of swaps needed to put both subsets $I'$ and $J'$ in order.
\end{enumerate}
\end{proposition}

Note that $\E(I)$ is taken to be empty set if $I$ has repeated elements.

\begin{proposition}
Let $p \in \P^{\NC_n}$ be an arbitrary point.  Then $\L \in E_n$ if and only if
\begin{enumerate}
\item
$L_\sigma(p) \geq 0$ for each $\sigma \in \NC_n$, and
\item
$L_\sigma(p)$ satisfies the relations of Proposition \ref{prop:quadratic} (only the relations with $k=1$ is sufficient).
\end{enumerate}
\end{proposition}
\begin{proof}
The ``only if" direction is clear.  Suppose $p \in \P^{\NC_n}$ satisfies both conditions.  Use the relations in Theorem \ref{thm:concordant} to produce a point $X_p \in \P^{\binom{2n}{n-1}-1}$ in Pl\"ucker projective space.  A point in Pl\"ucker space lies in the Grassmannian if and only if the Pl\"ucker relations are satisfied with one index swapped (that is $k = 1$ in Proposition \ref{prop:quadratic}).  Thus Condition (2) gives $X_p \in \X$.  Condition (1) gives $X_p \in \X_{\geq 0}$.  But by the argument in Proposition \ref{prop:injection}, $\Delta_I(X_p)$ determine $L_\sigma(p)$, and by the realizability statement in Theorem \ref{thm:realizability} we must have $p \in E_n$.  
\end{proof}


\begin{thebibliography}{XXX}
\bibitem[ALT]{ALT} {\sc J.~Alman, C.~Lian, and B.~Tran:}
Circular Planar Electrical Networks: Posets and Positivity.
J. Combin. Theory Ser. A 132 (2015), 58--101.
\bibitem[BB]{BB} {\sc A. Bj\"orner and F. Brenti:}
  Combinatorics of Coxeter Groups, Graduate Texts in Mathematics, 231.
Springer-Verlag, New York, 2005.
\bibitem[CIM]{CIM} {\sc E.B.~Curtis, D.~Ingerman and J.A.~Morrow:}
Circular planar graphs and resistor networks, Linear Algebra Appl., 283 (1998), no. 1-3, 115--150.
\bibitem[CGV]{CGV} {\sc Y.~Colin de Verdi\`ere, I.~Gitler, and D.~Vertigan:}
R\'{e}seaux \'{e}lectriques planaires. II, Comment. Math. Helv., 71(1) (1996), 144--167.
\bibitem[Ful]{Ful} {\sc W. Fulton:}
Young tableaux. With applications to representation theory and geometry. London Mathematical Society Student Texts, 35. Cambridge University Press, Cambridge, 1997. x+260 pp.
\bibitem[GK]{GK} {\sc A. B. Goncharov and R. Kenyon:}
Dimers and cluster integrable systems, Annales scientifiques de l'ENS, 46 (2013), 747--813.
\bibitem[HS]{HS} {\sc A.~Henriques and D.E.~Speyer:}
The multidimensional cube recurrence.
Adv. Math. 223 (2010), no. 3, 1107--1136. 
\bibitem[HWX]{HWX}
{\sc Y.-t. Huang, C. Wen, and D, Xie:} The Positive orthogonal Grassmannian and loop amplitudes of ABJM. Journal of Physics A: Mathematical and Theoretical. 2014;47 :474008.
\bibitem[Kenn]{Kenn} {\sc A.E.~Kennelly:}
Equivalence of triangles and stars in conducting networks, Electrical World and Engineer, 34 (1899),
413--414.
\bibitem[Ken]{Ken} {\sc R.~Kenyon:} The Laplacian on planar graphs and graphs on surfaces. Current developments in mathematics, 2011, 1--55, Int. Press, Somerville, MA, 2012.
\bibitem[KPW]{KPW} {\sc R.~Kenyon, J.~Propp, and D.~Wilson:}
Trees and Matchings, Elec. J. Comb. 7 (2000), Research Paper 25.
\bibitem[KW]{KW} {\sc R.~Kenyon and D.~Wilson:}
Boundary partitions in trees and dimers, Trans. Amer. Math. Soc. 363 (2011), no. 3, 1325--1364.
\bibitem[KLS]{KLS} {\sc A.~Knutson, T.~Lam, and D.~Speyer:} Positroid Varieties: Juggling and Geometry,  Compositio Mathematica 149, 1710--1752.
\bibitem[Kuo]{Kuo} {\sc E.~Kuo:} Applications of graphical condensation for enumerating matchings and tilings. Theoret. Comput. Sci. 319 (2004), no. 1-3, 29--57. 
\bibitem[Lam14]{Lamnotes} {\sc T.~Lam:} Totally nonnegative Grassmannian and Grassmann polytopes.  Current Developments in Mathematics, 2014, 51--152.
\bibitem[Lam15]{Lammatch} {\sc T.~Lam:} The uncrossing partial order on matchings is Eulerian. J. Combin. Theory Ser. A. 135 (2015), 105--111.
\bibitem[LP]{LP}
{\sc T. Lam and P.~Pylyavskyy:} Electrical networks and Lie theory, Algebra and Number Theory 9 (2015), 1401--1418.
\bibitem[Oh]{Oh} {\sc S. Oh:} Positroids and Schubert matroids, J. Combin. Theory Ser. A 118 (2011), 2426--2435.
\bibitem[Pos]{Pos} {\sc A. Postnikov:}
Total positivity, Grassmannians, and networks, preprint. \\
  \url{http://www-math.mit.edu/~apost/papers/tpgrass.pdf}
\bibitem[PSW]{PSW} 
{\sc A.~Postnikov, D.~Speyer, and L. Williams:} Matching polytopes, toric geometry, and the non-negative part of the Grassmannian. J. Algebraic Combin. 30 (2009), no. 2, 173--191.
\bibitem[Tal]{Tal} {\sc K. Talaska:} A formula for Pl\"ucker coordinates associated with a planar network. Int. Math. Res. Not. IMRN 2008, Art. ID rnn 081, 19 pp.
\end{thebibliography}
\end{document}